\documentclass[a4paper,11pt]{article}

 \usepackage{multicol}
\usepackage{amsmath,amssymb,graphics,epsfig,color,cite}
\usepackage{mathrsfs,bbm}
\usepackage{amsthm} %beginproof
\usepackage{bm}
\usepackage[toc,page]{appendix}
\usepackage{tikz}
\usetikzlibrary{arrows}
\usetikzlibrary{decorations.pathreplacing,decorations.markings}
\tikzset{arrow data/.style 2 args={%
      decoration={%
         markings,
         mark=at position #1 with \arrow{#2}},
         postaction=decorate}
      }%

 \usepackage{enumitem}

\usepackage{hyperref}
 
\usepackage{geometry}
\definecolor{light}{gray}{.10}

\DeclareMathAlphabet{\mathpzc}{OT1}{pzc}{m}{it}

\newcounter{savecntr}% Save footnote counter
\newcommand{\mbf}[1]{\mathbf{#1}}

\newcommand{\ve}{\varepsilon}

\DeclareMathOperator*{\argmin}{arg\,min}

\def\Hess{{\rm Hess\,}}% Hessian

\def\and {{\rm \; and \;}}

\newcommand{\ft}[1]{\mathsf{#1}} % icii
\newcommand {\pa}{\partial}

\newtheorem{theorem}{Theorem}
\newtheorem{proposition}{Proposition}%[section]
\newtheorem{definition}[proposition]{Definition}

\newtheorem{corollary}[proposition]{Corollary}
\newtheorem{remark}[proposition]{Remark}

\usepackage[hang,font={small,it},labelfont=sf]{caption}

\title{ \textbf{\Large{The exit from a metastable state: concentration of  the exit point distribution on the low energy saddle points, part 2}}}

\author{Tony
  Leli\`evre\setcounter{savecntr}{\value{footnote}}\thanks{CERMICS,
    \'Ecole des Ponts, Universit\'e  Paris-Est, INRIA, 77455
    Champs-sur-Marne, France. E-mail:  tony.lelievre@enpc.fr}\,,
  Dorian Le
  Peutrec\setcounter{savecntr}{\value{footnote}}\thanks{Institut Denis Poisson, Universit\'e d'Orl\'eans, Universit\'e de Tours, CNRS, Orl\'eans, France. E-mail: dorian.le-peutrec@univ-orleans.fr}$\, \ $and Boris Nectoux
  \thanks{LMBP, Universit\'e Clermont Auvergne, Aubi\`ere, France E-mail: boris.nectoux@uca.fr}}
 \date{}
\begin{document} 

 \maketitle

 \begin{abstract}
We consider     the first exit point distribution from a bounded domain~$\Omega$ of the stochastic process~$(X_t)_{t\ge 0}$ solution to the overdamped Langevin dynamics
$$d X_t = -\nabla f(X_t) d t + \sqrt{h} \ d B_t$$ 
starting from  deterministic initial conditions in~$\Omega$, under
rather general assumptions on $f$  (for instance, $f$ may have several
critical points in $\Omega$). This work is a continuation of the
previous paper~\cite{DLLN-saddle1} where the exit point distribution
from $\Omega$ is studied when~$X_0$ is initially distributed according
to the quasi-stationary distribution of $(X_t)_{t\ge 0}$ in~$\Omega$. The proofs
are based on analytical results on the dependency of the exit point
distribution on
the initial condition, large deviation techniques and results on the genericity of Morse functions.

% In the small temperature regime ($h\to 0$) and under rather general assumptions on $f$  (in particular, $f$ may have several critical points in $\Omega$), it is proven that the support of the distribution of the first exit point concentrates on some points realizing the minimum of~$f$ on $\pa \Omega$.  The proof relies on     tools to study  tunnelling effects in   semi-classical analysis. Extensions of the results to more general initial distributions than the quasi-stationary distribution are also presented.  

\end{abstract}

%\tableofcontents

\section{Introduction and main results}\label{sec:intro}
%Probability Theory and Related Fields
%analysis and pde
%Ann. Appl. Probab. 

The aim of this article is to extend the results
of~\cite{DLLN-saddle1} on the concentration of the first exit point
distribution from a domain to general initial conditions within the
domain. For the sake of consistency, we first recall in
Section~\ref{sec:exit} the motivation for such a study, some related
works in the literature, and an informal presentation of our
results. Section~\ref{nota-hypo} then gives precise statements of
the results we prove.

 \subsection{Motivation and informal presentation of the results}\label{sec:exit}
%\subsubsection{Overdamped Langevin dynamics and concentration of the law of  $X_{\tau_\Omega}$ when $h\to 0$ }   
  We are interested in the overdamped Langevin dynamics
\begin{equation}\label{eq.langevin}
d X_t = -\nabla f(X_t) d t + \sqrt{h} \ d B_t,
\end{equation}
where $X_t\in \mathbb R^d$ is a vector in $\mathbb R^d$,~$f: \mathbb R^d \to 
\mathbb R$ is a $C^\infty$ function, $h$ is a positive parameter and 
$(B_t)_{t\geq 0}$ is a standard $d$-dimensional Brownian motion. The
process~\eqref{eq.langevin} can be used to model the evolution of
molecular systems, for example. In this case,  $f$ is the potential function and $h>0$ is proportional to the temperature  of the system.  
Let us consider a domain~$\Omega \subset \mathbb R^d$ and the associated exit
event from $\Omega$. More precisely, let us introduce
\begin{equation}\label{eq.tau}
\tau_{\Omega}=\inf \{ t\geq 0 | X_t \notin \Omega      \}
\end{equation}
the first exit time from $\Omega$.  We are interested in the limit of
the law of $X_{\tau_{\Omega}}$ when $h \to 0$. Under some assumptions which will be made precise below, this law
concentrates on a subset of $\pa \Omega$, concentration being defined in the following sense. 
\begin{definition}\label{de.concentration}
Let $\mathcal Y\subset \pa \Omega$ and let us consider a family of
random variable $(Y_h)_{h >0}$ which admits a limit in distribution
when $h \to 0$.  The law of $Y_h$ concentrates on~$\mathcal Y$  in the limit $h\to 0$   if for every  neighborhood $\mathcal  V_{\mathcal Y}$ of ${\mathcal Y}$ in~$\pa \Omega$    $$\lim \limits_{h\to 0}\mathbb P  \left [ Y_h \in \mathcal  V_{\mathcal Y}\right]=1,$$ 
  and if for all $x\in \mathcal Y$ and for all neighborhood $\mathcal  V_x$ of $x$ in~$\pa \Omega$ 
  $$\lim \limits_{h\to 0}\mathbb P  \left [ Y_h \in \mathcal  V_x\right]>0.$$
  \end{definition} 
  \noindent
 In other words, the law of $Y_h$ concentrates on~$\mathcal Y$  if  $\mathcal Y$ is the support of the law of $Y_h$ in the limit $h\to 0$. 
 
 In this work, we investigate the concentration of the law of
  $X_{\tau_{\Omega}}$  on a subset of $\pa \Omega$ when $X_0=x\in
  \Omega$,  under general assumptions on the function $f:\overline
  \Omega\to \mathbb R$.   This is of practical interest in order to
  predict   where   the process~\eqref{eq.langevin} is more likely to
  leave $\Omega$ in the zero-noise limit. The study of the exit
  event in the small temperature regime has interesting theoretical
  and numerical counterparts, to relate continuous state space
  dynamics such as~\eqref{eq.langevin} to discrete state space
  dynamics (jump Markov model), and to accelerate the sampling of
  metastable trajectories, see~\cite{di-gesu-lelievre-le-peutrec-nectoux-17,Lelievrehandbook}.

 \medskip

%\subsubsection{Purpose of this work}  
%\label{sec.purpose}
\noindent
\textbf{Review of the literature.} Let us mention the main results
from the mathematical literature on the exit problem related to our
problem. We refer to~\cite{Day} for a more comprehensive review. 

The concentration of the  law of
$X_{\tau_\Omega}$ on  $\argmin_{\pa \Omega}f$ in the small temperature
regime ($h\to 0$) has been studied  in~\cite{MS77,schuss90, MaSc}
through formal computations, see also~\cite{MaSt}. Many of these results have been rigorously
proven either by studying the underlying partial differential
equations, or  by using large deviations techniques. In particular, when it holds
\begin{equation}\label{eq.prev-res1}
\partial_nf>0 \text{ on } \partial \Omega,
\end{equation}
where $\partial_nf$ is the normal derivative of $f$ on $\pa \Omega$ ($n$ is the unit outward normal vector to $D$), and 
\begin{equation}\label{eq.prev-res2}
\{x\in \Omega, \vert \nabla f(x)\vert =0\}=\{x_0\} \text{ with } f(x_0)=\min_{\overline \Omega} f \text{ and } \text{det Hess} f(x_0)>0,
 \end{equation}
  and $f$ attains its minimum on $\pa \Omega$ at one single point $y_0$, it is proved in~\cite[Theorem~2.1]{FrWe} that the law of~$X_{\tau_{\Omega}}$  concentrates on $y_0$ in the limit $h\to 0$,  when $X_0=x\in \Omega$.  This result has been extended in~\cite{kamin1979elliptic,Kam,Per,day1984a,day1987r}   when only~\eqref{eq.prev-res1} and~\eqref{eq.prev-res2} are satisfied:  it is proved there   that  the    the law of~$X_{\tau_{\Omega}}$  concentrates in this case on  $\argmin_{\pa \Omega}f$ in the limit
 $h\to 0$,  when $X_0=x\in \Omega$. 
 
In \cite[Theorem 5.1]{FrWe}, for $ \Sigma\subset \pa \Omega$,   the limit of $h\ln \mathbb P  \left [ X_{\tau_{\Omega}} \in \Sigma\right]$ when $h\to 0$ is related   to a minimization problem  involving  the quasipotential  of the process~\eqref{eq.langevin}. 
 Let us mention two limitations of~\cite[Theorem~5.1]{FrWe}. First,
 this theorem requires to be able to compute the quasipotential in
 order to get useful information: this is trivial under the
 assumptions~\eqref{eq.prev-res1} and~\eqref{eq.prev-res2} but more
 complicated under more general assumptions on $f$ (in particular when
 $f$ has several critical points in $\Omega$). Second, even when the
 quasipotential is analytically known, this result only gives the
 subset of $\partial \Omega$ where exit will not occur on an
 exponential scale in the limit $h\to 0$. It does not allow to exclude
 exit points with probability which goes to zero polynomially in $h$
 (this may indeed occur, see Remark~\ref{re.discussionHP} below and
 examples in~\cite[Section 1.4]{DLLN-saddle0}), and, when the exit point distribution
 concentrates on more than one point, it
 does not give the relative probabilities to exit through the various exit
 points in the limit $h\to 0$.
However, let us emphasize that the strength
of large deviation theory is that it is very general: $f$ may have
several local minima in $\Omega$, and the theory actually applies to
a much wider class of dynamics (in particular for non-gradient drift
and non-isotropic noise) and in a broader geometric setting~\cite{olivieri2005large,FrWe}. See for
example~\cite{BoRe,landim2017dirichlet,landim2019metastable} for   recent
contributions to the non reversible case. Other references where the
exit problem appears as an intermediate tool to study spectral
properties of the inifinitesimal generator are~\cite{BEGK,BGK,Mat,miclo-95}.

\medskip
 
\noindent
\textbf{Objective of this work. }%As explained above,  the concentration of the  law of~$X_{\tau_\Omega}$ on  $\argmin_{\pa \Omega}f$ was obtained when~\eqref{eq.prev-res1} and~\eqref{eq.prev-res2} hold (which imply in particular that~$f$ has only one critical point in $\Omega$). 
 Our work  aims at generalizing in the reversible case  the
 results of~\cite[Theorem 2.1]{FrWe} and~\cite{kamin1979elliptic,Kam,Per,day1984a,day1987r}, when $f$ has several critical points in
 $\Omega$. First, we exhibit a general set of assumptions on
 $f$ and an ensemble of initial conditions for which  the law of~$X_{\tau_\Omega}$ concentrates on
 points belonging to $\argmin_{\pa \Omega}f$ and we compute the
 relative probabilities to leave the domain through each of them (see
 Theorem~\ref{thm.main}): in this case, the limiting exit point
 distribution is the same as starting from the quasi-stationary
 distribution in $\Omega$, and we thus rely on our previous
 work~\cite{DLLN-saddle1}.
 Second, using this first result, we identify the exit points when the
 process starts more generally from initial conditions contained in a potential
 well which touches the boundary of $\Omega$, in a sense to be made precise (see Theorem~\ref{thm.2}
 and Theorem~\ref{thm.4}).

Concerning the assumptions on $f$,  $\partial_nf$ is not assumed to be
positive on~$\partial \Omega$ and $f$ may have several critical points
in~$\Omega$ with  larger energies than  $\min_{\partial
  \Omega}f$. However, we do not consider the case when~$f$  has
critical  points on~$\partial \Omega$. 

Here are representative examples of outputs of this work, which are
new to the best of our knowledge. Let us assume that~$f:\overline
\Omega\to \mathbb R$ and its restriction~$f|_{\pa \Omega}$  to the
boundary of~$\Omega$ are smooth Morse functions, and that~$f$ has no
critical point on $\pa \Omega$  (see~\eqref{H-M} below).
\begin{itemize}
\item We
prove that if   $\{y\in \Omega,
  \, f(y)<\min_{\pa \Omega}f\}$ is connected and contains all the 
local minima  of~$f$  in~$\Omega$, and if   $\pa_nf>0$ on
  $\argmin_{\pa \Omega}f$, then the 
law of~$X_{\tau_{\Omega}}$ concentrates on $\argmin_{\partial 
\Omega} f$ when $X_0 =x\in \{y\in \Omega,
  \, f(y)<\min_{\pa \Omega}f\}$ (see~Theorem~\ref{thm.main}).
  \item   Besides, when $\{y\in \Omega,
  \, f(y)<\min_{\pa \Omega}f\}$ is not connected (we denote in this case by   $\ft C_1,\ldots,\ft C_{\ft N}$ its connected components, with $\ft N\ge 2$) and if, for some $j\in \{1,\ldots,\ft N\}$, $\pa \ft C_j\cap \pa \Omega\neq \emptyset$ and 
  $\vert \nabla f\vert \neq 0$ on  $\pa \ft C_j$,     
   then, when $X_0 =x\in \ft C_j$,  the 
law of~$X_{\tau_{\Omega}}$ concentrates on $\pa \ft C_j\cap \pa \Omega
$ in the limit $h\to 0$ (see~Theorem~\ref{thm.2}). 
\item Furthermore,  when $X_0 =x\in \ft C_j$ and $z\in \argmin_{\pa \Omega} f\setminus \pa \ft C_j$, 
 for all sufficiently small neighborhood  $ \Sigma_z$ of $z$ in $\pa
 \Omega$, $\mathbb P  \left [
   X_{\tau_{\Omega}} \in \Sigma_z\right]=O(e^{-\frac ch})$ for $h$
 small enough (see item 1 in~Theorem~\ref{thm.2}).
\end{itemize}

%  and Remark~\ref{re.discussionHP}. 
%\end{sloppypar}
%  \item[(ii)]   When dropping the assumption that $\vert \nabla f\vert
%    \neq 0$ on  $\pa \ft C_j$, it may happen that, for all
%    sufficiently small neighborhood  $ \Sigma_z$ of $z$ in $\pa
%    \Omega$,  $\mathbb P  \left [ X_{\tau_{\Omega}} \in
%      \Sigma_z\right]=C\sqrt{h}\, (1+o(1))$, for some $C>0$
%    independent of $h$ (as explained in \cite[Section
%    1.4]{DLLN-saddle0}). \comment{C'est un peu bizarre qu'on fasse ref
%    a l'arXiv pour ca: on ne peut pas inclure ça dans ce papier ?}
  %\end{itemize} 
  
  Let us mention that the preprint~\cite{DLLN-saddle0} concatenates most of the results of
  this manuscript and~\cite{DLLN-saddle1}. A simplified version of the
  results of this work is also presented in~\cite{IHPLLN}.  

\medskip 

\noindent 
\textbf{On the results from~\cite{DLLN-saddle1}: metastability and the
  quasi-stationary distribution.  } As explained above, this article
generalizes to a broader class of initial conditions in $\Omega$ the results
of~\cite{DLLN-saddle1} where it is assumed that $X_0$ is distributed
according to  the quasi-stationary distribution in $\Omega$.  Let us quickly
recall what is the quasi-stationary distribution, and why it is
relevant to study the exit event starting from this distribution.

Let us assume  that~$\Omega
\subset \mathbb R^d$ is smooth, open, bounded and connected.
\begin{definition}\label{def:QSD}   A
quasi-stationary distribution for the dynamics~\eqref{eq.langevin} in
$\Omega$ is a probability measure $\nu_h$ supported in~$\Omega$ such that for all measurable sets $A\subset \Omega$ and for all $t\ge 0$\label{page.qsd}
\begin{equation}\label{eq.vhA} \nu_h(A)=\frac{\displaystyle\int_{\Omega} \mathbb P_x \left[X_t \in A,   t<\tau_{\Omega}\right]  \nu_h(dx ) }{\displaystyle \int_{\Omega}   \mathbb P_x \left[t<\tau_{\Omega} \right] \nu_h(dx )}.
\end{equation}
\end{definition}
Here and in the following, the subscript $x$ indicates that the
stochastic process starts from $x \in \mathbb R^d$ ($X_0=x$). 
 In words,~\eqref{eq.vhA} means that if
$X_0$ is distributed according to  $\nu_h$, then for all $ t>0$,  $X_t$
is still distributed according to  $\nu_h$ conditionally on 
$X_s \in \Omega$ for all $s \in [0,t]$.

The quasi-stationary distribution is related to an eigenvalue problem
on the infinitesimal generator of the dynamics (\ref{eq.langevin}), namely the differential operator
\begin{equation}\label{eq.L-0}
L_{f,h}= - \frac{h}{2}  \ \Delta+ \nabla f \cdot \nabla.
\end{equation}
Let us  define 
$$L^2_w(\Omega)=\Big\{u:\Omega \to \mathbb R 
,  \int_\Omega u^2 
 e^{-\frac{2}{h} f } \,   < \infty\Big\}.$$
The weighted Sobolev spaces $H^k_w(\Omega)$ are defined
similarly.  The operator $L_{f,h}$ with homogeneous Dirichlet boundary
conditions on $\partial \Omega$ is denoted by $L^{D}_{f,h}$. Its
domain is $\mathcal D(L^{D}_{f,h})=H^1_{w,0}(\Omega)\cap
H^2_w(\Omega)$, 
where $H^1_{w,0}(\Omega)=\{u \in H^1_w(\Omega), \, u=0 \text{ on
} \partial \Omega\}$. It is well know that $-L^{D}_{f,h}$ is self
adjoint on $L^2_w(\Omega)$, positive and has compact resolvent. Moreover, from standard results on
 elliptic operator (see for example~\cite{MR1814364, Eva}), its  smallest eigenvalue
  $\lambda_h$  (a.k.a. the principal eigenvalue)  
is non degenerate and its associated eigenfunction $u_h$
 has a sign on~$\Omega$ and is in $C^{\infty} (\overline \Omega)$. %
 Without loss of generality, one can then assume that:  
 \begin{equation} \label{eq.norma}
 u_{h}>0\ \text{on}\ \Omega\ \ \text{and}\ \ \int_{\Omega}u_{h}^{2} e^{-\frac 2hf} =1.
\end{equation}
The following result (see for example \cite{le2012mathematical})
relates the quasi-stationary distribution $\nu_h$ to the principal
eigenfunction $u_h$.
\begin{proposition} \label{uniqueQSD}
The unique quasi-stationary distribution $\nu_h$ associated with the
dynamics~\eqref{eq.langevin} and the domain~$\Omega$ is given by:
\begin{equation} \label{eq.expQSD}
\nu_h(dx)=\frac{\displaystyle  u_h(x) e^{-\frac{2}{h}  f(x)}}{\displaystyle \int_\Omega u_h(y) e^{-\frac{2}{h}  f(y)}dy}\,  dx.
\end{equation}
Moreover, whatever the law of the
initial condition $X_0$ with support in~$\Omega$, it holds:
\begin{equation}\label{eq.cv_qsd}
\lim_{t \to \infty} \| {\rm Law}(X_t| t < \tau_\Omega) - \nu_h \|_{TV} = 0.
\end{equation}
\end{proposition}
Here,~${\rm Law}(X_t| t < \tau_\Omega)$ denotes the law of $X_t$ conditional to the event $\{t<\tau_\Omega\}$.
For a given initial distribution of the process~\eqref{eq.langevin}, if  the convergence in~\eqref{eq.cv_qsd} is much quicker than the exit from $\Omega$, the exit from the domain~$\Omega$ is said to be metastable. In~\cite{DLLN-saddle1}, we have investigated 
the concentration of the law of~$X_{\tau_{\Omega}}$ on  $\argmin_{\pa
  \Omega}f$ in the limit $h\to 0$  when $X_0\sim \nu_h$, namely when
the exit is metastable. In this work,
we extend this study to the general case: $X_0=x \in \Omega$.

\subsection{Main results}
\label{nota-hypo}
In all this work, we assume that~$\Omega
\subset \mathbb R^d$ is smooth, open, bounded and connected, and that $f : \overline \Omega \to \mathbb R$ is a $C^{\infty}$ function.
This section is dedicated to the statement    of the main result of this work.

\subsubsection{Assumptions}
%In the following, we consider a setting that is more general than the one  of Section~\ref{qsd}: 
% $\overline\Omega$ is a $C^{\infty}$ oriented compact and  connected Riemannian manifold of dimension $d$ with boundary $\partial \Omega$.    

%%%%
 Let us now introduce the basic assumption which is used throughout  this work:
%\begin{itemize}
%
%\item 
%\eqref{H-M}  
\begin{equation}
\tag{\textbf{A0}}\label{H-M}
 \left.
    \begin{array}{ll}
        &\text{The function $f : \overline \Omega \to \mathbb R$ is a $C^{\infty}$ Morse function.}\\
       &\text{For all $x\in \pa \Omega$, $\vert \nabla f(x)\vert \neq 0$.}   \\
 &\text{The function $f : \{x \in  \pa \Omega, \pa_n f(x) > 0\} \to \mathbb R$ is  a Morse function.}\\
   &\text{The function~$f$   has  at least one local minimum in  $\Omega$.}
    \end{array}
\right \}
 \end{equation}

\begin{remark}
We recall that a function \label{page.HM} $\phi: \overline \Omega \to \mathbb R$ is a  Morse function if all its  critical points are non degenerate (which implies in particular that $\phi$ has a finite number of critical points since $\overline \Omega$ is compact and a non degenerate critical point is isolated from the other critical points). 
Let us recall that a critical point  $z\in \overline\Omega$ of~$\phi$  is non degenerate if the hessian matrix of~$\phi$ at $z$, denoted by    $\Hess \phi(z)$,  is invertible.
  We refer for example to~\cite[Definition 4.3.5]{Jost:2293721} for a
  definition of the hessian matrix on a manifold (see
  also~\cite[Remark 10]{di-gesu-le-peutrec-lelievre-nectoux-16} for
  explicit formulas).  
 A non degenerate critical point $z\in \overline\Omega$ of~$\phi$ is  said to have index $p\in\{0,\dots,d\}$
if   $\Hess \phi(z)$ has precisely $p$
negative eigenvalues (counted with multiplicity). In the case $p=1$,
$z$  is called a saddle point.
\end{remark}
%\begin{remark}
%The function $f$ and the domain $\Omega$  are assumed to be $C^\infty$ because some of the estimates we derive in this work and those obtained~\cite{DLLN-saddle1} (we will indeed use the results of~\cite{DLLN-saddle1})    are based on bootstrap arguments which require enough regularity of solutions of elliptic partial differential equations (see in particular the proof of Proposition~\ref{level}) and   which are needed to construct    WKB-approximations in~\cite[Section 3.2.2]{DLLN-saddle1}.  
%\end{remark}
  \noindent    
In order to introduce the remaining assumptions, we need to introduce
three notations. First, the following notation will be used for the
level sets of $f$:  for $a\in \mathbb R$,
$$\{f<a\}=\{x\in \overline \Omega, \ f(x)<a\},\ \ \{f\le a\}=\{x\in \overline \Omega,\  f(x)\le a\},$$
and\label{page.fa}
$$\{f=a\}=\{x\in \overline \Omega, \ f(x)=a\}.$$%%%%%
Second, for any local minimum $x$ of~$f$ in~$\Omega$, one defines\label{page.hfx}
\begin{equation}\label{eq.Hfx}  
\ft H_f (x):=\,  \inf   \Big\{  \max_{t\in [0,1]}\,  f\big ( \gamma(t)\big)\, \big | \,  \gamma \in C^0([0,1], \overline \Omega), \,  \gamma(0)=x, \text{ and }  \gamma(1)\in \pa \Omega\Big\},  
\end{equation}
where $C^0([0,1], \overline \Omega)$ is the set of continuous paths
from $[0,1]$ to $\overline \Omega$. 
 Intuitively, $\ft H_f(x)$ is the minimal energy any path connecting
$x$ to the boundary $\pa \Omega$ has to cross. This energy is necessarily either
the energy of a saddle point~$z$ in $\Omega$ (see e.g. $\partial \ft C_2$ on
Figure~\ref{fig:okay}), or the energy of a generalized saddle point~$z$ on $\pa
\Omega$ (see e.g. $\partial \Omega \cap \partial \ft C_3$ and $\partial 
\Omega
\cap \partial  \ft C_{\ft{max}}$ on Figure~\ref{fig:okay}; see also Equation~\eqref{eq.mathcalU1_bis} for 
a proper
definition of a generalized saddle point).
%
% 
% 
% Roughly speaking, $\ft H_f(x)$ is the height of the saddle connecting $x\in \Omega$ to $\mathbb R^d/\Omega$. This height can be reached by a saddle point $z$ in $\Omega$, i.e. a critical point of index $1$ of $f$ (see $\ft C_2$ in Figure~\ref{fig:okay}), or by a generalized saddle point $z$ on $\pa \Omega$, (see~\eqref{eq.mathcalU1_bis}, and $\ft C_3$ and $\ft C_{\ft{max}}$ in Figure~\ref{fig:okay}). \\
Third, for a local minimum $x$ of $f$ in $\Omega$, 
%\begin{equation}\label{eq.Cdef2}
\begin{equation}\label{eq.Cdef2}
\begin{aligned}
\ft C(x)  \text{ is the connected component of } \{f< \ft H_f(x)\} \text{ containing } x. 
\end{aligned}
\end{equation} 
Finally, set 
\begin{align}
\label{mathcalC-def}
&\mathcal C:=\big \{ \ft C (x), \, x  \text{ is  a local minimum of $f$ in }\Omega \big \}.
\end{align} 
\noindent
Let us mention  that  when \eqref{H-M} holds, one has: for all local
minima $x$ of $f$ in $\Omega$, $\ft C (x)$ is an open subset of $ \Omega$  (see~\cite[Remark 7]{DLLN-saddle0}). 
%Let us now define a set of assumptions which will ensure that
%\textbf{[P1]} and \textbf{[P2]} are satisfied  (see indeed
%Theorem~\ref{thm.main} and Section~\ref{discussion-hyp} for a
%discussion on these assumptions):

We are now in position to state the assumptions we will use in addition to~\eqref{H-M}:
\begin{itemize} 
\item
First geometric assumption:
%\begin{equation}\tag{\textbf{A1}}\label{eq.hip1}
% \argmax\,  \big \{ \ft H_f (x)-f(x), \ x \text{ is local minimum of } f \text{ in } \Omega \big\} \subset \ft C_{\ft{max}}.
%\end{equation}
\begin{equation}\tag{\textbf{A1}}\label{eq.hip1}
\text{~\eqref{H-M} holds and } \exists ! \ft C_{\ft{max}}  \in \mathcal C \text{ s.t. } \max\limits_{\ft C\in \mathcal C}  \,  \Big \{  \max_{\overline{\ft C}}f-\min_{\overline{\ft C}} f   \Big\} =  \max_{\overline{\ft C_{\ft{max}} }}f-\min_{\overline{\ft C_{\ft{max}} }} f.
\end{equation}
%\begin{sloppypar} 

\item    Second geometric assumption:
\label{page.hypo}
\begin{equation}\tag{\textbf{A2}} \label{eq.hip2}
\text{ \eqref{eq.hip1} holds and } \pa \ft C_{\ft{max}} \cap \pa \Omega\neq \emptyset.
\end{equation}
\item    
Third geometric assumption:
  \begin{equation}\tag{\textbf{A3}} \label{eq.hip3}
\text{ \eqref{eq.hip1} holds and } \pa  \ft C_{\ft{max}} \cap \pa \Omega\subset \argmin_{\pa \Omega} f.
\end{equation}  
 \end{itemize} 
  %To avoid any ambiguity, let us emphasize  that assumption~\eqref{H-M} is assumed to be satisfied in~\eqref{eq.hip1}, and  assumption~\eqref{eq.hip1} is assumed to be satisfied in~\eqref{eq.hip2} and in~\eqref{eq.hip3} (in~\eqref{eq.hip4} below as well). \\
  Assumption~\eqref{eq.hip1}  implies that there is   a unique  deepest well, namely $\ft C_{\ft{max}}$. 
  Assumptions~\eqref{eq.hip2} and~\eqref{eq.hip3} mean that the closure of this deepest well intersects $\pa \Omega$ at  points where $f$ reaches its minimum on $\pa \Omega$. Notice that these assumptions imply that $\ft C_{\ft{max}}$ contains the global minima of $f$ in $\Omega$. 
Assumptions~\eqref{eq.hip2} and~\eqref{eq.hip3}   ensure that when $X_{0}\sim\nu_h$ ($\nu_h$ being
the quasi-stationary distribution introduced in
Definition~\ref{def:QSD}) or $X_0=x\in  \ft C_{\ft{max}}$,   the law
of  $X_{\tau_{\Omega}}$ concentrates on the set $ \pa \ft C_{\ft{max}}
\cap \pa \Omega $, (see~\cite[Theorem 1]{DLLN-saddle1} and items~1 and~2 in
Theorem~\ref{thm.main} below).  
Finally, the last assumption is:
  \begin{itemize}
\item    Fourth geometric assumption:
%for any  $\ft C\in \mathcal C\setminus\{\ft C_{\ft{max}} \}$  it holds
 \begin{equation} \tag{\textbf{A4}} \label{eq.hip4}
\text{ \eqref{eq.hip1} holds and }\pa \ft C_{\ft{max}}  \cap \Omega \text{ contains no  \textit{separating saddle point} of $f$,}
 %\pa \ft C_{\ft{max}}  \cap \pa \ft C=\emptyset.  
 \end{equation}
where the proper definition  of a separating saddle point of $f$   is  introduced below in~Section~\ref{sec.SSP}. 
 \end{itemize}
 \noindent
 Assumption~\eqref{eq.hip4} is equivalent to the fact that any minimal energy
path connecting a point in $\ft C_{\ft{max}}$ to the boundary $\partial
\Omega$ remains necessarily within $\overline{\ft C_{\ft{max}}}$. In particular, such a path leaves $\Omega$ on $\partial \ft C_{\ft{max}} \cap \partial \Omega$. Notice indeed that  Assumption~\eqref{eq.hip4} implies  $\pa  \ft C_{\ft{max}} \cap \pa \Omega\neq \emptyset$   (this is a consequence of~\cite[Proposition 15]{DLLN-saddle1}). 
  The assumptions~\eqref{eq.hip4} together
  with~\eqref{H-M},~\eqref{eq.hip1},~\eqref{eq.hip2},
  and~\eqref{eq.hip3},  ensure that the probability  that the
  process~\eqref{eq.langevin} starting from $x\in  \ft C_{\ft{max}}$
  leaves $\Omega$ through any sufficiently small    neighborhood of
  $z\in \pa \Omega \setminus \pa \ft C_{\ft{max}}$ in~$\pa \Omega$  is
  exponentially small when $h\to 0$, see item~3 in
  Theorem~\ref{thm.main} below. We refer to Remark~\ref{re.discussionHP} below for
  a discussion on the necessity of the
  assumptions~\eqref{H-M},~\eqref{eq.hip1},~\eqref{eq.hip2},~\eqref{eq.hip3},
  and \eqref{eq.hip4} to get these results. 
  % ~\cite[Section
 % 1.4]{DLLN-saddle0}  for 
 Figure~\ref{fig:okay} gives a one-dimensional example where~\eqref{eq.hip1},~\eqref{eq.hip2},~\eqref{eq.hip3},  and~\eqref{eq.hip4} are satisfied, and Figure~\ref{fig:shema_nota}  gives a typical example where~\eqref{eq.hip1},~\eqref{eq.hip2},
 and~\eqref{eq.hip3} are satisfied  but not~\eqref{eq.hip4}.

%\begin{remark}
%It is proved in Proposition~\ref{pr.p1} that when \eqref{H-M} holds, for all local minima $x$ of $f$ in $\Omega$, one has $\ft C (x)\subset \Omega$ (see~\eqref{eq.Cdef2}). This implies that for all $y\in \ft C (x)$, $t_y=+\infty$  and  then, $\ft C (x)\subset \mathcal A(\ft C (x))$.  
%\end{remark}
 %%%%%%%%%%

  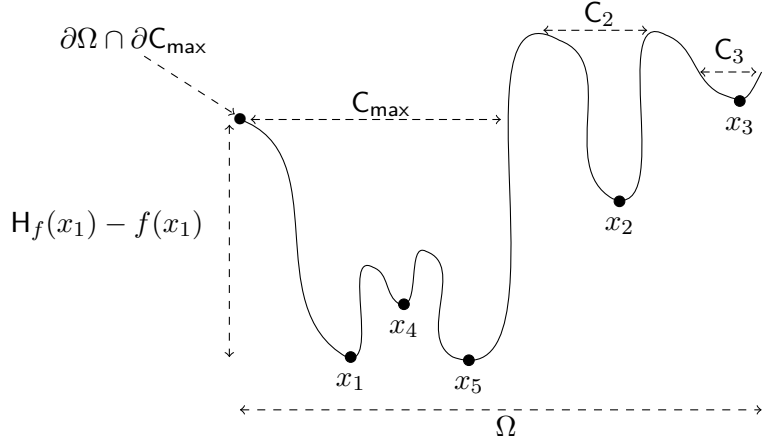
\begin{figure}[h!]
\begin{center}
\begin{tikzpicture}[scale=0.7]
\coordinate (b1) at (0,4.5);
\coordinate (b2) at (2 , 0);
\coordinate (b2a) at (2.5 , 1.7);
\coordinate (b2b) at (3 , 1);
\coordinate (b2c) at (3.5 , 2);
\coordinate (b4) at (4,0);
\coordinate (b5) at (6,6);
\coordinate (b6) at (7,3);
\coordinate (b7) at (8,6.1);
\coordinate (b8) at (9,5);
\coordinate (b9) at (10,6);
  \draw (b8) ..controls (9.5,4.8)   .. (9.8,5.4) ;
\draw [black!100, in=150, out=-10, tension=10.1]
  (b1)[out=-20]   to  (b2)  to (b2a) to (b2b) to (b2c)   to (b4) to (b5) to (b6) to (b7) to (b8);
     \draw [dashed, <->]   (0,-1) -- (9.8,-1) ;
      \draw (5,-1.3) node[]{$\Omega$};
   \draw [dashed, <->]   (-0.2,4.4) -- (-0.2,0) ;
      \draw (2.65,4.8) node[]{$ \ft C_{\ft{max}}$};
         \draw [densely  dashed, <->]  (0.2,4.5) -- (4.9,4.5) ;
          \draw (6.75,6.5) node[]{$ \ft C_2$};
             \draw [densely  dashed, <->]  (5.7,6.19) -- (7.65,6.19) ;
             \draw [densely  dashed, <->]  (8.65,5.4) -- (9.7,5.4) ;
              \draw (9.2,5.8) node[]{$ \ft C_3$};
          \draw (-2.5,2.5) node[]{$ \ft H_f(x_1)- f(x_1)$};
               \draw (-2,6) node[]{$ \pa \Omega \cap \pa \ft C_{\ft{max}}$};
\draw [dashed, ->]   (-1.8,5.7) -- (-0.1,4.6) ;
      \tikzstyle{vertex}=[draw,circle,fill=black,minimum size=4pt,inner sep=0pt]
        \draw (0.0,4.5) node[]{$\bullet$}; 
\draw (2.08,0) node[vertex,label=south: {$x_1$}](v){}; 
\draw (7.13,2.95) node[vertex,label=south: {$x_2$}](v){}; 
\draw (9.39,4.85) node[vertex,label=south: {$x_3$}](v){}; 
\draw (4.3,-0.06) node[vertex,label=south: {$x_5$}](v){}; 
\draw (3.08 , 1) node[vertex,label=south: {$x_4$}](v){}; 
    \end{tikzpicture}
\caption{A one-dimensional case where~\eqref{eq.hip1},~\eqref{eq.hip2},~\eqref{eq.hip3} and~\eqref{eq.hip4} are satisfied.  
%The point $x_{\ft{max}}$  defined in~\eqref{eq.x1-intro} could  also have been chosen equal to $x_2$ since $\argmax\,  \big \{ \ft H_f (x)-f(x), \ x \text{ is local minimum of } f \text{ in } \Omega \big\} =\{x_{\ft{max}},x_2\}$. 
On the figure, $f(x_1)=f(x_5)$, $\ft H_f(x_1)=\ft H_f(x_4)=\ft H_f(x_5)$,~$\mathcal C=\{\ft C_{\ft{max}}, \ft C_2,\ft C_3\}$,  $\pa  \ft C_2\cap \pa  \ft C_{\ft{max}}=\emptyset$ and $\pa \ft C_3\cap \pa  \ft C_{\ft{max}}=\emptyset$. }
 \label{fig:okay}
 \end{center}
\end{figure}

\begin{figure}[h!]
\begin{center}
\begin{tikzpicture}[scale=0.8]
\tikzstyle{vertex}=[draw,circle,fill=black,minimum size=4pt,inner sep=0pt]
\tikzstyle{ball}=[circle, dashed, minimum size=1cm, draw]
\tikzstyle{point}=[circle, fill, minimum size=.01cm, draw]
\draw [rounded corners=10pt] (1,0.5) -- (-0.25,2.5) -- (1,5) -- (5,6.5) -- (7.6,3.75) -- (6,1) -- (4,-0.3) -- (2,0) --cycle;
\draw [thick, densely dashed,rounded corners=10pt] (1.5,0.5) -- (.25,1.5) -- (0.5,2.5) -- (0.09,3.5) -- (2.75,3.75) -- (3.5,3) -- (2.4,2) -- (2.9,1.5) --cycle;
\draw [thick, densely dashed,rounded corners=10pt]    (3,3.9)  -- (3.4,3) -- (4.8,3.3) --(5.5,2.9)--(6.5,4) --(6.5,5)  -- (5.3,6) -- (3,4.58)  --cycle;
%%%
\draw [thick, densely dashed,rounded corners=10pt] (5.9,0.78)--(6,1.7) -- (5,2.3) -- (3.7,1.5) -- (3.4,0.6) --cycle;
 \draw (1.4,1.3) node[]{$\ft C_{\ft{max}}$};
  \draw (5.4,5.5) node[]{$\ft C_2$};
    \draw (4.3,1) node[]{$\ft C_3$};
     \draw  (2.4,4.9) node[]{$\Omega$};
    \draw  (7.8,3) node[]{$\pa \Omega$};
\draw (3.3 ,3.2) node[vertex,label=north east: {$z_5$}](v){};
\draw  (5.9,0.99) node[vertex,label=south east: {$z_4$}](v){};
\draw (1.7 ,2.5) node[vertex,label=north west: {$x_1$}](v){};
\draw (4.9 ,4.4) node[vertex,label=north : {$x_2$}](v){};
%\draw (3.1,1.2) node[vertex,label=south : {$z_6$}](v){};
\draw (0.38,1.45) node[vertex,label=south west: {$z_1$}](v){};
\draw (6.2,5.2) node[vertex,label=north east: {$z_3$}](v){};
\draw (0.17,3.4) node[vertex,label=north west: {$z_2$}](v){};
\draw(5,1.5)  node[vertex,label=north: {$x_3$}](v){};
%\draw (5.3,2.56)  node[vertex,label=east: {$z_7$}](v){};
%\draw (3.8,2.3) node[vertex,label=north : {$y_m$}](v){};
\end{tikzpicture}

\begin{tikzpicture}[scale=0.7]
\tikzstyle{vertex}=[draw,circle,fill=black,minimum size=5pt,inner sep=0pt]
\tikzstyle{ball}=[circle, dashed, minimum size=1cm, draw]
\tikzstyle{point}=[circle, fill, minimum size=.01cm, draw]

\draw [dashed] (-5.4,-0.6)--(6,-0.6);
\draw [dashed,->] (-5.4,-0.6)--(-5.4,3);
\draw [dashed] (6,-0.6)--(6,3);
\draw [densely dashed,<->] (-5.3,-1.4)--(5.8,-1.4);
 \draw (0,-1.68) node[]{$\pa \Omega$};
 \draw (-5.9,2.8) node[]{$f|_{\pa \Omega}$};
 \draw[thick] (-5.4,2) ..controls  (-5.2,1.96).. (-5,1.6);
\draw[thick] (-5,1.6) ..controls  (-3.7,-1.29).. (-2,1.6);
\draw[thick] (-1.5,1.6) ..controls  (-1,-1.58).. (0.2,2.8) ;
\draw[thick] (-2,1.6) ..controls  (-1.87,1.9) and (-1.63,2.1) .. (-1.5,1.6);
\draw [thick] (0.2,2.8)  ..controls  (0.3,3.2) and (0.67,3.4).. (0.8,3);
\draw[thick] (0.8,3) ..controls  (1.8,-1.68).. (2.7,2.4);
\draw[thick] (2.7,2.4) ..controls (2.8,2.7) and (3.1,2.7)..    (3.3,2.4) ;
\draw[thick] (3.3,2.4) ..controls  (4.3,0).. (5.5,1.5);
\draw[thick] (5.5,1.5) ..controls  (5.83,2).. (6,2);

\draw (1.8,-0.58) node[vertex,label=south: {$z_3$}](v){};
\draw (-3.7,-0.6) node[vertex,label=south: {$z_1$}](v){};
\draw (-0.95,-0.6)  node[vertex,label=south: {$z_2$}](v){};
\draw  (4.45,0.45)  node[vertex,label=south: {$z_4$}](v){};
\end{tikzpicture}

\caption{Schematic representation of     $\mathcal C$ (see~\eqref{mathcalC-def}) and $f|_{\pa \Omega}$ when the assumptions~\eqref{H-M}, ~\eqref{eq.hip1},~\eqref{eq.hip2} and~\eqref{eq.hip3} are satisfied. 
  In this representation, $x_1\in \Omega$ is the global minimum of $f$ in $\overline \Omega$ and  the other local minima of $f$ in $\Omega$ are  $x_2$  and $x_3$. Moreover, $\min_{\pa \Omega} f=f(z_1)=f(z_2)=f(z_3)=\ft H_f(x_1)=\ft H_f(x_2)<\ft H_f(x_3)=f(z_4)$, $\{f<\ft H_f(x_1)\}$ has two connected components: $\ft C_{\ft{max}}$ (see~\eqref{eq.hip1}) which contains $x_1$ and  $\ft C_2$ which contains~$x_2$.  Thus, one has $\mathcal C=\{\ft C_{\ft{max}},\ft C_2, \ft C_3\}$. In addition, $\{z_1,z_2,z_3\}=\argmin_{\pa \Omega} f$ ($\ft k_1^{  \pa \Omega}=3$),   $\{z_5\}=\overline{\ft C_{\ft{max}}}\cap\overline{\ft  C_2}$ is a separating saddle point of $f$ (thus~\eqref{eq.hip4} is not satisfied), and $\pa \ft C_{\ft{max}}\cap \pa \Omega=\{z_1,z_2\}$ ($\ft k_1^{\pa \ft C_{ \ft{max}}}=2$).  %The point $y_m\in \Omega$ is a local maximum of $f$ with $f(y_m)>f(z_i)$ for all $i\in \{1,\ldots,7\}$.   
}
%The point $z_4$ is a  {separating saddle} point as introduced in Definition~\ref{de.SSP}.}
 \label{fig:shema_nota}
 \end{center}
\end{figure}
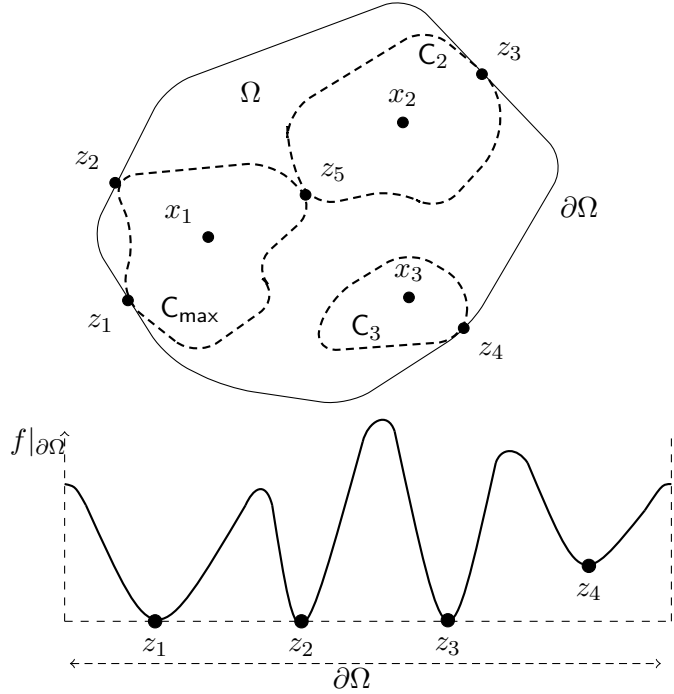

\subsubsection{Notation associated with the function~$f$ }
\label{sec.SSP}

In order to state our main result, we will need a few more notation
associated with the function~$f$.

\paragraph{Domain of attraction $\mathcal A(D)$.} The  domain of attraction of a subset
$D$ of $\Omega$ for the dynamics $\dot{x}=-\nabla f(x)$ is defined as follows.  Let $x\in \Omega$ and denote by~$\varphi_{t}(x)$ the solution to the ordinary differential equation 
\begin{equation}\label{hbb}
  \frac{d}{dt}\varphi_{t}(x)=-\nabla f(\varphi_{t}(x)) \text{ with } \varphi_{0}(x)=x,
  \end{equation}
on the interval $t\in [0,t_x]$, where $t_x > 0$ is defined by
$$t_x=\inf  \{t\ge 0, \ \varphi_{t}(x)\notin \Omega\}.$$
Let $x\in \Omega$ be such that $t_x=+\infty$. The $\omega$-limit set of $x$, denoted by $\omega(x)$, is defined by
$$\omega(x)=\{y\in \overline \Omega, \, \exists (s_n)_{n\in \mathbb N} \in (\mathbb R_+)^{\mathbb N}, \,\lim_{n\to \infty}s_n=+ \infty, \,\lim_{n\to \infty}\varphi_{s_n}(x)=y \}.$$
Let us recall that the $\omega$-limit set $\omega(x)$ is included in the set of the critical points of $f$ in~$\overline \Omega$. 
Moreover, when $f$ has a finite number of critical points in
$\overline \Omega$,  $\omega(x)$ is either empty (if $t_x < \infty$)
or of cardinality one (if $t_x = \infty$).  
%$$\exists y\in \overline \Omega, \  \omega(x)=\{y\}.$$
Let $D$ be a subset of $\Omega$. The domain of attraction of a subset $D$ of $\Omega$ is  defined by\label{page.AD}
\begin{equation}\label{eq.ad}
 \mathcal A(D)=\{ x\in \Omega, \,t_x=+\infty \text{ and } \omega(x)\subset D\}.
 \end{equation}
\begin{remark}
Recall that  when \eqref{H-M} holds, one has: for all local
minima $x$ of $f$ in $\Omega$, $\ft C (x)\subset \Omega$,
where $\ft C(x)$ is defined in~\eqref{eq.Cdef2} (see~\cite[Remark 7]{DLLN-saddle0}). This implies that for all $y\in \ft C (x)$,
$t_y=+\infty$  and  then, $\ft C (x)\subset \mathcal A(\ft C
(x))$.
% \comment{Est-ce que c'est le bon endroit pour faire cette
 % remarque ? A quoi sert-elle exactement ?}
\end{remark}

% The main purpose of this section is to introduce the  local minima   and the  generalized  saddle points of~$f$. These elements of $\overline \Omega$ are used extensively throughout this work and  play a crucial role in our analysis. Roughly speaking, the generalized  saddle points of~$f$ are the saddle points $z\in \overline \Omega$ of the extension of~$f$ by~$-\infty$ outside $\overline \Omega$.  Thus, when the function $f$ satisfies the assumption \eqref{H-M}, a generalized  saddle point of~$f$ (as introduced in~\cite{HeNi1}) is either  a saddle point $z\in \Omega$ of~$f$ or a local minimum $z\in \pa \Omega$ of $f|_{\pa \Omega}$ such that $\pa_nf(z)>0$.
\paragraph{Generalized saddle points of $f$.}  We introduce in this
paragraph an ensemble of points on $\partial \Omega$, which will
contain the exit points of the process $(X_t)_{t \ge0}$ from $\Omega$.
Let us define 
$$ 
\ft U_1^{\pa \Omega}:=\{z \in \pa \Omega, \,  z \text{ is a local minimum of } f|_{\pa \Omega} \text{ but }  \text{not a local minimum of~$f$ in~$\overline \Omega$}\}.
$$
 Notice that an equivalent definition of $\ft U_1^{\pa \Omega}$ is 
 \begin{align}
\label{eq.mathcalU1_bis}
\ft U_1^{\pa \Omega}=\{z \in \pa \Omega, \,  z \text{ is a local minimum of } f|_{\pa \Omega}  \, \text{ and }  \, \pa_nf(z)>0\},
\end{align}
which follows from the fact that $\vert \nabla f(x)\vert  \neq 0$ for all $x\in \pa \Omega$.
%If $\ft U_1^{\pa \Omega}$ is not empty, its elements are denoted by:
%\begin{equation}\label{eq.U1paOmega}
%\ft U_1^{\pa \Omega}=\{z_1,\ldots,z_{\ft m_1^{\pa \Omega}}\}\subset \pa \Omega.
%\end{equation}
The set $\ft U_1^{\pa \Omega}$ contains the so-called {\em generalized
saddle points} of $f$ on the boundary $\pa
\Omega$~\cite{HeNi1,Lep}. These generalized saddle points are indeed
geometrically saddle points of the function $f$ extended by $-\infty$
outside $\overline{\Omega}$, which is consistent with the fact that
  homogeneous Dirichlet boundary conditions are used to define the operator $L^D_{f,h}$.
  
Let us introduce the following notation for the $\ft k_1^{\pa \Omega}$ elements of  $\ft U_1^{\pa \Omega}\cap \argmin_{\pa \Omega} f$:\label{page.k1paomega}
\begin{equation}\label{eq.z11}
\{z_1,\ldots,z_{\ft k_1^{\pa \Omega}}\}=\ft U_1^{\pa \Omega}\cap \argmin_{\pa \Omega} f.
\end{equation}
%Notice that  $\ft k_1^{\pa \Omega}\in \{0,\ldots, \ft m_1^{\pa \Omega}\}$. 
Let us assume that the assumptions~\eqref{eq.hip1},~\eqref{eq.hip2}, and~\eqref{eq.hip3} are satisfied.  
Let us recall that $\ft C_{\ft{max}}$ is defined in~\eqref{eq.hip1}. Moreover, in this case, one has  $\ft k_1^{\pa \Omega}\ge 1$ and
$$\pa\ft C_{\ft{max}} \cap  \pa \Omega \subset \{z_1,\ldots,z_{\ft k_1^{\pa \Omega}}\}.$$
Indeed, by assumption $ \pa \ft C_{\ft{max}} \cap  \pa \Omega\subset \{f=\min_{\pa \Omega} f\}$ (see~\eqref{eq.hip3}) and there is no local minimum of~$f$ in~$\overline \Omega$  on~$\pa \ft C_{\ft{max}}$ (since $\ft C_{\ft{max}}$ is a connected component of a sublevel set of~$f$). 
 We assume lastly  that the set $\{z_1,\ldots,z_{\ft k_1^{\pa \Omega}}\}$ is  ordered such that:\label{page.k1pacmax}
\begin{equation}\label{eq.k1-paCmax}
\{z_1,\ldots,z_{\ft k_1^{\pa \ft C_{\ft{max}} }}\}=\{z_1,\ldots,z_{\ft k_1^{\pa \Omega}}\} \cap \pa  \ft C_{\ft{max}}.
\end{equation}
Notice that $\ft k_1^{\pa \ft C_{\ft{max}}}\in \mathbb N^*$ and $\ft k_1^{\pa \ft C_{\ft{max}}}\le \ft k_1^{\pa \Omega}$.  

\paragraph{Separating saddle points.} Let us finally introduce the notion of {\em separating saddle point}, following~\cite[Section 4.1]{HeHiSj}.
To this end, let us  first recall that according to~\cite[Section 5.2]{HeNi1}, for any non critical point $z\in\Omega$,  for  $r>0$ small enough
$$
\{f<f(z)\}\cap B(z,r) \text{ is connected},
$$
and  for any critical point $z\in\Omega$ of index $p$ of the Morse function $f$,  for  $r>0$ small enough, 
one has the three possible cases:
\begin{itemize}
\item either   $p=0 $ (i.e. $z$  is a local minimum of  $f$)		and  $\{f<f(z)\}\cap B(z,r)=\emptyset$,  \item or $p=1$ \text{ and  } $\{f<f(z)\}\cap B(z,r)$ \text{ has exactly two connected components}, 
\item  or $p\ge 2$ \text{ and  } $\{f<f(z)\}\cap B(z,r)$ \text{ is connected}, 
\end{itemize}
where $B(z,r):=\{x\in \overline \Omega \ \text{s.t.} \ |x-z|<r\}$. 
A separating saddle point of $f$ in $\Omega$  is a saddle point of $f$ in $\Omega$ such that   for any $r>0$ small enough, the two connected components of $\{f<f(z)\}\cap B(z,r)$ 
   are contained in different connected components  of
 $\{ f< f(z) \}$. 

\medskip

Figure~\ref{fig:shema_nota} gives an illustration of the notations
introduced in this section.

\subsubsection{Main results on the exit point distribution}

In Theorem~\ref{thm.main}, we first make explicit a geometric setting
and an ensemble of initial conditions for which the exit distribution
is the same as when starting from the quasi-stationary distribution.

\begin{theorem}\label{thm.main}
Let us assume  that the assumptions \eqref{H-M},~\eqref{eq.hip1},~\eqref{eq.hip2}, and~\eqref{eq.hip3} are satisfied. 
  Let  $F\in L^{\infty}(\partial \Omega,\mathbb R)$ and $(\Sigma_{i})_{i\in\{1,\dots,\ft k_{1}^{\pa \Omega}\}}$ be a family  of pairwise disjoint open subsets of~$\pa \Omega$ (i.e. such that $\Sigma_i\cap \Sigma_j=\emptyset$ whenever  $i\neq j$) such that 
  $$\text{for all } i\in\big \{1,\dots,\ft k_{1}^{\pa \Omega}\big \},  \ z_{i}\in \Sigma_{i},$$ 
  where we recall that $\big \{z_1,\dots,z_{\ft k_{1}^{\pa\Omega}}\big
  \}= \ft U_1^{\pa \Omega}\cap \argmin_{\pa \Omega} f$
  (see~\eqref{eq.z11}).   Let~$K$ be a compact subset of $\Omega$ such
  that $ K\subset \mathcal A(\ft C_{\ft{max}})$ (see~\eqref{eq.hip1}
  and~\eqref{eq.ad} for the definitions of $\ft C_{\ft{max}}$ and~$\mathcal A$).  Let $x\in K$. 
    Then:
  \begin{enumerate} 
  \item 
There exists $c>0$ such that in the limit  $h\to 0$:
\begin{equation} \label{eq.t1}
\mathbb E_x \left [ F\left (X_{\tau_{\Omega}} \right )\right]=\sum \limits_{i=1}^{\ft k_1^{\pa \Omega}}
\mathbb E_x \left [ (\mathbf{1}_{\Sigma_{i}}F)\left (X_{\tau_{\Omega}} \right )\right]  +O\big (e^{-\frac ch}\big )
 \end{equation}
 and
 \begin{equation} \label{eq.t2}
\sum \limits_{i=\ft k_1^{\pa \ft C_{\ft{max}} }+1}^{\ft k_1^{\pa \Omega}}
\mathbb E_x \left [ (\mathbf{1}_{\Sigma_{i}}F)\left (X_{\tau_{\Omega}} \right )\right]  =O\big (h^{\frac14} \big ),
 \end{equation}
 where we recall that $\big \{z_1,\ldots,z_{\ft k_1^{\pa \ft C_{\ft{max}} }}\big \}=\pa \ft C_{\ft{max}}\cap \pa \Omega$ (see~\eqref{eq.k1-paCmax}). 
\item When for some $i\in\big \{1,\dots,\ft k_{1}^{\pa \ft C_{\ft{max}} }\big \}$ the function  $F$ is $C^{\infty}$ in a neighborhood  of $z_{i}$, one has when $h\to 0$:
\begin{equation} \label{eq.t3}
 \mathbb E_x \left [ (\mathbf{1}_{\Sigma_{i}}F)\left (X_{\tau_{\Omega}} \right )\right]=F(z_i)\,a_{i} +O(h^{\frac14}),
\end{equation}
 where\label{page.ai}
 \begin{equation} \label{ai}
 a_i=\frac{  \partial_nf(z_i)      }{  \sqrt{ {\rm det \ Hess } f|_{\partial \Omega}   (z_i) }  } \left (\sum \limits_{j=1}^{\ft k_1^{\pa \ft C_{\ft{max}} }} \frac{  \partial_nf(z_j)      }{  \sqrt{ {\rm det \ Hess } f|_{\partial \Omega}   (z_j) }  }\right)^{-1}. \end{equation}
 \item 
When~\eqref{eq.hip4} is satisfied 
 the remainder term  $O( h^{\frac14})$ in \eqref{eq.t2}
 is   of the order  $O\big (e^{-\frac{c}{h}}\big )$ for some $c>0$  
and
the remainder term  $O\big (h^{\frac14}\big )$  in \eqref{eq.t3} is of the order $O(h)$  and admits a full asymptotic expansion in~$h$ (see~\eqref{eq.asymptoO(h)} below). 
\end{enumerate} 
Finally,   the constants involved in  the remainder terms in~\eqref{eq.t1},~\eqref{eq.t2}, and~\eqref{eq.t3}  are uniform with respect to $x\in K$. 
\end{theorem}
\noindent
Let us recall that for $\alpha>0$,~$(r(h))_{h>0}$ admits a full asymptotic expansion in~$h^\alpha$   if there exists a sequence $(a_k)_{k\geq 0}\in \mathbb R^{\mathbb N}$ such that for any $N\in \mathbb N$, it holds in the limit $h\to 0$: \begin{equation}\label{eq.asymptoO(h)}
r(h)=\sum_{k=0}^Na_kh^{\alpha k}+O\big (h^{\alpha(N+1)}\big ).
\end{equation}

\noindent
According to~\eqref{eq.t1}, when the function $F$ belongs to $C^{\infty}(\pa \Omega,\mathbb R)$ and $x\in \mathcal A(\ft C_{\ft{max}})$, one has  in the limit $h\to 0$:
$$
\mathbb E_{x} \left [ F\left (X_{\tau_{\Omega}} \right )\right] =
\sum_{i=1}^{\ft k_{1}^{\pa \ft C_{\ft{max}} } } a_i F(z_i) + O(h^{\frac14})= \frac{\sum \limits_{i=1}^{\ft k_{1}^{\pa \ft C_{\ft{max}} } }   {\int_{\Sigma_i} F \partial_nf \, e^{-\frac{2}{h}f }} }{\sum \limits_{i=1}^{\ft k_{1}^{\pa \ft C_{\ft{max}} } }   {\int_{\Sigma_i}  \partial_nf \, e^{-\frac{2}{h} f }  } }  + o_h(1), 
$$ 
where the order in $h$ of the remainder term $o_h(1)$  depends on the support of $F$ and on whether or not  the assumption~\eqref{eq.hip4} is satisfied.  
This is  reminiscent  of  previous  results  obtained
in~\cite{kamin1979elliptic,Kam,Per,day1984a,day1987r}. 

Theorem~\ref{thm.main} implies that when $X_0 = x\in \mathcal A(\ft
C_{\ft{max}})$,    the    law of~$X_{\tau_{\Omega}}$ concentrates on
$\{z_{1},\dots, z_{\ft k_1^{\ft C_{\ft{max}} }}\}=\pa \Omega \cap \pa
\ft C_{\ft{max}}$ in the limit $h\to 0$, with explicit formulas for   the probabilities  to exit through each of the $z_i$'s.  
Moreover, the probability to exit through a global minimum $z$ of
$f|_{\partial \Omega}$ which satisfies $\partial_nf(z)<0$ is
exponentially small in the limit $h\to 0$ (see~\eqref{eq.t1}) and when
assuming~\eqref{eq.hip4}, the probability to exit through $ z_{\ft
  k_1^{\ft C_{\ft{max}}}+1},\dots, z_{\ft k_1^{\pa \Omega}} $ is also
exponentially small even though  all these points belong to
$\argmin_{\pa\Omega}f$. Theorem~\ref{thm.main} is thus a
generalization of~\cite[Theorem 1]{DLLN-saddle1} to other
initial conditions than the quasi-stationary distribution $\nu_h$.

\begin{remark}\label{re.discussionHP} Assumptions~\eqref{H-M} and~\eqref{eq.hip1} ensure that $\ft C_{\ft{max}}$ appearing in Theorem~\ref{thm.main} 
% and $a_i$ ($i=1,\ldots, \ft k_{1}^{\pa \ft C_{\ft{max}} }$, see~\eqref{ai}) appearing in Theorem~\ref{thm.main}  are
is well  defined. Concerning the assumption~\eqref{eq.hip2}, there exist functions $f$    satisfying~\eqref{H-M} and~\eqref{eq.hip1} but not~\eqref{eq.hip2} such that when $X_0=x\in \ft C_{\ft{max}}$, the law of $X_{\tau_\Omega}$   concentrates on points which do not belong to $\argmin_{\pa \Omega} f$ (see indeed~\cite[Figure~4]{DLLN-saddle0}).  The same holds for the assumption~\eqref{eq.hip3}:  there exist functions $f$ which satisfy~\eqref{H-M},~\eqref{eq.hip1}, and~\eqref{eq.hip2} but not~\eqref{eq.hip3}  such that when $X_0=x\in \ft C_{\ft{max}}$,  the law of $X_{\tau_\Omega}$   concentrates on points which do not belong to $\argmin_{\pa \Omega} f$  (see Figure~\ref{fig:exemple1} and~\cite[Figure~5]{DLLN-saddle0}). Thus,~\eqref{eq.hip2} and~\eqref{eq.hip3} are necessary to ensure  that the law of $X_{\tau_\Omega}$   concentrates on points belonging to $\argmin_{\pa \Omega} f$. 
Finally, assumption~\eqref{eq.hip4}   is necessary to get  item~3 in
  Theorem~\ref{thm.main}. Indeed, in~\cite[Section~1.4]{DLLN-saddle0}, we provide examples of functions $f$ satisfying~\eqref{H-M},~\eqref{eq.hip1},~\eqref{eq.hip2},
  and~\eqref{eq.hip3} but not~\eqref{eq.hip4}, such that  for all
   sufficiently small neighborhood  $ \Sigma_z$ of $z\in \argmin_{\pa \Omega} f\setminus \pa \ft C_{\ft{max}}$ in $\pa
  \Omega$,  $\mathbb P  \left [ X_{\tau_{\Omega}} \in
       \Sigma_z\right]=C\sqrt{h}\, (1+o(1))$ in the limit $h\to 0$, for some $C>0$
    independent of $h$ and when $X_0=x\in \ft C_{\ft{max}}$.   
\end{remark}

\begin{remark}\label{re.txfini}
When $x\in \Omega$ is such  that $t_x<+\infty$, it is a simple consequence of the large deviations estimate~\eqref{eq.WFr-est} below that,  in the limit $h\to 0$,  the process~\eqref{eq.langevin} almost surely exits~$\Omega$ through any neighborhood of $\varphi_{t_x}(x)$ in $ \pa \Omega$. 
\end{remark}

It is also possible to describe the exit point distribution when
$X_0=x\in \mathcal A(\ft C)$ and $\ft C\in \mathcal C$
is not necessarily $\ft C_{\ft{max}}$ (we recall that $\mathcal C$ is
defined in~\eqref{mathcalC-def}). This is the objective of
Theorem~\ref{thm.2}, whose proof uses Theorem~\ref{thm.main} applied to
a suitable subdomain of $\Omega$ containing~$\ft C$.

\begin{theorem}\label{thm.2}
Let us assume that  \eqref{H-M} holds. Let $\ft C\in \mathcal C$. Let us assume that 
\begin{equation}\label{eq.cc1}
\pa \ft C\cap \pa \Omega\neq \emptyset \  \text{ and } \  \vert \nabla f\vert \neq 0 \text{ on } \pa \ft C.
\end{equation}
%\begin{equation}\label{eq.cc1}
%\pa \ft C\cap \pa \Omega\neq \emptyset \  \text{ and } \  \overline{\ft C} \text{ is a connected component of $\{f\le \lambda\}$}.
%\end{equation}
 Recall that  $\pa \ft C\cap \pa \Omega \subset \ft U_1^{\pa \Omega}$
 (see~\eqref{eq.mathcalU1_bis} for a definition of $\ft U_1^{\pa \Omega}$). Let $F\in L^\infty(\pa \Omega,\mathbb R)$. 
For all $z \in \pa \ft C\cap \pa \Omega$, let  $\Sigma_{z}$ be   an open subset of~$\pa \Omega$ such that $ z\in \Sigma_{z}$. 
Let $K$ be a compact subset of $\Omega$ such that $K\subset \mathcal A(\ft C)$. Then:
\begin{enumerate}
\item There exists $c>0$ such that for  $h$ small enough,
$$\sup_{x\in K}\mathbb E_x\Big [\big (F \, \mbf 1_{ \pa \Omega \setminus  \bigcup_{z\in \pa \ft C\cap \pa \Omega} \Sigma_{z}}\big )\left (X_{\tau_{\Omega}} \right )\Big ]\le e^{-\frac ch}.$$
\end{enumerate}
Assume moreover that   the sets $(\Sigma_{z})_{z\in \pa \ft C\cap \pa \Omega}$  are pairwise disjoint (i.e. such that $\Sigma_{z}\cap \Sigma_{z'}=\emptyset$ whenever $z\neq z'$).  Let $z\in \pa \ft C\cap \pa \Omega$. 
\begin{enumerate}
\addtocounter{enumi}{1}
\item
If  $F$ is  $C^\infty$ in a neighborhood of $z$,   it holds for all $x\in K$, 
$$\mathbb E_{x}[(F\, \mbf 1_{ \Sigma_z})\left (X_{\tau_{\Omega}} \right )]= F(z)\frac{  \partial_nf(z)      }{  \sqrt{ {\rm det \ Hess } f|_{\partial \Omega}   (z) }  } \left (\sum \limits_{y\in\pa \ft C \cap \pa \Omega  } \frac{  \partial_nf(y)      }{  \sqrt{ {\rm det \ Hess } f|_{\partial \Omega}   (y) }  }\right)^{-1}+O(h),$$
in the limit $h \to 0$ and uniformly in~$x \in K$.
\end{enumerate}
\end{theorem}

\noindent
Theorem~\ref{thm.2} implies that when $ \ft C \in \mathcal C$
satisfies~\eqref{eq.cc1} (for instance, this is the case for~$\ft C_3$
on Figure~\ref{fig:shema_nota}),    the law of $X_{\tau_\Omega}$ when
$X_0=x\in \mathcal A(\ft C)$ concentrates when $h\to 0$ on  $\pa \ft C
\cap \pa \Omega$.

 Theorems~\ref{thm.main} and~\ref{thm.2} are actually special cases of a  more general result which will be stated and illustrated in Section~\ref{sec.gene_fin}  (see indeed Theorem~\ref{thm.4}). The proof of Theorem~\ref{thm.4} is a simple extension of the proof of Theorem~\ref{thm.2}. For pedagogical purposes, we prefer to first present Theorem~\ref{thm.main} and~\ref{thm.2}, before stating the more general and abstract result of Theorem~\ref{thm.4}.

 %because we do not assume that $\pa \ft C_1\cap \pa \Omega\subset \argmin_{\pa \Omega}f$.
%
%\comment{ici... makes sense if we add the generalisation at the end?}When ${\ft  C}\in \mathcal C$ does not satisfy~\eqref{eq.cc1},  it is
%much harder to   exhibit explicit assumptions  on $\ft C$ to give the
%most probable places of exit from $\Omega$  of the
%process~\eqref{eq.langevin} when $h\to 0$  (we refer to~\cite[Appendix
%B]{DLLN-saddle0} for a discussion on one-dimensional examples). 
% % \comment{ non on fait un peu mieux que ~\eqref{eq.cc1} dans le Thm~\ref{thm.4} a la fin}

\subsection{Organization of the paper}\label{sec:orga}

The rest of this paper is dedicated to the proofs of
Theorem~\ref{thm.main}, in Section~\ref{sec-xo-nuh2} and of
Theorems~\ref{thm.2} and~\ref{thm.4}, in Section~\ref{sec:proof_th2}. The proof of
Theorem~\ref{thm.main} heavily relies on results
from~\cite{DLLN-saddle1} (recalled in Theorem~\ref{thm.main-bis} below), together with a so-called leveling result on
$x \mapsto \mathbb E_x[F(X_{\tau_\Omega})]$. The proof of
Theorem~\ref{thm.2} uses Theorem~\ref{thm.main} applied to a domain
$\Omega_{\ft C}$ which contains $\ft C$. The construction of this
domain uses tools from differential topology related to the genericity
of Morse functions. Finally, Section~\ref{sec:conc} gives conclusions and
perspectives.

\section{Proof of Theorem~\ref{thm.main}} 
\label{sec-xo-nuh2}

After recalling some results from~\cite{DLLN-saddle1} in
Section~\ref{sec:prev}, we prove a so-called leveling result (as
initially introduced in~\cite{Day4})  in~$\ft C_{\ft{max}}$ for
$x\mapsto \mathbb E_x[F(X_{\tau_\Omega})]$, in
Section~\ref{sec.level}. Then, combining the results of these two
sections, one proves Theorem~\ref{thm.main} in
Section~\ref{sec:proofTh1} for a smooth function~$F$ and finally for a
measurable bounded $F$ in Section~\ref{sec:end_proof}.
 
\subsection{Previous results on the principal eigenfunction of $L^D_{f,h}$ and on $\nu_h$}\label{sec:prev}
Let us recall the following result from~\cite[Theorem 4]{DLLN-saddle1}
on the spectral gap  of $L^D_{f,h}$.  We recall that $\lambda_h$ is the
principal eigenvalue of $-L^D_{f,h}$ (see Section~\ref{sec:exit}).
\begin{proposition}
\label{pr.lambdah}
Assume that the assumptions \eqref{H-M}  and \eqref{eq.hip1}  are satisfied. 
Let us denote by  $f_{\ft{max}}$ the value of $f$ on $ \pa \ft
C_{\ft{max}}$ and $x_{\ft{max}}$ a minimum point of $f$ in ${\overline{\ft C_{\ft{max}}}}$:
\begin{equation}\label{xzmax}
 f \equiv  f_{\ft{max}} \text{ on }  \pa \ft C_{\ft{max}} \ \text{ and } \ x_{\ft{max}}\in \argmin_{\overline{\ft C_{\ft{max}}}}f.
\end{equation}  
 Notice that  $ f_{\ft{max}}=\max_{\overline{\ft C_{\ft{max}}} } f$
 and $x_{\ft{max}}\in \ft C_{\ft{max}}$. Then, there exists 
$C>1$ and $\gamma\in \mathbb R$,  such that, for $h$ small enough,
$$C^{-1}\,  h^{\gamma}\,  e^{-\frac 2h  (f_{\ft{max}}- f(x_{\ft{max}}))}  \le \lambda_h\le C\,  h^{\gamma}\,  e^{-\frac 2h  (f_{\ft{max}}- f(x_{\ft{max}}))}.$$
Moreover, there exists $c>0$ such that  for $h$ small enough, $\min \sigma(L^D_{f,h})\setminus\{\lambda_h\}\ge e^{\frac ch} \lambda_h$.  
%
%
%
% Let us moreover assume that 
%$$\min_{\overline{\ft C_{\ft{max}}}}f= \min_{\overline \Omega}f.$$
% Let~$u_h$ be  the eigenfunction associated with the principal eigenvalue $\lambda_h$ of $-L_{f,h}^{D,(0)}$ (see~\eqref{eq.lh})  which satisfies~\eqref{eq.norma}. Let $\ft O$ be an open subset of $\Omega$. 
%On the one hand, if 
% $$\ft O\cap \argmin_{\ft C_{\ft{max}}}f\neq \emptyset,$$
% one has in the limit $h\to 0$:
% \begin{equation}\label{concentration01}
%\int  _{\ft O} u_h \ e^{- \frac{2}{h} f } =h^{\frac{d}{4} }\, \pi^{\frac{d}{4} }\frac{   \sum_{x\in\ft O\cap \argmin_{\ft C_{\ft{max}}}f}  \big( {\rm det \ Hess } f   (x)   \big)^{-\frac12}  }{ \Big(\sum_{x\in \argmin_{\ft C_{\ft{max}}}f}  \big( {\rm det \ Hess } f   (x)   \big)^{-\frac12}\Big)^{\frac12} }\ e^{-\frac{1}{h} \min_{\overline \Omega}f} \  \big(1+O(h) \big).
% \end{equation}
%On the other hand, if
%$$\overline{\ft O}\cap \argmin_{\ft C_{\ft{max}}}f= \emptyset,$$
% then, there exists $c>0$ such that  when $h \to 0$:
%\begin{equation}\label{eq.concentration02}
%\int  _{\ft O} u_h \ e^{- \frac{2}{h} f } =O\big ( e^{-\frac{1}{h}( \min_{\overline \Omega}f+c)} \big).
%  \end{equation}
\end{proposition}
%When  \eqref{H-M}  and \eqref{eq.hip1}  are satisfied and  when  $\min_{\overline{\ft C_{\ft{max}}}}f= \min_{\overline \Omega}f$ holds,  Proposition~\ref{pr.masse} implies that     when $h\to 0$, $u_h\,e^{-\frac 2h f}$ concentrates in the   $L^1$-norm      on the global minima of~$f$ in~$\ft C_{\ft{max}}$.  Notice that when $\ft O=\Omega$ in Proposition~\ref{pr.masse}, one has from~\eqref{concentration01},  when $h\to 0$:
%\begin{equation}\label{eq.concentration1}
%\int  _{\Omega} u_h \ e^{- \frac{2}{h} f } = 
%h^{\frac{d}{4} }\, \pi^{\frac{d}{4} }e^{-\frac{1}{h} \min_{\overline \Omega}f} 
%  \Big(\sum_{x\in\argmin_{ \ft C_{\ft{max}}}f}  \big( {\rm det \ Hess } f   (x)   \big)^{-\frac12}\Big)^{\frac12} \ \big(1+O(h) \big).
%  \end{equation}
 \noindent
We now give a series of three results which are consequences of the
previous proposition. These results  can all be found
in~\cite{DLLN-saddle1}. We provide the proofs for these results since
they are short and they are opportunities to introduce  some  notation
which will be needed later on. 

A direct corollary of  Proposition~\ref{pr.lambdah} is the following. 
\begin{corollary}\label{co.pr}
Let us assume that the assumptions~\eqref{H-M} and \eqref{eq.hip1} are satisfied. Then, there exists $\beta_0>0$ such that for all $\beta\in (0,\beta_0)$, there exists $h_0>0$ such that  for all $h\in (0,h_0)$, the orthogonal projector in $L^2_w(\Omega)$ 
\label{page.tildephio}
$$ \pi_h:=\pi_{\big [0,e^{-\frac{2}{h}( f_{\ft{max}}- f(x_{\ft{max}})-\beta)}\big )}(L^{D}_{f,h}) \text{ has rank } 1.$$
\end{corollary}
\noindent
 Here and in the following,  for a Borel set $E\subset \mathbb
 R$,~$\pi_{E}(L^{D}_{f,h})$ is the spectral projector of $L^{D}_{f,h}$
on $E$.\\
For $\alpha>0$, one defines 
\begin{equation}\label{eq.c1a}
\ft C_{\ft{max}}(\alpha)= \ft C_{\ft{max}}\cap \big\{f<f_{\ft{max}}  -\alpha\big\}.  
\end{equation} 
For $\alpha>0$, let $\chi\in C^\infty(\Omega)$ be such that 
$$ \chi=0 \text{ on } \overline{\Omega}\setminus  \ft C_{\ft{max}}(\alpha) \text{ and } \chi=1 \text{ on }  \ft C_{\ft{max}}(2\alpha).$$
Using Corollary~\ref{co.pr}, we obtain a good approximation of $u_h$
using the function~$\chi$. 
%\comment{Tu n'as bien besoin que d'estim\'ee
 % $L^2$ ici?}
\begin{proposition}
\label{pr.appro-lambda}
Let us assume that the assumptions~\eqref{H-M} and \eqref{eq.hip1} are satisfied.  
Let us define 
$$\psi := \frac{\chi}{\Vert \chi\Vert_{L^2_w(\Omega)}}.$$
Then, for all $\alpha>0$ small enough, there exists $c>0$ such that in the limit $h\to 0$:
$$u_h=  \psi \ +\ O(e^{-\frac ch}) \ \text{ in } L^2_w(\Omega).
$$
\end{proposition}

\begin{proof}
Let us recall that for all $u\in H^1_{w,0}(\Omega)$ and $b>0$,
$$\left\Vert \pi_{[b,+\infty)} (L^D_{f,h}) \, u\right\Vert^2 \leq \frac{\frac h2\int_{\Omega} \vert \nabla u\vert ^2e^{-\frac 2h f} }{b}.$$
Thus, one has:
\begin{equation}\label{eq.nablav0}
  \big \|    (1-   \pi_h )       \psi  \big\|_{L^2_w(\Omega)}^2  \le  e^{\frac{2}{h}(f_{\ft{max}}- f(x_{\ft{max}})-\beta )} \, \frac h2\big\|   \nabla        \psi   \big\|_{L^2_w(\Omega)}^2.
  \end{equation} 
Moreover, by definition of the function $\psi$, it holds,
\begin{equation}\label{eq.nablav1}
\big\|   \nabla        \psi   \big\|_{L^2_w(\Omega)}^2=\frac{  \int_\Omega   \vert \nabla \chi  \vert ^2 e^{-\frac 2h f} }{  \int_\Omega  \chi^2 e^{-\frac 2h f}}.
\end{equation} 
Since $\chi=0 \text{ on } \overline{\Omega} \setminus \ft C_{\ft{max}}(\alpha)$, it holds,  $ \int_\Omega  \chi^2 e^{-\frac 2h f}=  \int_{\ft C_{\ft{max}}(\alpha)}   \chi^2 e^{-\frac 2h f}$. 
Since $f$ has finite number of global minima in $\overline{ \ft C_{\ft{max}}}$ which are all included in $ \ft C_{\ft{max}}$, one deduces that for all  $\alpha>0$ small enough, 
\begin{equation}\label{eq.cmax-arg}
\argmin_{\overline{ \ft C_{\ft{max}}}}f \subset \ft C_{\ft{max}}(2\alpha).
 \end{equation}
Consequently, using a Laplace's method together with the fact that $\chi=1 \text{ on }  \ft C_{\ft{max}}(2\alpha)$, one has in the limit $h\to 0$, 
\begin{equation}\label{eq.laplace}
 \int_\Omega  \chi^2 e^{-\frac 2h f}= h^{\frac{d}{2} }\, \pi^{\frac{d}{2} }e^{-\frac{2}{h} f(x_{\ft{max}})} 
 \sum_{x\in\argmin_{ \ft C_{\ft{max}}}f}  \big( {\rm det \ Hess } f   (x)   \big)^{-\frac12}\ \big(1+O(h) \big).
 \end{equation}
 Since the support of $\nabla \chi$ is included in $\overline{\ft
   C_{\ft{max}}(\alpha)}\setminus \ft C_{\ft{max}}(2\alpha)$ (thus
 $f\ge f_{\ft{max}}-2\alpha$ on the support of
 $\nabla \chi$ 
 %\comment{Relire. Tu avais ecrit $f\ge f_{\ft{max}}-
  % f(x_{\ft{max}})-2\alpha$ ??? J'ai propage ci-dessous},
    it holds for some $C>0$:
 $$ \int_\Omega   \vert \nabla \chi  \vert ^2 e^{-\frac 2h f}\le C\, h^{1-\frac d2} \, e^{-\frac 2h (f_{\ft{max}}-2\alpha)}. $$
Plugging these estimates in~\eqref{eq.nablav1} and using~\eqref{eq.nablav0}, one finally deduces that for  
$$  \|    (1-  \pi_h )        \psi\|_{L^2_w(\Omega)}^2\le C\, e^{-\frac 2h(\beta-2\alpha)}.$$
Choosing $\alpha<\beta/4$, this implies that for $h$ small enough, $
\big \|    (1-   \pi_h )        \psi\big\|_{L^2_w(\Omega)}\le
e^{-\frac{\beta}{ h}}$. Therefore,   using Corollary~\ref{co.pr} and the fact that
 the functions $ u_{h}$ and $\psi$ are non negative,
$$
u_h= \frac{\   \pi_h \,     \psi}{\|     \pi_h \,    \psi \|_{L^2_w(\Omega)}}=\     \psi+O ( e^{-\frac{c}{h}})\ \text{  in  } L^2_w(\Omega)
$$
for some positive $c$. This concludes the proof of  Proposition~\ref{pr.appro-lambda}. 
\end{proof}

\noindent
We end this section, by the following consequence of Proposition~\ref{pr.appro-lambda}. 
\begin{corollary}\label{co.moyenne}
Let us assume that the assumptions~\eqref{H-M} and \eqref{eq.hip1} are satisfied. Let us moreover assume that 
$$\min_{\overline{\ft C_{\ft{max}}}}f= \min_{\overline \Omega}f.$$
Then, when $h\to 0$:
$$
\int  _{\Omega} u_h \ e^{- \frac{2}{h} f } = 
h^{\frac{d}{4} }\, \pi^{\frac{d}{4} }e^{-\frac{1}{h} \min_{\overline \Omega}f} 
  \Big(\sum_{x\in\argmin_{ \ft C_{\ft{max}}}f}  \big( {\rm det \ Hess } f   (x)   \big)^{-\frac12}\Big)^{\frac12} \ \big(1+O(h) \big).
$$
Moreover, if $\ft O$ is an open subset of $\Omega$ such that $\overline{\ft O}\cap \argmin_{ \ft C_{\ft{max}}}f=\emptyset$, then, there exists $c>0$ such that for $h$ small enough, $$\int  _{\ft O} u_h \ e^{- \frac{2}{h} f }=O\big(e^{-\frac{1}{h} (\min_{\overline \Omega}f+c)} \big).$$
\end{corollary}
%%%%
\begin{proof}
Let $\ft O$ be an open subset of $\Omega$. 
 Using Proposition~\ref{pr.appro-lambda} and thanks to the Cauchy-Schwarz inequality, one obtains  in the limit $h\to 0$:
\begin{align}
\nonumber
\int  _{\ft O}u_h \ e^{-\frac{2}{h} f} &= \int  _{\ft O}  \     \psi  \,e^{-\frac{2}{h} f} + O ( e^{-\frac{c}{h}})  \sqrt{  \int  _{\ft O}  e^{-\frac{2}{h} f} }= \int  _{\ft O}  \     \psi  \,e^{-\frac{2}{h} f} + O \Big ( e^{-\frac{1}{h}\,  (\min_{\overline{\Omega}} f + c  )}\Big).
\end{align}
Then, the first statement in Corollary~\ref{co.moyenne} follows choosing $\ft O=\Omega$ in the previous equality and  using~\eqref{eq.laplace} (notice that the same estimate holds replacing $\chi^2$ by  $\chi$ in~\eqref{eq.laplace}), the definition of $\psi$ (see~Proposition~\ref{pr.appro-lambda}) together with the fact that by assumption  $f(x_{\ft{max}})=\min_{\overline{\ft C_{\ft{max}}}}f= \min_{\overline \Omega}f$. The second statement in Corollary~\ref{co.moyenne} follows using~\eqref{eq.laplace}  and the fact that when $\overline{\ft O}\cap \argmin_{ \ft C_{\ft{max}}}f=\emptyset$, there exists $c>0$ such that for $h$ small enough,  $ \int  _{\ft O}  \     \chi  \,e^{-\frac{2}{h} f}= O \big ( e^{-\frac{2}{h}\, (\min_{\overline{\Omega}} f + c )}\big)$. 
\end{proof}

\noindent
We end this section by recalling the results of~\cite[Theorem 1]{DLLN-saddle1} on the law of $X_{\tau_{\Omega}}$ when $X_0$ is initially distributed according to the quasi-stationary distribution $\nu_h$ of the process~\eqref{eq.langevin}. 

\begin{theorem}\label{thm.main-bis}
Let us assume  that the assumptions \eqref{H-M},~\eqref{eq.hip1},~\eqref{eq.hip2}, and~\eqref{eq.hip3} are satisfied. 
  Let  $F\in L^{\infty}(\partial \Omega,\mathbb R)$ and $(\Sigma_{i})_{i\in\{1,\dots,\ft k_{1}^{\pa \Omega}\}}$ be a family  of pairwise disjoint open subsets of~$\pa \Omega$ such that 
  $$\text{for all } i\in\big \{1,\dots,\ft k_{1}^{\pa \Omega}\big \},  \ z_{i}\in \Sigma_{i},$$ 
  where we recall that $\big \{z_1,\dots,z_{\ft k_{1}^{\pa\Omega}}\big \}= \ft U_1^{\pa \Omega}\cap \argmin_{\pa \Omega} f$ (see~\eqref{eq.z11}).     
    Then:
  \begin{enumerate} 
  \item 
There exists $c>0$ such that when  $h\to 0$:
\begin{equation} \label{eq.t1-bis}
\mathbb E_{\nu_h} \left [ F\left (X_{\tau_{\Omega}} \right )\right]=\sum \limits_{i=1}^{\ft k_1^{\pa \Omega}}
\mathbb E_{\nu_h} \left [ (\mathbf{1}_{\Sigma_{i}}F)\left (X_{\tau_{\Omega}} \right )\right]  +O\big (e^{-\frac ch}\big )
 \end{equation}
 and
 \begin{equation} \label{eq.t2-bis}
\sum \limits_{i=\ft k_1^{\pa \ft C_{\ft{max}} }+1}^{\ft k_1^{\pa \Omega}}
\mathbb E_{\nu_h} \left [ (\mathbf{1}_{\Sigma_{i}}F)\left (X_{\tau_{\Omega}} \right )\right]  =O\big (h^{\frac14} \big ),
 \end{equation}
 where we recall that $\big \{z_1,\ldots,z_{\ft k_1^{\pa \ft C_{\ft{max}} }}\big \}=\pa \ft C_{\ft{max}}\cap \pa \Omega$ (see~\eqref{eq.k1-paCmax}). 
\item When for some $i\in\big \{1,\dots,\ft k_{1}^{\pa \ft C_{\ft{max}} }\big \}$ the function  $F$ is $C^{\infty}$ in a neighborhood  of $z_{i}$, one has when $h\to 0$:
\begin{equation} \label{eq.t3-bis}
 \mathbb E_{\nu_h} \left [ (\mathbf{1}_{\Sigma_{i}}F)\left (X_{\tau_{\Omega}} \right )\right]=F(z_i)\,a_{i} +O(h^{\frac14}) \ \text{ where $a_i$ is defined by~\eqref{ai}.}
\end{equation} 
  \item 
When~\eqref{eq.hip4} is satisfied 
 the remainder term  $O( h^{\frac14})$ in \eqref{eq.t2}
 is   of the order  $O\big (e^{-\frac{c}{h}}\big )$ for some $c>0$  
and
the remainder term  $O\big (h^{\frac14}\big )$  in \eqref{eq.t3} is of the order $O(h)$  and admits a full asymptotic expansion in~$h$. 
\end{enumerate} 
 
\end{theorem}

\subsection{Leveling results}\label{sec.level}

To go from Theorem~\ref{thm.main-bis} to Theorem~\ref{thm.main}, the basic idea is to write
$$\mathbb E_{\nu_h}[F(X_{\tau_\Omega})]=\int_{\Omega} \mathbb
E_x[F(X_{\tau_\Omega})] \nu_h(dx)
$$
and to use the fact that the function $x\mapsto \mathbb
E_x[F(X_{\tau_\Omega})]$ is ``more and more constant'' as $h \to 0$:
this is called a {\em leveling property}.

\begin{definition}\label{level0}
Let $K$ be a compact subset of $\Omega$ and $F\in C^{\infty}(\partial \Omega,\mathbb R)$. We say that $x\mapsto \mathbb E_x[F(X_{\tau_\Omega})]$ satisfies a leveling property on $K$ if  
\begin{equation}  \label{eq.leveling}
\lim_{h\to 0}\big \vert \,\mathbb E_{x} \left [ F\left (X_{\tau_{\Omega}} \right )\right]-\mathbb E_{y} \left [ F\left (X_{\tau_{\Omega}} \right )\right] \, \big  \vert=0
\end{equation} 
and this limit holds uniformly with respect to $(x,y)\in K\times K$.
\end{definition}
\noindent
The leveling property~\eqref{eq.leveling} has been widely studied in the literature in various geometrical settings, see for example  \cite{Per,Kam,DeFr,Day4,FrWe,Eiz,kamin1979elliptic}. 
We prove the following proposition which is a leveling property in our framework. 

\begin{proposition} \label{level}
Let us assume that the assumption \eqref{H-M} holds. Let $\lambda \in \mathbb R$ and $\ft C$ be a connected component of $\{f<\lambda\}$ such that $\ft C\subset \Omega$. Then, for any path-connected compact set $K\subset\ft C$ and for any $F\in  C^{\infty}(\partial \Omega,\mathbb R)$, there exist $c>0$ and $M>0$, such that for all $(x,y)\in K\times K$,
\begin{equation}\label{eq.LP}
\big \vert \,\mathbb E_{x} \left [ F\left (X_{\tau_{\Omega}} \right )\right]-\mathbb E_{y} \left [ F\left (X_{\tau_{\Omega}} \right )\right] \, \big \vert \leq M e^{-\frac{c}{h}}.
\end{equation}
\end{proposition}

\begin{proof}
 The proof is inspired from techniques used in \cite{DeFr}.  The proof of Proposition~\ref{level} is divided into two steps. 
%Before starting the proof, let us recall the following elliptic regularity estimate which will be used several times in the proof of Proposition~\ref{level}. For~$g\in L^2(\Omega)$, let $u\in H^1(\Omega)$ be the weak solution of the elliptic boundary value problem
%\begin{equation*}
%% \label{eq.elliptic-F}
%\left\{
%\begin{aligned}
%  \Delta u &=  g    \ {\rm on \ }  \Omega  \\ 
%u &= 0  \ {\rm on \ } \partial \Omega .\\
%\end{aligned}
%\right.
%\end{equation*}
%Let $m\geq 0$ and assume that $g\in H^m(\Omega)$,  then $u\in H^{m+2}(\Omega)$ and there exists $C>0$ which depends only on $\Omega$ and $m$ such that 
%\begin{equation} \label{est1-ell}
%\Vert u\Vert_{H^{m+2}(\Omega)}\leq C  \Vert g\Vert_{H^m(\Omega)}.\end{equation} 
%The proof of~\eqref{est1-ell} can be found for instance in~\cite[Theorem 5, Section 6.3]{Eva}. 
In the following $C>0$ is a constant which can change from one occurrence to another and which does not depend on $h$.

\medskip
 \noindent
 \textbf{Step 1.}
 Let $F\in  C^{\infty}(\partial \Omega, \mathbb R)$. Let us denote by~$v_h\in H^1(\Omega)$ the unique weak solution to the elliptic boundary value problem
\begin{equation} \label{vh}
\left\{
\begin{aligned}
 \frac{h}{2}  \ \Delta v_h -\nabla f \cdot \nabla v_h  &=  0    \ {\rm on \ }  \Omega  \\ 
v_h&= F \ {\rm on \ } \partial \Omega .\\
\end{aligned}
\right.
\end{equation}
Then, we will prove in this step that~$v_h$ belongs to $ C^{\infty}(\overline \Omega, \mathbb R)$ and  that  for all $k\in \mathbb N$, there exist $C>0$,~$n \in \mathbb N$ and $h_0>0$ such that for all $h\in (0,h_0)$, it holds
\begin{equation} \label{est1}
\Vert v_h\Vert_{H^{k+2}(\Omega)}\leq \frac{C}{h^n} \left ( \Vert \nabla v_h\Vert_{L^2(\Omega)}+1\right).
\end{equation} 
Here $L^2(\Omega)= \big \{u:\Omega \to \mathbb R,  \int_\Omega u^2   < \infty \big \}$, and, for $k\ge 1$, 
$$H^k(\Omega)= \big \{u:\Omega \to \mathbb R,  \forall \alpha\in \mathbb N^k \text{ such that } \vert \alpha\vert \le k, \  \int_\Omega \vert \pa^{\alpha}u\vert ^2 < \infty \big \}.$$ 
Moreover, the Dynkin's formula implies that\label{page.vh}
\begin{equation} \label{eq.vh-dynkin}
\forall x\in \overline \Omega, \ v_h(x)=\mathbb E_{x} \left [F\left ( X_{\tau_{\Omega}} \right)\right].
\end{equation}
Let us prove that $v_h$ belongs to $ C^{\infty}(\overline \Omega, \mathbb R)$ and~\eqref{est1}. Since $F$ is $C^\infty$, for all $k\geq 1$, there exists $\widetilde F\in H^k(\Omega)$ such that $\widetilde F=F$ on $\pa \Omega$ and $$\Vert \widetilde F\Vert_{H^k}\le C\Vert F\Vert_{H^{k-\frac 12}(\pa \Omega)}.$$
From~\eqref{vh}, the function $\widetilde v_h=v_h-\widetilde F\in H^1(\Omega)$   is the weak solution to
\begin{equation} \label{est1-2}
\left\{
\begin{aligned}
  \Delta \widetilde v_h &=  \frac{2}{h}  \nabla f \cdot \nabla v_h-\Delta \widetilde F      \ {\rm on \ }  \Omega  \\ 
\widetilde v_h&= 0 \ {\rm on \ } \partial \Omega .
\end{aligned}
\right.
\end{equation}
Thus, using~\cite[Theorem 5, Section 6.3]{Eva},~$ \widetilde v_h\in H^2(\Omega)$ (and thus $v_h\in H^2(\Omega)$)  and there exist $C>0$ and $h_0>0$ such that for all $h\in (0,h_0)$
\begin{align}  
\nonumber
  \Vert \widetilde v_h\Vert_{H^2(\Omega)}&\leq C \left (  \frac 1h \Vert \nabla v_h \Vert_{L^2(\Omega)}  +\Vert  \widetilde F \Vert_{H^2(\Omega)} \right) \leq \frac Ch \Big ( \Vert \nabla v_h \Vert_{L^2(\Omega)}+\Vert  F \Vert_{H^{\frac 32}(\pa \Omega)}\Big),
\end{align} 
and thus
\begin{equation} \label{est-H2}
\Vert v_h\Vert_{H^2(\Omega)} \le \frac Ch \left (  \Vert \nabla v_h \Vert_{L^2(\Omega)}+ 1\right).
\end{equation} 
This proves~\eqref{est1} for $k=0$. The inequality~\eqref{est1} is
then obtained by a bootstrap argument, by induction on $k$. This implies by Sobolev embeddings that  $v_h$ belongs to $ C^{\infty}(\overline \Omega, \mathbb R)$. 
%Since $\nabla v_h\in H^1(\Omega)$, using~\eqref{est1-2} and~\cite[Theorem 5, Section 6.3]{Eva}, the function $ \widetilde v_h\in H^3(\Omega)$ (and thus $v_h\in H^3(\Omega)$)  and there exist $C>0$ and $h_0>0$ such that for all $h\in (0,h_0)$
%\begin{align}  
%\nonumber
%  \Vert \widetilde v_h\Vert_{H^3(\Omega)}&\leq C \left (  \frac 1h \Vert \nabla v_h \Vert_{H^1(\Omega)}  +\Vert  \widetilde F \Vert_{H^3(\Omega)} \right) \leq \frac Ch \Big (  \Vert   v_h \Vert_{H^2(\Omega)}+\Vert  F \Vert_{H^{\frac 52}(\pa \Omega)}\Big),
%\end{align} 
%and thus, using~\eqref{est-H2}, it holds 
%$$
%\Vert v_h\Vert_{H^3(\Omega)} \le \frac{C}{h^2} \left (  \Vert \nabla v_h \Vert_{L^2(\Omega)}+ 1\right).
%$$

 Let us  now prove that there exist $\alpha>0$ and $C>0$ such that 
\begin{equation} \label{est2}
\Vert v_h\Vert_{L^{\infty}(\Omega)}+ \Vert \nabla v_h\Vert_{L^{\infty}(\Omega)} \leq C h^{-\alpha}.
\end{equation}
  Notice that from~\eqref{eq.vh-dynkin}, one has that for all
  $h>0$,~$\Vert v_h\Vert_{L^{\infty}(\Omega)}\leq \Vert
  F\Vert_{L^\infty(\pa \Omega)}$. Using this bound,~\eqref{vh} and~\eqref{est-H2}, there exists $C>0$ such that for any $\ve>0$ and $\ve'>0$,
\begin{align*} 
h \int  _{\Omega} \vert  \nabla v_h \vert^2   &\leq C \left (h \int_{\partial \Omega} \vert F\, \partial_n v_h \vert \, d\sigma+ \Vert F\Vert_{L^\infty(\pa \Omega)}\int  _{\Omega} \vert \nabla f \cdot \nabla v_h  \vert   \right)\\
&\leq C \left (h \frac{\Vert F\Vert_{L^2(\partial \Omega)}^2}{4\ve} +  \, h  \, \ve\, \Vert v_h\Vert_{H^2(\Omega)}^2 + \frac{\Vert \nabla f\Vert_{L^2( \Omega)}^2}{4\ve'} +   \ve'\, \Vert \nabla v_h\Vert_{L^2(\Omega)}^2  \right)\\
&\leq C  \, h \frac{\Vert F\Vert_{L^2(\partial \Omega)}^2}{4\ve} +  C  \, h  \, \ve\, C_1\, h^{-2}\left ( \Vert \nabla v_h\Vert_{L^2(\Omega)}^2+ 1\right) \\
&\quad + C  \,  \frac{\Vert \nabla f\Vert_{L^2( \Omega)}^2}{4\ve'} +C  \,    \ve'\, \Vert \nabla v_h\Vert_{L^2(\Omega)}^2.
\end{align*}
Choosing $\ve= \frac{h^2}{4(CC_1+1)}$ and $\ve'= \frac{h}{4(C+1)}$ we get 
$$\Vert \nabla v_h\Vert_{L^{2}(\Omega)} \leq  \frac{C}{h}.$$
Therefore, using~\eqref{est1}, one obtains that for all $k\geq 0$, there exist $C>0$,~$n\in \mathbb N$ and $h_0>0$ such that for all $h\in (0,h_0)$
  $$\Vert v_h\Vert_{H^{k}(\Omega)} \leq  \frac{C}{h^n}.$$
 Let $k\geq 0$ such that $k-\frac{d}{2}>1$. Then, one obtains~\eqref{est2} from the continuous Sobolev injection $H^k(\Omega)\subset W^{1,\infty}(\Omega)$. 

\medskip
 \noindent
 \textbf{Step 2.}  
Let us assume that~\eqref{H-M} holds. Let $\lambda \in \mathbb R$ and $\ft C$ be a connected component of $\{f<\lambda\}$ such that $\ft C\subset \Omega$. Let us now define the set $\ft C_r$ by \label{page.cr}
\begin{equation}\label{pr.cr}
\ft C_r=\{f<\lambda-r\}\cap\ft C\subset \Omega
\end{equation}
which is not empty  and $C^{\infty}$  for all $r\in (0,r_1)$, for some $r_1>0$. Indeed, the boundary of $\ft C_r$ is   the set $\{f=\lambda-r\}\cap \ft C$ (since $\ft C\subset \Omega$)  which contains no critical points of~$f$  for $r\in (0,r_1)$, with  $r_1>0$ small enough (since there is  a finite number of critical points under the assumption \eqref{H-M}). 
In this step, we will prove that for all $r_0\in (0,r_1)$ there exists $\alpha_0>0$ such that 
\begin{equation}\label{est3}
\Vert \nabla v_h\Vert_{L^{\infty}(\ft C_{r_0})} \leq  e^{-\frac{\alpha_0}{h}}.
\end{equation}
%To prove Proposition~\ref{level}, we will  prove that for  any compact subset $K$ of $\ft C$, there exists $c_K>0$ such that 
%\begin{equation} \label{eq1}\Vert \nabla   v_h\Vert_{L^{\infty}(K)}\leq C\,e^{-\frac{c_K}{h}}.\end{equation}
Equation~\eqref{est3} implies that for any compact subset $K$ of $\ft C$, there exist $c>0$ and $C>0$  such that 
%Indeed, since $K$ is  path-connected,  the following inequality 
$$\forall (x,y)\in K\times K, \  \vert v_h(x)-v_h(y)\vert \leq C e^{-\frac ch},$$
which will then conclude the proof
of~\eqref{eq.LP}. This follows from the fact that there exists $r_0\in (0,r_1)$ such that  $K\subset \ft C_{r_0}$.\\
Let us now prove~\eqref{est3}. Let $r$ be such that $2^nr=r_0$ where $n\in \mathbb N^*$ will be fixed later (since $r\le r_0$,~$r\in (0,r_1)$). Equation~\eqref{vh} rewrites 
$$
\left\{
\begin{aligned}
 {\rm div}\, \big (e^{-\frac{2}{h}f} \nabla v_h \big) &=  0    \ {\rm on \ }  \Omega  \\ 
v_h&= F \ {\rm on \ } \partial \Omega.\\
\end{aligned}
\right.
$$
 Using~\eqref{est2},  there exist $C>0$  and $\alpha>0$ such that,
$$
\left \vert \, \,  \int_{\ft C_{r/2}} \vert  \nabla v_h \vert^2 \, e^{-\frac{2}{h}f}\,  \right \vert = \left \vert \ \int_{\partial \ft C_{r/2}} \, e^{-\frac{2}{h}f} v_h\, \partial_n v_h  \, d\sigma \right\vert \leq \frac{C}{h^{\alpha}}\, e^{-\frac{2}{h}(\lambda-\frac r2)},
$$
where we used  the Green formula (valid since $\ft C_r$ is $C^{\infty}$  for all $r\in (0,r_1)$) and the inclusion $\partial \ft C_{r/2}\subset \{f=\lambda-\frac r2\}$. 
In addition, since $\ft C_r\subset \ft C_{r/2}$ it holds,
\begin{align*} e^{-\frac{2}{h}(\lambda-r)}\, \int_{\ft C_{r}} \vert  \nabla v_h \vert^2 \,  \leq \int_{\ft C_{r}} \vert  \nabla v_h \vert^2 \, e^{-\frac{2}{h}f}\,  
&\leq \int_{\ft C_{r/2}} \vert  \nabla v_h \vert^2 \, e^{-\frac{2}{h}f}\leq \frac{C}{h^{\alpha}}\ e^{-\frac{2}{h}(\lambda-\frac r2)}.
\end{align*}
Therefore, there exists $\beta>0$ such that for $h$ small enough,
$$\int_{\ft C_{r}} \vert  \nabla v_h \vert^2 \,  \leq \frac{C}{h^{\alpha}}\, e^{-\frac{r}{h}}\leq  C\,e^{-\frac{\beta}{h}},$$
and from~\eqref{vh}, we then have $\Vert \Delta  v_h\Vert_{L^{2}(\ft C_{r})}\leq C\, e^{-\frac{\beta}{h}}$ for some constant  $\beta>0$.  In the following,~$\beta>0$ is a  constant which may change from one occurrence to another and does not depend on $h$. Let $\chi_1 \in C_c^{\infty}(\ft C_r)$ be such that $\chi_1 \equiv 1$ on $\ft C_{2r}$.  Since $\Delta (\chi_1 v_h)=\chi_1 \, \Delta v_h+v_h\, \Delta \chi_1 +2\nabla \chi_1 \cdot \nabla v_h$, there exists $C$, such that $\Vert \Delta (\chi_1 v_h)\Vert_{L^{2}(\ft C_{r})}\leq C$ for $h$ small enough. By  elliptic regularity (see~\cite[Theorem 5, Section 6.3]{Eva}) it comes 
$$\Vert  v_h\Vert_{H^{2}(\ft C_{2r})}\leq C.$$
Let $\alpha\in (0,1)$ be an irrational number such that  $p_1=\frac{2d}{d-2\alpha}>0$. From the Gagliardo-Nirenberg interpolation inequality (see \cite[Lecture II]{Nir}), the following inequality holds
$$
\Vert \nabla v_h\Vert_{L^{p_1}(\ft C_{2r})} \leq C \Vert v_h\Vert_{H^{2}(\ft C_{2r})}^\alpha\Vert \nabla v_h\Vert_{L^{2}(\ft C_{2r})}^{1-\alpha}+C\Vert \nabla v_h\Vert_{L^{2}(\ft C_{2r})}\leq C\, e^{-\frac{\beta}{h}}. 
$$
From~\eqref{vh},~$\Vert \Delta  v_h\Vert_{L^{p_1}(\ft C_{2r})}\leq C\, e^{-\frac{\beta}{h}}$. Using a cutoff  function $\chi_2 \in C_c^{\infty}(\ft C_{2r})$ such that $\chi_2 \equiv 1$ on $\ft C_{4r}$, we get, as previously, from the elliptic regularity $\Vert  v_h\Vert_{W^{2,p_1}(\ft C_{4r})}\leq C$. Let $p_2=\frac{2d}{d-4\alpha}$ (i.e. $1/p_2=1/p_1-\alpha/d$). If $p_2<0$, then~\cite[Lecture II]{Nir} implies
$$
\Vert \nabla v_h\Vert_{L^{\infty}(\ft C_{4r})} \leq C \Vert v_h\Vert_{W^{2,p_1}(\ft C_{4r})}^\alpha\Vert \nabla v_h\Vert_{L^{p_1}(\ft C_{4r})}^{1-\alpha}+C\Vert \nabla v_h\Vert_{L^{p_1}(\ft C_{4r})}\leq C\, e^{-\frac{\beta}{h}}. 
$$
Thus,~\eqref{est3} is proved  (if one  chooses $n=2$, i.e. $2^2r=r_0$). Otherwise, we prove~\eqref{est3} by induction as follows. 
From the Gagliardo-Nirenberg interpolation inequality (see \cite[Lecture II]{Nir}), we get
$$
\Vert \nabla v_h\Vert_{L^{p_2}(\ft C_{4r})} \leq C \Vert v_h\Vert_{W^{2,p_1}(\ft C_{4r})}^\alpha\Vert \nabla v_h\Vert_{L^{p_1}(\ft C_{4r})}^{1-\alpha}+C\Vert \nabla v_h\Vert_{L^{p_1}(\ft C_{4r})}\leq C\, e^{-\frac{\beta}{h}}. 
$$ 
We repeat this procedure $n$ times where $n$ is the first integer such that $d-2n\alpha< 0$ and  the Gagliardo-Nirenberg interpolation inequality  implies that $\Vert \nabla v_h\Vert_{L^{\infty}(\ft C_{2^nr})} \leq  C\, e^{-\frac{\beta}{h}}$ which ends the proof of~\eqref{est3}.  This concludes the proof of Proposition~\ref{level}. 
\end{proof}

\subsection{Link between the law of $X_{\tau_\Omega}$ when $X_0\sim \nu_h$ and $X_0=x\in \mathcal A(\ft C_{\ft{max}})$}\label{sec:proofTh1}

Using Proposition~\ref{level}, one can now compare $\mathbb E_{\nu_h}
\left [ F\left (X_{\tau_{\Omega}}\right)\right ]$ and $\mathbb E_{x}
\left [ F\left (X_{\tau_{\Omega}}\right)\right ]$ for smooth functions
$F$: the next proposition combined with Theorem~\ref{thm.main-bis} already gives
the result of Theorem~\ref{thm.main} for smooth functions~$F$.

 \begin{proposition}\label{pr.exp-qsd-x}
Assume that the assumptions \eqref{H-M}  and \eqref{eq.hip1}  are
satisfied and that 
$$\min_{\overline{\ft   C_{\ft{max}}} }f= \min_{\overline \Omega}f,$$
where we recall that  $\ft C_{\ft{max}}$ is introduced in~\eqref{eq.hip1}.  
  Let $K$ be  a compact subset of $\Omega$ such that $K\subset \mathcal A(\ft   C_{\ft{max}})$ and let $F\in C^{\infty}(\partial \Omega,\mathbb R)$. Then, there exists $ c>0$ such that for all $x\in K$:
$$
\mathbb E_{\nu_h}  \left [ F\left (X_{\tau_{\Omega}}\right)\right ]=\mathbb E_{x}  \left [ F\left (X_{\tau_{\Omega}}\right)\right ]  +O\big (e^{-\frac{c}{h} }\big )
$$
in the limit $h \to 0$ and uniformly in~$x \in K$.
\end{proposition}

\begin{proof}
Assume that the assumptions \eqref{H-M}  and \eqref{eq.hip1}  are satisfied and  that 
$$\min_{\overline{ \ft C_{\ft{max}}} }f= \min_{\overline \Omega}f.$$
 \textbf{Step 1.}    For $\alpha>0$ small enough, let $\ft C_{\ft{max}}(\alpha)$ be  as introduced in~\eqref{eq.c1a}: 
$$
\ft C_{\ft{max}}(\alpha)= \ft C_{\ft{max}}\cap \big\{f<f_{\ft{max}}  -\alpha\big\}.
$$ 
Let $F\in C^{\infty}(\partial \Omega,\mathbb R)$. In this first step, we will prove that ~$\exists \alpha_0>0$,  $\forall \alpha\in (0,\alpha_0)$,~$\exists c>0$,~$\forall y\in \overline{\ft C_{\ft{max}}(\alpha)}$:
\begin{equation}\label{eq.step1-e}
\mathbb E_{\nu_h}  \left [ F\left (X_{\tau_{\Omega}}\right)\right ]=\mathbb E_{y}  \left [ F\left (X_{\tau_{\Omega}}\right)\right ]  +O\big (e^{-\frac{c}{h} }\big )
\end{equation}
in the limit $h \to 0$ and uniformly in~$y \in \overline{\ft C_{\ft{max}}(\alpha)}$. 
Let us recall that from the notation of Proposition~\ref{level} (see~\eqref{eq.vh-dynkin}), for all $x\in \overline \Omega$:
$$v_h(x)=\mathbb E_{x}  \left [ F\left (X_{\tau_{\Omega}}\right)\right ].$$
From~\eqref{eq.expQSD}, one has:
\begin{align}
\nonumber
\mathbb E_{\nu_h}  \left [ F\left (X_{\tau_{\Omega}}\right)\right ]&= \left (\int_{\Omega} u_h\, e^{-\frac{2}{h}f}\right )^{-1}\ \int_{\Omega}v_h\, u_h\, e^{-\frac{2}{h}f} \\
\label{eq.equation1} 
&= \frac{1}{Z_h(\Omega)}\, \int_{\overline{\ft C_{\ft{max}}(\alpha)}   } \ v_h\, u_h\, e^{-\frac{2}{h}f}
 + \frac{1}{Z_h(\Omega)}\int_{ \Omega \setminus \overline{\ft C_{\ft{max}}(\alpha)}  } \ v_h\, u_h\, e^{-\frac{2}{h}f} ,
\end{align}
where 
$$Z_h(\Omega):=\displaystyle \int  _{\Omega} u_h \ e^{- \frac{2}{h} f } $$ and $u_h$ is the principal eigenfunction   of $-L^{D}_{f,h}$ which satisfies~\eqref{eq.norma}. 
Let us first deal with the second term in~\eqref{eq.equation1}. Since~\eqref{H-M} and \eqref{eq.hip1} hold, and because it is assumed that $\min_{\overline{\ft C_{\ft{max}}}}f= \min_{\overline \Omega}f$, one obtains from Corollary~\ref{co.moyenne} that   there exists $C>0$ such that for   $h$ small enough:
$$\frac{1}{Z_h(\Omega) }  \le Ch^{-\frac d4} e^{\frac 1h \min  \limits_{\overline \Omega} f}. $$
%Furthermore, since $\Vert v_h\Vert_{L^{\infty}(\overline \Omega)}\leq \Vert F\Vert_{L^{\infty}(\partial \Omega)}$, one has:
%$$
%\frac{1}{Z_h(\Omega) }  \int_{ \Omega \setminus \overline{\ft C_{\ft{max}}(\alpha)}  }\   v_h\, u_h\, e^{-\frac{2}{h}f}  = O\left( \frac{1}{Z_h(\Omega) } \int_{ \Omega \setminus \overline{\ft C_{\ft{max}}(\alpha)}  }  u_h\, e^{-\frac{2}{h}f}\right).
%$$
%Let us recall that from~\eqref{hs2-c1}, one has  $\argmin_{\ft C_{\ft{max}}}f=\argmin_{\Omega}f=\argmin_{\overline\Omega}f$ and thus for $\alpha$ small enough
%$$\argmin_{\overline\Omega}f=\argmin_{\Omega}f =\argmin_{\ft C_{\ft{max}}(\alpha) }f.$$
Let us recall that for   $\alpha>0$ small enough, one has (see~\eqref{eq.cmax-arg}), 
$$\argmin_{\overline{ \ft C_{\ft{max}}}}f \subset \ft C_{\ft{max}}(\alpha).$$
Therefore, using~the second statement in  Corollary~\ref{co.moyenne} with $\ft O=  \Omega \setminus \overline{\ft C_{\ft{max}}(\alpha)}$,  for all $\alpha>0$ small enough,  there exists $c>0$ such that  when $h\to 0$:
$$ \int_{ \Omega \setminus \overline{\ft C_{\ft{max}}(\alpha)}  }  u_h\, e^{-\frac{2}{h}f}=O\Big(e^{-\frac 1h  (\min  \limits_{\overline \Omega} f+c )}\Big).$$
 Thus,  there exists $\alpha_0>0$ such that for all $\alpha\in (0,\alpha_0)$ there exists $c>0$ such that  when $h\to 0$:
\begin{equation}
\label{eq.pp2}
\frac{1}{Z_h(\Omega) } \int_{ \Omega \setminus \overline{\ft C_{\ft{max}}(\alpha)}  }  u_h\, e^{-\frac{2}{h}f} =O\big (e^{-\frac ch}\big).
\end{equation}
Then, since $\Vert v_h\Vert_{L^{\infty}(\overline \Omega)}\leq \Vert F\Vert_{L^{\infty}(\partial \Omega)}$,  one obtains that
 \begin{equation}
\label{eq.pp2-bis}
\frac{1}{Z_h(\Omega) }  \int_{ \Omega \setminus \overline{\ft C_{\ft{max}}(\alpha)}  }  \ v_h\, u_h\, e^{-\frac{2}{h}f}=O\big (e^{-\frac ch}\big).
\end{equation}
Let us now deal with the first term in~\eqref{eq.equation1}. Let us recall that 
 $\ft C_{\ft{max}}\subset \Omega$ is a connected component of $ \ft \{f<\max_{\overline{\ft C_{\ft{max}}}}f\}$.  
Moreover, for  $\alpha\in (0,\alpha_0)$ ($\alpha_0>0$ small enough),  
the compact set $\overline{\ft C_{\ft{max}}(\alpha)}$ is   connected and  $\overline{\ft C_{\ft{max}}(\alpha)}\subset \ft C_{\ft{max}}$. Therefore, from Proposition~\ref{level} applied to  $K=\overline{\ft C_{\ft{max}}(\alpha)}$ for $\alpha\in (0,\alpha_0)$, one obtains that there exists $\delta_{\alpha}>0$ such that for all $y\in \overline{\ft C_{\ft{max}}(\alpha)}$,
\begin{equation}\label{eq.equality}
\frac{1}{Z_h(\Omega) } \int_{\overline{\ft C_{\ft{max}}(\alpha)}} \  v_h\, u_h\, e^{-\frac{2}{h}f}=\, \frac{v_h(y)}{Z_h(\Omega) } \int_{\overline{\ft C_{\ft{max}}(\alpha)}}  \ u_h\, e^{-\frac{2}{h}f}+ \frac{O\big( e^{-\frac{\delta_{\alpha}}{h}} \big ) }{Z_h(\Omega) }  \displaystyle{\int_{\overline{\ft C_{\ft{max}}(\alpha)}}  \  u_h\, e^{-\frac{2}{h}f} } 
\end{equation}
 in the limit $h\to 0$ and uniformly with respect to $y \in \overline{\ft C_{\ft{max}}(\alpha)}$.
Moreover,  for all $\alpha\in (0,\alpha_0)$ there exists $c>0$ such that  in the limit $h\to 0$:
\begin{equation}\label{eq.1-ka}
 \frac{1}{Z_h(\Omega) }  \int_{\overline{\ft C_{\ft{max}}(\alpha)}}  u_h\, e^{-\frac{2}{h}f} =1+O\left (e^{-\frac{c}{h}}\right ).
\end{equation}
which follows from the fact that 
$$\frac{1}{Z_h(\Omega) }  \int_{\overline{\ft C_{\ft{max}}(\alpha)}}  u_h\, e^{-\frac{2}{h}f}    =1- \frac{1}{Z_h(\Omega) } \int_{\Omega\setminus \overline{\ft C_{\ft{max}}(\alpha)}}  u_h\, e^{-\frac{2}{h}f},$$
together with~\eqref{eq.pp2}. 
Let us now fix $\alpha\in (0,\alpha_0)$. Then, using~\eqref{eq.equality} and~\eqref{eq.1-ka},  $\exists c>0$,  $\exists \delta_\alpha>0$,  $\forall y\in  \overline{\ft C_{\ft{max}}(\alpha)}$:
\begin{equation}\label{eq.equ1}
\frac{1}{Z_h(\Omega) }  \, \int_{\overline{\ft C_{\ft{max}}(\alpha)}} \ v_h\, u_h\, e^{-\frac{2}{h}f}=v_h(y)\left (1+O\left (e^{-\frac{c}{h}}\right )\right ) +O\left ( e^{-\frac{\delta_{\alpha}}{h}}  \right)
\end{equation}
in the limit $h \to 0$ and uniformly with respect to~$y \in \overline{\ft C_{\ft{max}}(\alpha)}$. Therefore, using~\eqref{eq.equation1}, \eqref{eq.pp2-bis} and \eqref{eq.equ1},~$\exists \alpha_0>0$,  $\forall \alpha\in (0,\alpha_0)$,~$\exists c>0$,~$\forall y\in \overline{\ft C_{\ft{max}}(\alpha)}$:
$$
\mathbb E_{\nu_h}  \left [ F\left (X_{\tau_{\Omega}}\right)\right ]= \mathbb E_{y}  \left [ F\left (X_{\tau_{\Omega}}\right)\right ]+O\big (e^{-\frac{c}{h} }\big ),
$$
in the limit $h \to 0$ and uniformly with respect to  $y \in \overline{\ft C_{\ft{max}}(\alpha)}$. 
This concludes the proof of~\eqref{eq.step1-e}.
\medskip

\noindent
 \textbf{Step 2.}  Let us now conclude the  proof of Proposition~\ref{pr.exp-qsd-x} by considering a compact subset~$K$ of~$\Omega$ such that $ K\subset \mathcal A(\ft C_{\ft{max}})$. Let us recall that (see~\eqref{eq.ad}):
$$\mathcal A(\ft C_{\ft{max}})=\{ x\in \Omega, \,t_x=+\infty \text{ and } \omega(x)\subset \ft C_{\ft{max}}\}.$$ 
Since $\ft C_{\ft{max}}$ is open and stable by the flow
$\varphi_t(\cdot)$ (defined by~\eqref{hbb}), the continuity of  $\varphi_t(\cdot)$   implies that  there exists  $T_K\ge 0$ such that for all  $x\in K$, 
$$\varphi_{T_K}(x)\in \ft C_{\ft{max}}.$$ 
Moreover, since $K$ is a compact subset of $\Omega$ and for all $x\in K$, $t_x=+\infty$ (i.e. $\varphi_t(x)\in \Omega$ for all $t\ge 0$), there exists $\delta>0$ such that  all continuous curves $\gamma: [0,T_K]\to \overline \Omega$ such that  
$$\exists x\in K, \  \sup_{t\in [0,T_K]} \big \vert \gamma(t)-\varphi_t(x) \big \vert \le \delta,  $$
satisfy: 
\begin{equation}\label{eq.incluo}
\forall t\in [0,T_K], \ \gamma(t)\in  \Omega.
\end{equation}
Furthermore, up to choosing $\delta>0$ smaller, there exists $\alpha_K>0$ such that 
\begin{equation}\label{eq.incluo2}
\big \{\varphi_{T_K}(x)+z,  \ x\in K \text{ and } \vert z\vert \le \delta \big  \}\subset \ft C_{\ft{max}}( {\alpha_K})
\end{equation}
where, we recall, $\ft C_{\ft{max}}( {\alpha_K})$ is defined by~\eqref{eq.c1a}.
Let us now recall the following estimate of  Freidlin and Wentzell (see~\cite[Theorems 2.2 and 2.3 in Chapter 3, and Theorem 1.1 in Chapter 4]{FrWe},~\cite{Day2},~\cite[Theorem 3.5]{DZ} and~\cite[Theorem 5.6.3]{friedman2012stochastic}): for all $x\in K$, it holds:
\begin{equation}\label{eq.WFr-est}
\limsup_{h\to 0} h\, \ln \mathbb P_x\Big [ \sup_{t\in [0,T_K]} \big \vert X_t-\varphi_t(x) \big \vert \ge \delta   \Big ]\le - I_{x,T_K},
\end{equation}
where 
$$I_{x,T_K}=  \frac 12\,  \inf_{ \gamma \in H^1_{x,T_K} (\delta) } \,  \int_0^{T_K}\Big \vert \frac{d}{dt}  \gamma(t)+ \nabla f( \gamma(t))\Big \vert^2dt \ \ \in \mathbb R_+^*\cup \{+\infty\},$$
and $H^1_{x,T_K}(\delta)$ is the set of curves $ \gamma: [0,T_K]\to \Omega$ of regularity $H^1$ such that $ \gamma(0)=x$ and $\sup_{t\in [0,T_K]} \big \vert  \gamma(t)-\varphi_t(x) \big \vert \ge \delta$. 
Since  $K$ is compact, there exists $\eta_K>0$ such that for    $h$ small enough, it holds:
\begin{equation}\label{eq.WFr}
\sup_{x\in K}\mathbb P_x\Big [ \sup_{t\in [0,T_K]} \big \vert X_t-\varphi_t(x) \big \vert \ge \delta    \Big ]\le e^{-\frac{\eta_K}{h}}.
\end{equation}
Notice that  when $X_0=x\in K$ and  $\sup_{t\in [0,T_K]} \big \vert X_t-\varphi_t(x) \big \vert \le \delta$, it holds from~\eqref{eq.incluo} and~\eqref{eq.incluo2}:
\begin{equation}\label{eq.TtC}
\tau_\Omega>T_K \text{ and } X_{T_K}\in \ft C_{\ft{max}}( {\alpha_K}).
\end{equation}
Let us now consider $F\in C^\infty(\pa \Omega,\mathbb R)$. Let  $x\in K$. Then, 
$$ \mathbb E_{x}  \left [ F\left (X_{\tau_{\Omega}}\right)\right ]=  \mathbb E_{x}  \left [ F\left (X_{\tau_{\Omega}}\right) \mathbf{1}_{\sup_{t\in [0,T_K]} \big \vert X_t-\varphi_t(x) \big \vert \le \delta} \right ]+ \mathbb E_{x}  \left [ F\left (X_{\tau_{\Omega}}\right) \mathbf{1}_{\sup_{t\in [0,T_K]} \big \vert X_t-\varphi_t(x) \big \vert \ge \delta} \right ].$$
Using~\eqref{eq.WFr}, it holds for $h$ small enough:
$$\Big \vert \mathbb E_{x}  \left [ F\left (X_{\tau_{\Omega}}\right) \mathbf{1}_{\sup_{t\in [0,{T_K}]} \big \vert X_t-\varphi_t(x) \big \vert \ge \delta} \right ]\Big \vert  \le  \Vert F\Vert_{L^{\infty}}    \,  e^{-\frac{\eta_K}{h}}.$$
Using~\eqref{eq.TtC},~\eqref{eq.step1-e} (with $\alpha= \alpha_K$),~\eqref{eq.WFr}, and  the Markov property of the process~\eqref{eq.langevin},  there exists $c>0$ such that for all $x\in K$, one has when $h\to 0$:
\begin{align*}
  \mathbb E_{x}  &\left [ F\left (X_{\tau_{\Omega}}\right) \mathbf{1}_{\sup_{t\in [0,{T_K}]} \big \vert X_t-\varphi_t(x) \big \vert \le \delta} \right ]=  \mathbb E_{x}  \left [   \mathbb E_{X_{T_K}}\big[F\left (X_{\tau_{\Omega}}\right)\big] \mathbf{1}_{\sup_{t\in [0,{T_K}]} \big \vert X_t-\varphi_t(x) \big \vert \le \delta} \right ] 
  \\
  &=\Big(\mathbb E_{\nu_h}  \left [ F\left
    (X_{\tau_{\Omega}}\right)\right ]  +O\big (e^{-\frac{c}{h} }\big
    )\Big)\mathbb P_x\Big  [ \sup_{t\in [0,{T_K}]} \big \vert X_t-\varphi_t(x) \big \vert \le \delta    \Big  ] \\
  &=\mathbb E_{\nu_h}  \left [ F\left (X_{\tau_{\Omega}}\right)\right ]  +O\big (e^{-\frac{c}{h} }\big ),
  \end{align*}
  uniformly in $x\in K$. This concludes the proof of Proposition~\ref{pr.exp-qsd-x}.
\end{proof}

\subsection{From smooth functions $F$ to non-smooth functions $F$}\label{sec:end_proof}

 Let us now complete the proof of Theorem~\ref{thm.main}. 
 \begin{proof} In the following we assume that  \eqref{H-M},
   \eqref{eq.hip1}, \eqref{eq.hip2} and \eqref{eq.hip3} are
   satisfied. We recall that this implies that $\min_{\overline{\ft
       C_{\ft{max}}}}f= \min_{\overline \Omega}f$ and thus, the
   results of  Proposition~\ref{pr.exp-qsd-x} hold.
 Let $K$ be  a compact subset of $\Omega$ such that $$K\subset \mathcal A(\ft C_{\ft{max}})$$
 and let us assume that the process starts from $X_0=x\in K$.
Let $F\in L^{\infty}(\partial \Omega,\mathbb R)$.
 The proof of Theorem~\ref{thm.main} is  divided into three steps.

\medskip

\noindent
\textbf{Step 1.} Proof of~\eqref{eq.t1}  and \eqref{eq.t2}. %\overline{\ft C_{\ft{max}}(\alpha)}$. 
\medskip

\begin{sloppypar}
\noindent
 Let    us first show that if $\Sigma\subset \partial \Omega$ is open and   there exists $\beta>0$ such that $\Sigma\cap \bigcup_{i=1}^{k_1^{\pa \Omega}}B_{\partial \Omega}(z_i,\beta)=\emptyset$ (where $B_{\partial \Omega}(z_i,\beta)$ is the open ball in~$\partial \Omega$ of radius $\beta$ centered at~$z_i$), then, for all $x\in K$, \end{sloppypar}
 %~$\exists \alpha_0>0$,  $\forall \alpha\in (0,\alpha_0)$,~$\exists c>0$,~$\forall x\in \overline{\ft C_{\ft{max}}(\alpha)}$:

\begin{equation}\label{eq.sigma-x}
\mathbb P_{x}  \left [ X_{\tau_{\Omega}}\in \Sigma \right ]=O\big(e^{-\frac{c}{h}}\big)
\end{equation}
in the limit $h \to 0$ and uniformly in~$x \in K$. % \overline{\ft C_{\ft{max}}(\alpha)}$.
 To this end, let us 
 consider $\tilde F\in C^{\infty}(\partial \Omega,[0,1])$ be such that 
 $$\tilde F=1 \text{ on } \Sigma \ \text{ and } \ \tilde F=0 \text{ on }\, \bigcup_{i=1}^{\ft k_1^{\pa \Omega}}B_{\partial \Omega}(z_i,\frac{\beta}{2}).$$
  Using Proposition~\ref{pr.exp-qsd-x}, 
there exists $  c>0$ such that for all $x\in K$:
 $$\mathbb P_{x}  \left [ X_{\tau_{\Omega}}\in \Sigma \right ]\le \mathbb E_{x}  \left [ \tilde F(X_{\tau_{\Omega}})\right]= \mathbb E_{\nu_h}  \left [ \tilde F(X_{\tau_{\Omega}})\right]+O\left(e^{-\frac ch}\right )$$
 in the limit $h \to 0$ and uniformly in~$x \in K$. %\overline{\ft C_{\ft{max}}(\alpha)}$.
Then, Equation~\eqref{eq.sigma-x} follows from~\eqref{eq.t1-bis}
applied to  $\tilde F$ and the family of sets $\Sigma_i=B_{\partial
  \Omega}(z_i,\frac{\beta}{2})$ for $i\in\{1,\ldots,\ft k_1^{\pa
  \Omega}\}$. 

Let us now prove~\eqref{eq.t1} and~\eqref{eq.t2}. 
Let $F\in L^{\infty}(\partial \Omega,\mathbb R)$ and for all $i\in\{1,\dots,\ft k_{1}^{\pa \Omega}\}$, let $\Sigma_{i}\subset \pa \Omega$ be an open set which contains  $z_{i}$.  Let us assume in addition that $\Sigma_i\cap \Sigma_j=\emptyset$ if $i\neq j$. One has for any 
 $x\in K$%$x\in \overline{\ft C_{\ft{max}}(\alpha)}$:
$$\mathbb E_{x}  \left [ F(X_{\tau_{\Omega}})\right]=\sum \limits_{i=1}^{\ft k_1^{\pa \Omega}}
\mathbb E_{x} \left [ (\mathbf{1}_{\Sigma_{i}}F)\left (X_{\tau_{\Omega}} \right )\right]  +  \mathbb E_{x} \Big  [ (\mathbf{1}_{\pa \Omega\setminus \bigcup_{i=1}^{\ft k_1^{\pa \Omega}}\Sigma_{i} }F)\left (X_{\tau_{\Omega}} \right )\Big ].$$
Moreover, one has:
$$\Big \vert  \mathbb E_{x} \Big  [ (\mathbf{1}_{\pa \Omega\setminus\bigcup_{i=1}^{\ft k_1^{\pa \Omega}}\Sigma_{i}}F)\left (X_{\tau_{\Omega}} \right )\Big  ]\Big \vert \le \Vert F\Vert_{L^{\infty}}\, \mathbb P_{x} \Big [ X_{\tau_{\Omega}} \in \pa \Omega\setminus \bigcup_{i=1}^{\ft k_1^{\pa \Omega}}\Sigma_{i}   \Big].$$
Using~\eqref{eq.sigma-x} with $\Sigma=\pa \Omega\setminus
\bigcup_{i=1}^{\ft k_1^{\pa \Omega}}\Sigma_{i} $, one
gets~\eqref{eq.t1}. 

Let us now prove~\eqref{eq.t2}. Let $j\in \{\ft k_1^{\pa \ft
  C_{\ft{max}}}+1,\ldots,\ft k_1^{\pa \Omega}\}$. Let $\delta>0$ be such that for any $k\in \{1,\ldots,\ft k_1^{\pa \Omega}\}$ with $k\neq j$, the sets $ B_{\partial \Omega}(z_k, \delta) $  and $\tilde \Sigma_j:=\cup_{z\in \Sigma_j} B_{\partial \Omega}(z, \delta)$ are disjoint. Let us consider 
$$G\in C^{\infty}_c(\tilde \Sigma_j,[0,1]) \text{ such that } G=1 \text{ on } \Sigma_j.$$
 Using Proposition~\ref{pr.exp-qsd-x}, there exists $  c>0$ such that for all $x\in K$,
\begin{align*}
\big\vert \mathbb E_{x}  \left [ (\mathbf{1}_{\Sigma_{j}}F)(X_{\tau_{\Omega}})\right] \big \vert \le \Vert F\Vert_{L^{\infty}}\, \mathbb P_{x} [ X_{\tau_{\Omega}} \in  \Sigma_j  ]
&\le \Vert F \Vert_{L^{\infty}}\,\mathbb E_{x}  \left [ G(X_{\tau_{\Omega}})\right]\\
&= O(\mathbb E_{\nu_h}  \left [ G(X_{\tau_{\Omega}})\right]  )+ O(e^{-\frac ch})
\end{align*}
 in the limit $h \to 0$ and uniformly in~$x \in K$.   
Then, using~\eqref{eq.t2-bis} and item 3 in Theorem~\ref{thm.main-bis}, it holds when $h\to 0$: 
$$  \mathbb E_{x}  \left [ (\mathbf{1}_{\Sigma_{j}}F)(X_{\tau_{\Omega}})\right]=O\big (h^{\frac 14}\big ),$$
 and when \eqref{eq.hip4} holds, one has  when $h\to 0$: $$ \mathbb E_{x}  \left [( \mathbf{1}_{\Sigma_{j}}F)(X_{\tau_{\Omega}})\right]=O(e^{-\frac ch}),$$
 for some $c>0$.  
 This concludes the proof of~\eqref{eq.t2}. 

\medskip
\noindent
\textbf{Step 2.} Proof of~\eqref{eq.t3}.    
\medskip

\noindent
For all  $j\in\{1,\dots,\ft k_{1}^{\pa \Omega}\}$, let $\Sigma_{j}$ be open subset of $\pa \Omega$ such that $z_{j}  \in \Sigma_{j}$.  Let us assume that $\Sigma_k\cap \Sigma_j=\emptyset$ if $k\neq j$. Let $F\in L^{\infty}(\partial \Omega,\mathbb R)$ be $C^{\infty}$ in a neighborhood  of $z_i$ for some $i\in\{1,\dots,\ft k_{1}^{\pa \ft C_{\ft{max}}}\}$.  Let $\beta>0$ be such that $F$ is $C^{\infty}$ on $B_{\partial \Omega}(z_i,2\beta )\subset \Sigma_i$ and let 
$\chi_i\in C^{\infty}(\pa \Omega,[0,1])$ be such that 

$${\rm supp}\, \chi_i\subset B_{\partial \Omega}(z_i,\beta )\, \text{ and } \, \chi_i=1 \text{ on } B_{\partial \Omega}(z_i,\beta/2).$$ 
One has:
$$\mathbb E_{x}  \left [ (\mathbf{1}_{\Sigma_i}F)(X_{\tau_{\Omega}})\right]=\mathbb E_{x}  \left [( \chi_iF)(X_{\tau_{\Omega}})\right]+\mathbb E_{x}  \left [ \big((\mathbf{1}_{\Sigma_i}-\chi_i) F\big)(X_{\tau_{\Omega}})\right].$$
Using Proposition~\ref{pr.exp-qsd-x} with $\chi_iF\in C^\infty$ and~\eqref{eq.t1-bis}-\eqref{eq.t3-bis} with $X_0\sim \nu_h$,~$F\chi_i$ and the family of pairwise disjoint open sets $\{ \Sigma_j, j=1,\ldots,\ft k_1^{\pa \Omega}, j\neq i \} \cup \{B_{\partial \Omega}(z_i,\frac{\beta}{2})\}$,  
there exists $  c>0$ such that for all $x\in K$:
%$\exists \alpha_0>0$,  $\forall \alpha\in (0,\alpha_0)$,~$\exists c>0$,~$\forall x\in \overline{\ft C_{\ft{max}}(\alpha)}$:
\begin{align*}
\mathbb E_{x}  \left [ (\chi_iF)(X_{\tau_{\Omega}})\right] &=\mathbb E_{\nu_h}  \left [ (\chi_iF)(X_{\tau_{\Omega}})\right] +O\left (e^{-\frac{c}{h} }\right )\\
&=\mathbb E_{\nu_h}  \left [ (\mathbf{1}_{B_{\partial \Omega}(z_i,\frac{\beta}{2})}F)(X_{\tau_{\Omega}})\right] +O\left (e^{-\frac{c}{h} }\right )=F(z_i)\, a_i+O\big (h^{\frac 14}\big )
\end{align*}
 in the limit $h \to 0$ and uniformly in~$x \in K$,  %\overline{\ft C_{\ft{max}}(\alpha)}$, 
and where $a_i$ is defined in~\eqref{ai}. In addition, using item 3 in Theorem~\ref{thm.main-bis}, when~\eqref{eq.hip4} holds, one can replace $O\big (h^{\frac 14}\big )$ in the last computation by~$O(h)$.
Moreover, using~\eqref{eq.sigma-x} with $\Sigma=\Sigma_i\setminus  B_{\partial \Omega}(z_i,\frac{\beta}{2})$:  there exists $  c>0$ such that for all $x\in K$:
%$\exists \alpha_0>0$,  $\forall \alpha\in (0,\alpha_0)$,~$\exists c>0$,~$\forall x\in \overline{\ft C_{\ft{max}}(\alpha)}$:
$$\big \vert  \mathbb E_{x}  \left [ ((\mathbf{1}_{\Sigma_i}-\chi_i)F)(X_{\tau_{\Omega}})\right] \big \vert \le \Vert F\Vert_{L^{\infty}}\, \mathbb P_{x} \left [ X_{\tau_{\Omega}} \in \Sigma_i\setminus  B_{\partial \Omega}\Big (z_i,\frac{\beta}{2}\Big ) \right]= O\left (e^{-\frac ch}\right)$$
in the limit $h \to 0$ and uniformly in~$x \in K$. % \overline{\ft C_{\ft{max}}(\alpha)}$. 
Thus, one has when $h\to 0$ and uniformly with respect to $x\in K$:
%\overline{\ft C_{\ft{max}}(\alpha)}$: 
$$\mathbb E_{x}  \left [ (\mathbf{1}_{\Sigma_i}F)(X_{\tau_{\Omega}})\right]=F(z_i)\, a_i+O\big (h^{\frac 14}\big ),$$ and when~\eqref{eq.hip4} holds, one has: 
$$\mathbb E_{x}  \left [( \mathbf{1}_{\Sigma_i}F)(X_{\tau_{\Omega}})\right]=F(z_i)\, a_i+O(h).$$  This concludes the proof of~\eqref{eq.t3}. The proof of   Theorem~\ref{thm.main} si complete.
\end{proof}

%%%%%

\section{On the exit point distribution when $X_0=x\in \mathcal A(\ft
  C)$, where
  $\ft C \in \mathcal C$}\label{sec:proof_th2}

In this section, one proves Theorem~\ref{thm.2} which aims at giving
the concentration of the law of $X_{\tau_\Omega}$ in the limit $h \to
0$, when  $X_0=x\in \mathcal A(\ft C)$, where
  $\ft C \in \mathcal C$ (we recall that $\mathcal C$ has been defined in~\eqref{mathcalC-def}).

 %%%%%
 \subsection{Proof of Theorem~\ref{thm.2}}
 \label{sec.proffthm2}
 
\begin{proof}[Proof of Theorem~\ref{thm.2}.]
Let us assume that  \eqref{H-M} holds. Let $\ft C\in \mathcal C$. Assume that  (see~\eqref{eq.cc1})
$$
\pa \ft C\cap \pa \Omega\neq \emptyset \  \text{ and } \  \vert \nabla f\vert \neq 0 \text{ on } \pa \ft C.
$$
To prove Theorem~\ref{thm.2}, the strategy consists in using
Theorem~\ref{thm.main}  with a    subdomain   $\Omega_{\ft C}$ of
$\Omega$ containing $\ft C$ such that in the limit  $h\to 0$,  the
most probable places of exit of the process~\eqref{eq.langevin} from
$\Omega_{\ft C}$ when $X_0=x\in \ft C$  are the elements of $\pa \ft C
\cap \pa \Omega$. This will   imply (since the trajectories of the
process \eqref{eq.langevin} are continuous) that the most probable
places of exit of the process~\eqref{eq.langevin} from $\Omega$ when
$X_0=x\in \ft C$  are the elements of $\pa \ft C \cap \pa \Omega$,
which is the statement of Theorem~\ref{thm.2}. This result will be extend to initial conditions $X_0=x\in \mathcal A(\ft C)$ using a large deviations method.

The proof of Theorem~\ref{thm.2} is divided into two steps. 
\medskip

\noindent
\textbf{Step 1}: Construction of a domain $\Omega_{\ft C}$ containing $\ft C$.  
\medskip

\noindent
In this step, one  constructs a subset $\Omega_{\ft C}$ of $\Omega$ such that 
\begin{equation}\label{eq.OmegaC}
\left\{
\begin{aligned}
&\text{$\Omega_{\ft C}$ is a $C^\infty$ connected open subset of $\Omega$ containing $\ft C$},\\
 &\text{$\pa \Omega_{\ft C}\cap \pa \Omega $ is a neighborhood of  $\pa \ft  C \cap \pa \Omega$ in $\pa \Omega$},  \\ 
  &\text{argmin}_{\pa \Omega_{\ft C}}f=\pa \ft C\cap \pa \Omega, \\
  &\big \{ x\in  \overline{\Omega_{\ft C}},\,  f(x)<\min_{\pa \Omega_{\ft C}}f\big \}=\ft C,\\
  & \text{the critical points of $f$ in  } \overline{\Omega_{\ft C}} \text{ are included in } \ft C,
\end{aligned}
\right.
\end{equation}
and
\begin{equation}\label{eq.OmegaC2}
f: \pa \Omega_{\ft C} \to \mathbb R \text{ is a Morse function}.
\end{equation}

To construct a domain $\Omega_{\ft C}\subset \Omega$ which satisfies~\eqref{eq.OmegaC} and~\eqref{eq.OmegaC2}, we first introduce  a neighborhood $\ft V_{\ft C}$ of $ \overline{ \ft C}$ in $\overline \Omega$ as follows.   
Let $\lambda\in \mathbb R$ be such that $\ft C$ is a connected component of $\{f<\lambda\}$ (see~\eqref{eq.Cdef2}).   Then, for $z\in \pa \ft C$, we introduce a ball of radius $\ve_z>0$ centred at $z$ in $\overline \Omega$ as follows:  
\begin{enumerate}
\item If $z\in \pa \ft C\cap  \Omega$: Since $z\in \Omega$ and $\vert \nabla f(z)\vert \neq 0$, there exists  $\ve_z >0$ such that $\overline {B(z,\ve_z)}\subset \Omega$, 
 $\vert \nabla f(z)\vert \neq 0$ on $\overline {B(z,\ve_z)}$, and,  according to~\cite[Section 5.2]{HeNi1}, $ {B(z,\ve_z)}\cap \{f<\lambda \}$ is connected and   $ {B(z,\ve_z)}\cap \pa \{f<\lambda \}= {B(z,\ve_z)}\cap \{f=\lambda \}$ (where we recall that $B(z,\ve_z)=\{x\in \overline \Omega \ \text{s.t.} \ |x-z|<\ve_z\}$).

\item If $z\in \pa \ft C\cap \pa \Omega$: Recall  that $z\in 
\ft U_1^{\pa \Omega}$  (see~\eqref{eq.mathcalU1_bis})  and thus, $\pa_nf(z)>0$ and $z$ is a non degenerate local minimum of $f|_{\pa \Omega}$. Thus, there exists $\ve_z >0$,  such that $\vert \nabla f(z)\vert \neq 0$ on $\overline {B(z,\ve_z)}$  and such that, according to~\cite[Section 5.2]{HeNi1},   
 $ {B(z,\ve_z)}\cap \{f<\lambda \}$ is connected and included in $\Omega$. In addition,   $ {B(z,\ve_z)}\cap \pa \{f<\lambda \}= {B(z,\ve_z)}\cap \{f=\lambda \}$.
Finally,   up to choosing   $\ve_z>0$ smaller,  one has:
\begin{equation}\label{eq.GAMMAC0}
\argmin_{  \overline{B_{\pa \Omega}(z,\ve_z )} }   f=\{z\},
\end{equation}
where we recall that $B_{\partial \Omega}(z, \ve_z )$ is the open ball of radius $\ve_z$ centred in~$z$ in~$\pa \Omega$, and, 
\begin{equation}\label{eq.panf-b}
\vert \nabla_Tf \vert \neq 0 \text{ on } \overline {B_{\pa \Omega}(z,\ve_z)}\setminus\{z\}   \ \text{ and } \ \pa_nf>0 \text{ on }   \overline{B(z,\ve_z)}\cap \pa \Omega.
\end{equation}
\end{enumerate}
Items 1 an 2 above imply that for all $z\in \pa \ft C$, by definition of $\ft C$ (see Theorem~\ref{thm.2}), 
\begin{equation}\label{eq.GAMMAC01}
{B(z,\ve_z)}\cap \ft C={B(z,\ve_z)}\cap \{f<\lambda \}  \text{ and thus, }  {B(z,\ve_z)}\cap \pa \ft C={B(z,\ve_z)}\cap \{f=\lambda \}.
\end{equation}  
One then defines:
$$\ft V_{\ft C}:=  \left (\,  \bigcup_{z\in \pa \ft C}  B(z,\ve_z)\,  \right)\bigcup \ft C  .$$
The set $\ft V_{\ft C}$ is an   open neighborhood of   $\overline{\ft C}$ in $\overline \Omega$. Moreover, according to items~1 and~2 above,
\begin{equation}\label{eq.vC-s}
\vert\nabla f\vert \neq 0 \text{ on } \overline{\ft V_{\ft C}}\setminus \ft C,
\end{equation}
and  using in addition~\eqref{eq.GAMMAC01}, 
\begin{equation}\label{eq.vC-s2}
\{f< \lambda\} \cap \ft V_{\ft C}=\ft C\,  \text{ and } \{f\le \lambda\} \cap \ft V_{\ft C}= \overline{\ft C}.
\end{equation} 
The second statement in~\eqref{eq.vC-s2} implies that  $\overline{\ft C}$ is a connected component of $\{f\le \lambda\}$. Thus, for $r>0$ small enough $\overline{ {\ft C}(\lambda+r) }\subset \ft V_{\ft C}$, where ${\ft C}(\lambda+r)$ is the connected component of  $\{f< \lambda+r\}$ which contains $\ft C$. 
 This suggests that a natural candidate to satisfy~\eqref{eq.OmegaC} and~\eqref{eq.OmegaC2}  is  the domain ${\ft C}(\lambda+r)$. However, for $r>0$ small enough, the boundary of  ${\ft C}(\lambda+r)$ is not~$C^\infty$: it is composed  of two smooth pieces     $\overline{\pa\ft C(\lambda+r)\cap \Omega}=\{x\in \pa \ft C(\lambda+r), \, f(x) = \lambda + r\, \} $   and  $  \pa\ft C(\lambda+r) \cap \partial \Omega$. The union of this two sets  gives  rise  to "corners". 
  Moreover, the function $f|_{\pa\ft C(\lambda+r)\cap \Omega}$ is not a Morse function since $f\equiv \lambda+r$ on $\overline{\pa\ft C(\lambda+r)\cap \Omega}$.

To justify the existence of  a domain~$\Omega_{\ft C}$ which satisfies \eqref{eq.OmegaC} and~\eqref{eq.OmegaC2}, we now proceed in two steps, as follows.
%%%%%%% 
\begin{itemize}
\item \begin{sloppypar}\underline{Domain  $D_{\ft C}$   containing $\ft C$ which satisfies \eqref{eq.OmegaC} and $\pa_nf>0$ on $\pa D_{\ft C}$}. 
The subdomain $D_{\ft C}$ of $\Omega$ is constructed as a smooth regularization of the set $\ft C(\lambda+r)$ with $r>0$ such that $\overline{ {\ft C}(\lambda+r) }\subset \ft V_{\ft C}$ by modifying ${\ft C}(\lambda+r)$ in a neighborhood of ${ \{x\in \pa \ft C(\lambda+r), \, f(x) = \lambda + r\, \} } \cap \partial \Omega$ (where the two smooth pieces of $\pa  \ft C(\lambda+r)$ intersect each other).      Moreover, $\pa_n f>0$ on   $\overline{\pa\ft C(\lambda+r)\cap \Omega}$  (since there is no critical point of $f$ on $\overline{\pa\ft C(\lambda+r)\cap \Omega}= \{x\in \pa \ft C(\lambda+r), \, f(x) = \lambda + r\, \} $)   and on  $  \pa\ft C(\lambda+r) \cap \partial \Omega$ (since $\overline{ {\ft C}(\lambda+r) }\subset \ft V_{\ft C}$ and $\pa_nf>0$ on $\ft V_{\ft C}\cap \pa \Omega$,  see the second inequality in~\eqref{eq.panf-b}).\end{sloppypar}
% This indeed follows from the fact there is no critical point of $f$ on $\pa\ft C(\lambda+r)\cap \Omega$ on which $f\equiv \lambda+r$ and from the second inequality in~\eqref{eq.panf-b}. 
 Thus, using in addition~\eqref{eq.vC-s2} together with the fact that $\ft V_{\ft C}$ is an   open neighborhood of   $\overline{\ft C}$ in $\overline \Omega$,  there exists    a $C^\infty$ connected   open   subset  $D_{\ft C}$ of $\Omega$ such that 
\begin{equation}\label{G-DC0}
\ft C\subset D_{\ft C}, \ \overline{D_{\ft C}}\subset \ft V_{\ft C},
\end{equation}
and  
\begin{equation}\label{eq.pan-dc}
\pa_n f>0\text{ on } \pa D_{\ft C},
\end{equation}
which satisfies,  for some $\beta>0$ and $\Sigma_{\ft C}\subset \Omega$, 
\begin{equation}\label{G-DC}
\pa D_{\ft C}=  \left (\,  \bigcup_{z\in \pa \ft C\cap \pa \Omega}  B_{\pa \Omega}(z,\ve_z/2)    \right)\bigcup \,\overline{ \Sigma_{\ft C}}, \,  \text{ where, }  \, f\ge  \lambda+\beta \, \text{ on } \, \overline{\Sigma_{\ft C}}.  
\end{equation}
Finally, according to the first statement in~\eqref{eq.panf-b}, there exists $\delta_0>0$ such that for any  open  $\delta$-neighborhood $U_{\pa \Omega}^\delta$ of $\pa \Omega$ in $\overline \Omega$, with $\delta\in (0,\delta_0)$, one has 
\begin{equation}\label{eq.PC-dc}
\vert  \nabla_T f\vert \neq 0 \text{ on } \overline{ \pa D_{\ft C} \cap  {U_{\pa \Omega}^\delta}} \, \setminus (\pa \ft C\cap \pa \Omega),
\end{equation}
where $\nabla_Tf$ is the tangential gradient of $f$ on $\pa D_{\ft C}$. 
%One can easily check that $D_{\ft C}$ satisfies \eqref{eq.OmegaC}. Nevertheless, it does not necessarily satisfies~\eqref{eq.OmegaC2}. 
%%%%%%%%%%%%%%

\item \underline{Domain  $\Omega_{\ft C}$   containing $\ft C$ which satisfies \eqref{eq.OmegaC} and~\eqref{eq.OmegaC2}}.  The domain $\Omega_{\ft C}$ will be constructed as a perturbation of $D_{\ft C}$,  using an argument related to the
genericity of Morse functions, and more precisely, a method due to
Ren\'e Thom based on
Sard's theorem. All the details will be given in the next section. More precisely,
we apply Proposition~\ref{Lau1}, which is stated in the
next section, with    $D=D_{\ft C}$, $\mathcal V_-=\ft C$,  $\mathcal
V_+=\ft V_{\ft C}$, and, for a $\delta \in (0,\delta_0)$ (see~\eqref{eq.PC-dc}):
\begin{itemize}
\item[(i)] $S_1=\pa D_{\ft C} \cap  U^{\delta/2}_{\pa \Omega}$, which
  is such that  $f : \overline{S_1} \to \mathbb R$
  is a Morse function with no critical point on $\pa S_1$
  (see~\eqref{eq.PC-dc} together with the fact that $\pa \ft C \cap
  \pa \Omega$ is the union  of non degenerate critical points of $f|_{\pa \Omega}$),
\item[(ii)]  $S_1'=\pa D_{\ft C} \cap  U^{\delta/4}_{\pa \Omega}$ which satisfies, according to~\eqref{eq.PC-dc}, $\vert \nabla_T f\vert \neq 0$ on $\overline{S_1\setminus S_1'}$. 
\end{itemize}
Therefore,   using in addition the fact that $D_{\ft C}$ satisfies~\eqref{G-DC0}--\eqref{G-DC},    there exists  a $C^\infty$ connected   open   subset  $\Omega_{\ft C}$ of $\Omega$ such that $\ft C\subset \Omega_{\ft C}$, $\overline{\Omega_{\ft C}}\subset \ft V_{\ft C}$, 
  $$f: \pa \Omega_{\ft C} \to \mathbb R \text{ is a Morse function},$$ 
  and for some $r>0$ and $\Gamma_{\ft C}\subset \Omega$, 
 \begin{equation}\label{eq.GAMMAC}
\pa \Omega_{\ft C}=  \left (\,  \bigcup_{z\in \pa \ft C\cap \pa \Omega}  B_{\pa \Omega}(z,\ve_z/2)    \right)\bigcup \,\overline{ \Gamma_{\ft C}}, \,  \text{ where, }  \, f\ge  \lambda+r \, \text{ on } \, \overline{\Gamma_{\ft C}}.  
 \end{equation}
\end{itemize}
%%%%%%%%%%%%%%%%%%
It then remains to check that  $\Omega_{\ft C}$ satisfies~\eqref{eq.OmegaC}. 
From~\eqref{eq.GAMMAC} and~\eqref{eq.GAMMAC0}, $\Omega_{\ft C}$ satisfies the two first statements in~\eqref{eq.OmegaC} and $\min_{\pa \Omega_{\ft C}}f=\lambda$. 
Since  $\ft C\subset \Omega_{\ft C}$ and $\overline{\Omega_{\ft C}}\subset \ft V_{\ft C}$, one deduces from the first statement in~\eqref{eq.vC-s2}, that 
$$\big \{ x\in  \overline{\Omega_{\ft C}},\,  f(x)<\lambda \big \}=\ft C,$$
and from~\eqref{eq.vC-s},   
$$\vert\nabla f\vert \neq 0 \,  \text{ on } \,  \overline{\Omega_{\ft C}}\setminus \ft C.$$
This proves that   $\Omega_{\ft C}$ satisfies the two last statements in~\eqref{eq.OmegaC}. 
This concludes the construction of a domain $\Omega_{\ft C}$ which satisfies~\eqref{eq.OmegaC} and~\eqref{eq.OmegaC2}. 
 A schematic representation of such a domain $ {\Omega_{\ft C}}$ is given on Figure~\ref{fig:oo}.  
\medskip

\noindent
\textbf{Step 2}: End of the proof of Theorem~\ref{thm.2}.
\medskip

\noindent
 For all $z \in \pa \ft C\cap \pa \Omega$, let   $\Sigma_{z}$  be an open subset   of~$\pa \Omega$ such that~$ z\in \Sigma_{z}$. Let $K$ be a compact subset of~$\Omega$ such that $K\subset \mathcal A(\ft C)$. 
Let us first consider the case when~$K\subset \ft C$. 
\medskip

Let $\Omega_{\ft C}$ be the $C^\infty$ subdomain of $\Omega$
constructed in the previous step and which, we recall, contains $\ft
C$ and satisfies~\eqref{eq.OmegaC} and~\eqref{eq.OmegaC2}. Then, one
easily deduces  that when~$\Omega$ is replaced by~$\Omega_{\ft C}$,
the function $f: \overline{\Omega_{\ft C}}\to \mathbb R$
satisfies~\eqref{H-M} and $\mathcal C=\{\ft C\}$
(see~\eqref{mathcalC-def} for the definition of $\mathcal C$). Thus,
in this case   $\ft C_{\ft{max}}=\ft C$. Moreover, using in addition
the second and third statements in~\eqref{eq.OmegaC},   one obtains
that the assumptions~\eqref{eq.hip1},~\eqref{eq.hip2},~\eqref{eq.hip3}
and~\eqref{eq.hip4}  are satisfied for the function $f:
\overline{\Omega_{\ft C}}\to \mathbb R$. Thus, according to
Theorem~\ref{thm.main} applied to the function $f:
\overline{\Omega_{\ft C}}\to \mathbb R$,  the most probable places of
exit of the process~\eqref{eq.langevin} from $\Omega_{\ft C}$ when
$X_0=x\in \ft C$,  are $\pa \ft C \cap \pa \Omega$, and the relative
asymptotic probabilities to exit through each point in $\pa \ft C \cap \pa \Omega$ are given by  item 2 in Theorem~\ref{thm.main}. In particular, from items~1 and~3 in Theorem~\ref{thm.main}, for any open subset $\Sigma$ of $\pa \Omega_{\ft C}$ such that 
$$\min_{\overline \Sigma} f>\min_{\pa \Omega_{\ft C}} f \ \, \text{(where we recall $\argmin_{\pa \Omega_{\ft C}} f =\pa \ft C\cap \pa \Omega$, see~\eqref{eq.OmegaC})},$$
there exists $c>0$ such that for $h$ small enough:
\begin{equation}\label{eq.item13}
 \sup_{x\in K} \mathbb P_x \big[  X_{\tau_{\Omega_{\ft C}}}  \in  \Sigma \big]\le e^{-\frac ch},
\end{equation}
where $\tau_{\Omega_{\ft C}}$ is  the first exit time from $\Omega_{\ft C}$ of the process~\eqref{eq.langevin}. 
\medskip

\noindent
\textbf{Step 2a}: Proof of the   first asymptotic estimate in Theorem~\ref{thm.2} when $K\subset \ft C$. 
\medskip

\noindent
Writing $ \pa \Omega =( \pa \Omega \cap \pa \Omega_{\ft C}  )\cup  ( \pa \Omega \setminus \pa \Omega_{\ft C} )$, it holds:
\begin{equation}\label{eq.utdec}
  \Big (\pa \Omega \setminus \bigcup_{z\in \pa\ft C\cap \pa \Omega } \Sigma_z\Big)  \   \subset \  \left (\pa \Omega_{\ft C} \cap \pa \Omega \setminus  \bigcup_{z\in \pa\ft C\cap \pa \Omega } \Sigma_z  \right) \   \bigcup   \   \big (  \pa \Omega \setminus  \pa \Omega_{\ft C}   \big ) . 
  \end{equation}
To prove the first asymptotic estimate in Theorem~\ref{thm.2}, let us prove that when $X_0=x\in K$, the probabilities that   $X_{\tau_\Omega}$ belongs to each of the two sets in the right-hand side of~\eqref{eq.utdec} are exponentially small when $h\to 0$. 
Let us recall that~$\tau_{\Omega_{\ft C}}$ is  the first exit time from $\Omega_{\ft C}$ of the process~\eqref{eq.langevin} and thus,  when  $X_0=x\in \Omega_{\ft C}$, $\tau_{\Omega_{\ft C}}\le \tau_{\Omega}$, and  
\begin{equation}\label{eq.tauC}
\text{$\tau_{\Omega_{\ft C}}=\tau_{\Omega}$ if and only if  $X_{\tau_{\Omega_{\ft C}}}\in \pa \Omega_{\ft C}\cap  \pa \Omega$.}
\end{equation} 
Thus, from~\eqref{eq.tauC}, when $X_0=x\in  \Omega_{\ft C}$,  it holds:
\begin{align*}
\Big\{ X_{\tau_\Omega}\in \pa \Omega_{\ft C} \cap \pa \Omega \setminus  \cup_{z\in \pa\ft C\cap \pa \Omega } \Sigma_z   \Big\} \subset  \big \{X_{\tau_{\Omega_{\ft C}}}  \in    \pa \Omega_{\ft C}\setminus \pa \Omega  \} \cup \Big \{ X_{\tau_{\Omega_{\ft C} }}  \in  \pa \Omega_{\ft C} \cap \pa \Omega \setminus  \cup_{z\in \pa\ft C\cap \pa \Omega } \Sigma_z   \Big\}.
\end{align*}
Using~\eqref{eq.item13}, there exists $c>0$ such that for $h$ small enough: 
$$\sup_{x\in K} \mathbb P_x\Big [  X_{\tau_{\Omega_{\ft C} }}  \in  \pa \Omega_{\ft C} \cap \pa \Omega \setminus  \bigcup_{z\in \pa\ft C\cap \pa \Omega } \Sigma_z    \Big]\le e^{-\frac ch}$$
and
  \begin{equation}\label{eq.oubli-11}
\sup_{x\in K} \mathbb P_x\Big [  X_{\tau_{\Omega_{\ft C}}}  \in    \pa \Omega_{\ft C}\setminus \pa \Omega   \Big]\le e^{-\frac ch}.
\end{equation}
Thus, there exists $c>0$ such that for $h$ small enough: 
\begin{equation}\label{eq.ut}
\sup_{x\in K} \mathbb P_x\Big [  X_{\tau_{\Omega }}  \in  \pa \Omega_{\ft C} \cap \pa \Omega \setminus  \bigcup_{z\in \pa\ft C\cap \pa \Omega } \Sigma_z    \Big]\le e^{-\frac ch}.
\end{equation}
Let us now consider the case when $X_{\tau_\Omega}\in   \pa \Omega \setminus  \pa \Omega_{\ft C} $. When $X_0=x\in K$, it holds from~\eqref{eq.tauC}:
$$\big\{X_{\tau_\Omega}\in   \pa \Omega \setminus  \pa \Omega_{\ft C} \big\}\subset \big\{ X_{\tau_{\Omega_{\ft C}}}  \in   \pa \Omega_{\ft C}\setminus \pa \Omega   \big\}.$$
Therefore, from~\eqref{eq.oubli-11}, there exists $c>0$ such that for $h$ small enough: 
  \begin{equation}\label{eq.ut2}
\sup_{x\in K} \mathbb P_x\Big [  X_{\tau_{\Omega }}  \in    \pa \Omega \setminus  \pa \Omega_{\ft C}  \Big]\le e^{-\frac ch}.
\end{equation}
In conclusion, from~\eqref{eq.utdec},~\eqref{eq.ut} and~\eqref{eq.ut2}, one obtains that there exists $c>0$ such that for $h$ small enough: 
\begin{equation}\label{utili1}
\sup_{x\in K} \mathbb P_x\Big [X_{\tau_{\Omega}}\in  \pa \Omega
\setminus \bigcup_{z\in \pa\ft C\cap \pa \Omega } \Sigma_z  \Big]\le e^{-\frac ch}.
\end{equation}
This proves the first asymptotic estimate in Theorem~\ref{thm.2} when $K\subset\ft  C$. 
\medskip

\noindent
\textbf{Step 2b}: Proof of the second  asymptotic estimate in Theorem~\ref{thm.2} when $K\subset\ft  C$.
\medskip

\noindent 
Let us assume   that   the open sets $(\Sigma_{z})_{z\in \pa \ft C\cap \pa \Omega}$  are pairwise disjoint. Let us consider  $z\in \pa \ft C\cap \pa \Omega$ and $\beta>0$ such that (see indeed the second statement in~\eqref{eq.OmegaC}), 
\begin{equation}\label{eq.bomegac}
 B_{\partial \Omega}(z, \beta)  \subset   \Sigma_z  \cap \pa \Omega_{\ft C}.
 \end{equation}
Then, one writes:
\begin{equation}\label{eq.dec-o}
\mathbb P_x  [X_{\tau_{\Omega}}\in \Sigma_z   ]=\mathbb P_x\big [X_{\tau_{\Omega}}\in       B_{\partial \Omega}(z, \beta)  \big]+ \mathbb P_x\big [X_{\tau_{\Omega}}\in    \Sigma_z \setminus    B_{\partial \Omega}(z, \beta)  \big].
 \end{equation}
Let us first deal with the second term in the right-hand side of~\eqref{eq.dec-o}. It holds (since the sets $(\Sigma_{y})_{y\in \pa \ft C\cap \pa \Omega}$  are pairwise disjoint and $B_{\partial \Omega}(z, \beta)  \subset   \Sigma_z  $, see~\eqref{eq.bomegac}), when $X_0=x\in \Omega$:
$$
\mathbb P_x\big [X_{\tau_{\Omega}}\in    \Sigma_z \setminus    B_{\partial \Omega}(z, \beta)  \big]\le \mathbb P_x\left [X_{\tau_{\Omega}}\in    \pa \Omega  \setminus  \Big(    B_{\partial \Omega}(z, \beta) \cup \bigcup_{y\in \pa\ft C\cap \pa \Omega,y\neq z } \Sigma_y\Big)   \right].
$$
Thus, from~\eqref{utili1} (applied with    $B_{\partial \Omega}(z, \beta) $ instead  of   $\Sigma_z$), one obtains   that   there exists $c>0$ such that for $h$ small enough: 
\begin{equation}\label{eq.step2-ch}
\sup_{x\in K} \mathbb P_x\Big [X_{\tau_{\Omega}}\in  \Sigma_z \setminus    B_{\partial \Omega}(z, \beta)     \Big]\le e^{-\frac ch}.
 \end{equation}
Let us now deal with the first term in the right-hand side~\eqref{eq.dec-o}. It holds from~\eqref{eq.bomegac}, when $X_0=x\in K$:
\begin{align} 
\nonumber
\mathbb P_x\big [X_{\tau_{\Omega}}\in       B_{\partial \Omega}(z, \beta)  \big]&=\mathbb P_x\big [X_{\tau_{\Omega}}\in       B_{\partial \Omega}(z, \beta),\tau_{\Omega_{\ft C}  }<\tau_{\Omega}  \big]+\mathbb P_x\big [X_{\tau_{\Omega_{\ft C}} }\in       B_{\partial \Omega}(z, \beta),
 \tau_{\Omega_{\ft C}} =\tau_{\Omega} \big]\\
 \label{eq.step2-ch2}
 &=O(e^{-\frac ch})+ \mathbb P_x\big [X_{\tau_{\Omega_{\ft C}} }\in       B_{\partial \Omega}(z, \beta)],
%\mathbb P_x\big [X_{\tau_{\Omega_{\ft C}}}\in       B_{\partial \Omega}(z, \beta)  \big].
 \end{align}
 where we used the fact that $\{\tau_{\Omega_{\ft C}  }<\tau_{\Omega}\}\subset \{X_{\tau_{\Omega_{\ft C}}}  \in   \pa \Omega_{\ft C}\setminus \pa \Omega \}$ (see~\eqref{eq.tauC}) and~\eqref{eq.oubli-11}. 
\begin{figure}[h!]
\begin{center}
\begin{tikzpicture}[scale=0.87]
\tikzstyle{vertex}=[draw,circle,fill=black,minimum size=5pt,inner sep=0pt]
\tikzstyle{ball}=[circle, dashed, minimum size=1cm, draw]
\tikzstyle{point}=[circle, fill, minimum size=.01cm, draw]
\draw [rounded corners=10pt] (0.1,0.5) -- (-1,2.5) -- (1,5) -- (5,6.5) -- (7.6,3.75) -- (6,1) -- (4,0) -- (0.2,-0.3) --cycle;
\draw [thick, densely dashed, rounded corners=10pt] (1.9,-0.24) -- (.5,1.5) -- (0,2.5) -- (-0.05,3.9) -- (2.2,3.5) -- (2.7,2.7) -- (2,2) -- (3,1.5) --cycle;
\draw [thick, densely dashed,rounded corners=10pt]    (2.75,4.2)  -- (3.4,4.4) -- (4,3) --(5.5,2.5)--(6.5,4) --(6.5,5)  -- (5,6) -- (3,5)  --cycle;
 \draw (1.6,1.5) node[]{$\ft C_2$};
  \draw (5,4) node[]{$\ft C_{\ft{max}}$};
     \draw  (5,1.3) node[]{$\Omega$};
    \draw  (6.7,1.5) node[]{$\pa \Omega$};

\draw [very thick,black] (0.44,4.3)--(-0.21,3.5);
\draw [very thick,black] (1.2,-0.2)--(1.9,-0.14) ; 
\draw [very thick,black] (2.6,-0.11)--(1.9,-0.14) ;

\draw [very thick,black]  (-0.21,3.5) ..controls (-0.4,3)  and (0.28,0.8)  .. (0.45,0.6) ;
\draw [very thick,black]  (2.6,-0.11) ..controls (2.9,0.05)  ..  (3.5,1);
\draw [very thick,black]  (0.44,4.3) ..controls (1.2,4.8)  .. (2.9,3.3);
\draw [very thick,black] (3.5,1)  ..controls (4,2)  .. (2.9,3.3);
\draw [very thick,black] (0.45,0.6)  ..controls (0.9,-0.1)  .. (1.2,-0.2);

     \draw  (-0.1,4.2) node[]{$z_3$};
\draw (0.04,3.8) node[vertex,label=north west: {}](v){};
\draw (1.4,2.5) node[vertex,label=north west: {$x_2$}](v){};
\draw (4.3 ,4.4) node[vertex,label=north west: {$x_1$}](v){};
\draw (1.9,-0.14) node[vertex,label=south west: {$z_2$}](v){};
\draw (6.2,5.2) node[vertex,label=north east: {$z_1$}](v){};
     \draw  (3.4,2) node[]{$ \Omega_{\ft C_2}$};

       \draw [<->,very thick, densely dotted] (-0.2,4.9)--(-1,4) ;
  \draw  (-1,4.6) node[]{$\ve_{z_3}$};

\draw [<->,very thick, densely dotted] (2.6,-0.8)--(1.2,-0.9) ;
  \draw  (2,-1.2) node[]{$\ve_{z_2}$};
     
     \draw[ultra thick] (9.8,2)--(10.8,2);
   \draw (11.4,2)node[]{$\pa \Omega_{\ft C_2}$};

\end{tikzpicture}

\caption{Schematic representation  of $\Omega_{\ft C_2}$ satisfying~\eqref{eq.OmegaC} when $\ft C=\ft C_2$. On the figure, $\pa \ft C_2 \cap \pa \Omega=\{z_2,z_3\}$, $x_2$ is the global minimum of $f$ in $\ft C_2$ and $\ft C_{\ft{max}}$ is another element of~$\mathcal C$. }
 \label{fig:oo}
 \end{center}
\end{figure}
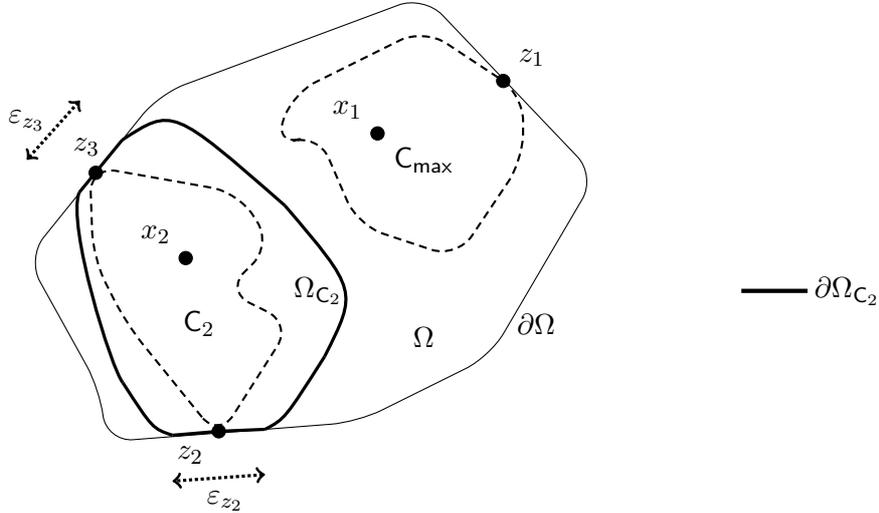
\noindent
Applying item 2 in Theorem~\ref{thm.main} with the function $f: \overline{\Omega_{\ft C}}\to \mathbb R$ and $F=\mbf{1}_{B_{\partial \Omega}(z, \beta)}$, one has:
$$\mathbb P_{x}[X_{\tau_{\Omega_{\ft C}}} \in  B_{\partial \Omega}(z, \beta) ]= \frac{  \partial_nf(z)      }{  \sqrt{ {\rm det \ Hess } f|_{\partial \Omega}   (z) }  } \left (\sum \limits_{y\in\pa \ft C \cap \pa \Omega  } \frac{  \partial_nf(y)      }{  \sqrt{ {\rm det \ Hess } f|_{\partial \Omega}   (y) }  }\right)^{-1}(1+O(h)),$$
in the limit $h \to 0$ and uniformly in~$x \in K$. 
 Together with~\eqref{eq.dec-o},~\eqref{eq.step2-ch}, and~\eqref{eq.step2-ch2}, this concludes the proof of the second asymptotic estimate in Theorem~\ref{thm.2} for initial conditions  $X_0=x\in   K\subset \ft C$ and when $F=1$ on $\pa \Omega$. To extend the second asymptotic estimate in Theorem~\ref{thm.2} to functions $F\in L^\infty(\pa \Omega,\mathbb R)$ which are smooth in a neighborhood of $z$ in $\pa \Omega$, one uses the same procedure as the one made in the second step in Section~\ref{sec:end_proof}. 
 \medskip
 
Finally, the case when $X_0=x\in K\subset \mathcal A( \ft C)$ is
proved using the estimate of  Freidlin and
Wentzell~\eqref{eq.WFr-est}, as in the second step of the proof of Proposition~\ref{pr.exp-qsd-x}. This concludes the proof of Theorem~\ref{thm.2}.
\end{proof}

%%%%%%%%%%%%%%%%%%%%%%%%%%%%%%%%%%%%%%%%%%%
\subsection{Proof of the existence of a domain $\Omega_{\ft C}$ satisfying (\ref{eq.OmegaC2})}
\label{sec.OMEGAC-morse}

In this section, we prove the existence of a domain $\Omega_{\ft C}$  which satisfies~\eqref{eq.OmegaC2} in addition to~\eqref{eq.OmegaC}.  To this end, we first  give in Proposition~\ref{pr.Morsesigma} a simple perturbation result to present the main idea of the proof. Then, we extend this result to the setting we are interested in to prove   the existence of such a domain $\Omega_{\ft C}$ in  Proposition~\ref{Lau1}.

\begin{proposition}\label{pr.Morsesigma}
Let $f:\mathbb R^{d}\to \mathbb R$ be a $C^\infty$ function and 
$D$ be a $C^\infty$ open  bounded and connected  subset of $\mathbb R^d$. Let us assume that 
$$\forall x\in \pa D, \ \nabla f(x)\oplus T_x\pa D=\mathbb R^d.$$
Then, for any open sets $ \mathcal V_-$ and $ \mathcal V_+$ such that $\overline{ \mathcal V_-} \subset D \text{ and } \overline{ D} \subset \mathcal V_+$,    
there exists a $C^\infty$ open bounded  and connected  subset $D'$ of~$\mathbb R^d$ such that 
$$\overline{ \mathcal V_-} \subset D', \ \overline{ D'} \subset \mathcal V_+, \ \text{ and } \ f|_{\pa D'} \text{  is a Morse function}.$$
 
\end{proposition}

\begin{remark}
We are thankful to Fran\c cois Laudenbach who gave us the main ingredient of the proof of Proposition~\ref{pr.Morsesigma}. The proof is inspired  by a method due to  Ren\'e Thom~\cite{thom} based on Sard's theorem~\cite{sard}, see~\cite[Section 5.6]{laudenbach2011transversalite}. 
\end{remark}

\begin{proof}
Let   $ \mathcal V_-$ and $ \mathcal V_+$ be two open subsets of $\mathbb R^d$ such that $\overline{ \mathcal V_-} \subset D \text{ and } \overline{ D} \subset \mathcal V_+$. 
Let us denote by  $S$ the boundary of $D$ which is a  smooth compact hypersurface of $\mathbb R^d$. 
 For $r>0$, one denotes    by $B(0,r) $ the ball of radius $r$ centred at $0$ in $\mathbb R^d$.  Let   $\mathcal V$ be a neighborhood of $S$
in $\mathbb R^{d}$.  
 By assumption on $S$, there exist  $\varepsilon_{0}>0$ and $\varepsilon_{1}>0$ such that the map
$$   (x,\lambda)\in  S\times (-\varepsilon_{0},\varepsilon_{0})     \mapsto x+\lambda
\nabla f(x)\in \mathcal V
$$
is well defined and is a diffeomorphism onto its image, and, for all $ (x,v) \in S\times B(0,\varepsilon_{1}) $, there exists a unique $\lambda(x,v)\in  
(-\varepsilon_{0},\varepsilon_{0})$ such that 
$$
   f\big(x+\lambda(x,v)\nabla f(x)\big) = f(x)+  v\cdot x\,.
$$
Moreover, for every $v\in B(0,\varepsilon_{1}) $, according to the implicit function theorem, the map $x\in S\mapsto \lambda(x,v)\in (-\varepsilon_{0},\varepsilon_{0})$
is smooth and then also is $x\in S\mapsto x+ \lambda(x,v)\nabla
f(x)\in \mathbb R^{d}$.
The latter application is then an injective immersion and hence, since $S$ is compact,
it follows that $S_{v}:=\{x+ \lambda(x,v)\nabla f(x)\}$ is a smooth compact   hypersurface. Up to choosing $\ve _1>0$ smaller, for any $v\in B(0,\ve_1)$, $S_{v}$ is the boundary of a $C^\infty$ open bounded  and connected  subset $D_v$ of~$\mathbb R^d$ such that 
$$\overline{ \mathcal V_-} \subset D_v \, \text{ and }  \, \overline{ D_v} \subset \mathcal V_+.$$
To prove Proposition~\ref{pr.Morsesigma},  it remains to show  that there exists   $v\in B(0,\varepsilon_{1})$
such that $f|_{S_{v}}$ is a Morse function. 
Let us   introduce
the function
$$
F:(x,v)\in S\times B(0,\varepsilon_{1})\mapsto  f|_{S_{v}}\big(x+\lambda(x,v)\nabla f(x)\big) =f(x)+  v\cdot x
\in \mathbb R\,.
$$
For all $x\in  S$ and for all $v\in B(0,\varepsilon_{1})$, let   $v^{T}_{x}\in T_xS$ and $v_{x}^{N}\in \mathbb R$ be such that    
\begin{equation}\label{vtt}
v=v^{T}_{x} + v_{x}^{N} n(x),
\end{equation}
where  we recall that $n(x)$ is the unit outward  normal vector to $D$ at $x\in  \pa D$. 
At  $(x,v)\in S\times B(0,\varepsilon_{1})$,   it holds  $\partial _{x}F (x,v): z \in T_xS\mapsto d_xf(x)z+v_x^T\cdot z$, where $\partial_xF(x,v)$ is the $x$-derivative of $F$ at $(x,v)$. 
The function
$G: S\times B(0,\varepsilon_{1})\to T_x^*S $ defined
by
$$G: (x,v)\mapsto (x, \partial_{x}F (x,v))$$
 is a submersion onto a small tube 
around the zero section of $T^*S$. 
This is obvious by considering the $v$-derivative of $G$. 
Hence, $G$ is transverse to the zero section $0_{T^{*}S}$ of $T^{*}S$ (see~\cite[Chapitre~5.1]{laudenbach2011transversalite} for the definition of transversality). Using the parametric transversality theorem   (which is a consequence of Sard's theorem, see for instance~\cite[Chapitre~5.3.1]{laudenbach2011transversalite}), one obtains that for almost every   $v\in B(0,\varepsilon_{1})$,
$\partial_{x}(F|_{S\times \{v\}})=d (f|_{S_{v}})$ is transverse  to $0_{T^{*}S}$,  which is equivalent to say that~$f|_{S_{v}}$ is a Morse function. 
This concludes the proof of Proposition~\ref{pr.Morsesigma}. 
\end{proof}
\noindent
The next proposition gives sufficient conditions on $D$ and $f$ to modify the result of Proposition~\ref{pr.Morsesigma} so that the perturbed domain $D'$ has the same boundary as $D$ on a prescribed subset $S_1'$ of $\partial D$ on which $f$ is already a Morse function.

\begin{proposition}\label{Lau1}
Let $f:\mathbb R^{d}\to \mathbb R$ be a $C^\infty$ function and 
$D$ be a $C^\infty$ open  bounded and connected  subset of $\mathbb R^d$. Let us assume that 
$$\forall x\in \pa D, \ \nabla f(x)\oplus T_x\pa D=\mathbb R^d.$$
Furthermore, let us assume that there exists  an open subset $S_1$ of $\pa D$ such that  $f: \overline {S_1}\to \mathbb R$ is a Morse function with no critical point on $\pa   S_1$.
Let us now consider an open set  $S_1'$ such that $\overline{S_1'}  \subset S_1$ and  $ f|_{\pa D}$ has no critical point on  $\overline{S_1\setminus S_1'}$.  
Then, for any open sets $ \mathcal V_-$ and $ \mathcal V_+$ such that $\overline{ \mathcal V_-} \subset D\cup S_1' \text{ and } \overline{ D} \setminus S_1' \subset \mathcal V_+$,    
there exists a $C^\infty$ open bounded  and connected  subset $D'$ of~$\mathbb R^d$ such that  $S_1'\subset \pa D'$, 
$$\overline{ \mathcal V_-} \subset D'\cup S_1', \ \overline{ D'}  \setminus S_1' \subset \mathcal V_+, \ \text{ and } \, f|_{\pa D'} \text{  is a Morse function}.$$
\end{proposition}

 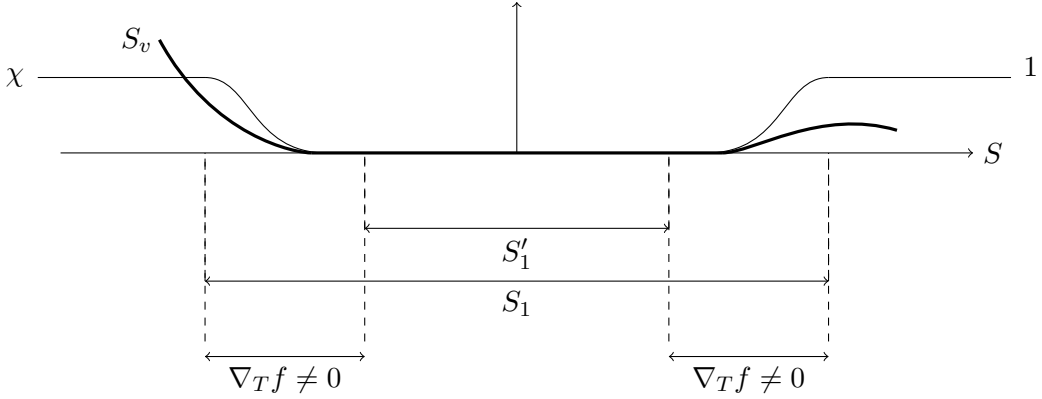
\begin{figure}[h!]
  \begin{center}
 \begin{tikzpicture}
\draw[->] (-6,0)--(6,0) node[right] {$S$} ;
 \draw[->] (0,0)--(0,2);
\draw (6.99 ,1.4) node [anchor=north east] {1};
\draw[dashed] (-2,-1)--(-2,0);
\draw[dashed] (2,-1)--(2,0);
\draw[dashed] (-4.1,-1.7)--(-4.1,0);
\draw[dashed] (4.1,-1.7)--(4.1,0);
\draw[<->](-2,-1)--(2,-1)  node[midway,below] {$S_1'$} ;
\draw[very thick] (-4.7,1.5) ..controls (-4,0.2) and  (-2.9,0) .. (-2.7,0);
\draw[very thick] (5,0.3) ..controls (4,0.6) and  (3.1,0) .. (2.7,0);
   \draw  (-5,1.5) node[]{$S_v$};
\draw[very thick](-2.7,0)--(2.7,0);
\draw[<->](-4.1,-1.7)--(4.1,-1.7)  node[midway,below] {$S_1$} ;
\draw (-4.1,1) ..controls (-3.5,1) and  (-3.5,0) .. (-2.5,0);
\draw (4.1,1) ..controls (3.5,1) and  (3.5,0) .. (2.5,0);
\draw (-6.3,1) -- (-4.1,1);
\draw (6.5,1) -- (4.1,1);
\draw[dashed] (-4.1,-2.5)--(-4.1,0);
\draw[dashed] (-2,-2.5)--(-2,0);
\draw[<->] (-4.1,-2.7)--(-2,-2.7) node[midway,below] {$\nabla_Tf\neq 0$} ;
\draw[dashed] (4.1,-2.5)--(4.1,0);
\draw[dashed] (2,-2.5)--(2,0);
\draw[<->] (4.1,-2.7)--(2,-2.7) node[midway,below] {$\nabla_Tf\neq 0$} ;

   \draw  (-6.6,1) node[]{$\chi$};
\end{tikzpicture}
\caption{The support of $\chi$ on $S$, the compact sets $S_1$ and  $S_1'$, and the hypersurface $S_v$.}
 \label{fig:chii}
  \end{center}
\end{figure}

\begin{proof}
Let   $ \mathcal V_-$ and $ \mathcal V_+$ be two open subsets of $\mathbb R^d$ such that $\overline{ \mathcal V_-} \subset D\cup S_1' \text{ and } \overline{ D} \setminus S_1' \subset \mathcal V_+$.  
Let us denote by  $S$ the boundary of $D$. The submanifold $S$ is a  smooth compact hypersurface of $\mathbb R^d$. 
 Let us introduce a function $\chi\in C^\infty(S)$ such that $\chi(x)=1$ for all   $x\in S\setminus S_1$ and  $\chi(x)=0$ for all   $x\in \mathcal V_{S_1'}$ where $\mathcal V_{S_1'}$ is an open neighborhood of $\overline{S_1'}$ in $S$ such that $\overline{\mathcal V_{S_1'}}\subset S_1$. To prove Proposition~\ref{Lau1}, one  uses the cutoff function $\chi$  in the definition of $\lambda(x,t)$   to ensure that $S_1'\subset S_v$ (see the proof of Proposition~\ref{pr.Morsesigma} for the notation $S_v$). This is made as follows. 
Let us first consider   $\varepsilon_{0}>0$ and $\varepsilon_{1}>0$ such that the map
$$   (x,\lambda)\in  S\times (-\varepsilon_{0},\varepsilon_{0})     \mapsto x+\lambda
\nabla f(x)\in \mathcal V
$$
is well defined and is a diffeomorphism onto its image, and, for all $ (x,v) \in S\times B(0,\varepsilon_{1}) $, there exists a unique $\lambda(x,v)\in  
(-\varepsilon_{0},\varepsilon_{0})$ such that 
$$
   f\big(x+\lambda(x,v)\nabla f(x)\big) = f(x)+  \chi(x)\, v\cdot x\,.
$$
Notice that $\lambda(x,v)=0$ for all $x\in \mathcal V_{S_1'}$ and $v\in B(0,\varepsilon_{1})$ (since $\chi=0$ on $\mathcal V_{S_1'}$). Thus, for all $v\in B(0,\varepsilon_{1})$,  $ \mathcal V_{S_1'}\subset S_v$ which implies that $S_1'\subset  S_v$. 
Again, $S_{v}:=\{x+ \lambda(x,v)\nabla f(x)\}$ is a smooth compact  hypersurface. 
A schematic representation of the function $\chi$ and the hypersurface $S_v$ are given in Figure~\ref{fig:chii}. 
Up to choosing $\ve _1>0$ smaller, for any $v\in B(0,\ve_1)$, $S_{v}$ is the boundary of a $C^\infty$ open bounded  and connected  subset $D_v$ of~$\mathbb R^d$ such that, since  $\mathcal V_{S_1'}\subset S_v$,
$$ \overline{ \mathcal V_-} \subset D_v\cup S_1', \text{ and }  \overline{ D_v}  \setminus S_1' \subset \mathcal V_+.$$
 Let us now show that there exists   $v\in B(0,\varepsilon_{1})$
such that $f|_{S_{v}}$ is a Morse function. 
For that purpose, we    consider 
the function
$$
F:(x,v)\in S\times B(0,\varepsilon_{1})\mapsto  f|_{S_{v}}\big(x+\lambda(x,v)\nabla f(x)\big) =f(x)+ \chi(x)\,  v\cdot x
\in \mathbb R\,,
$$
and the function  
$G: S \times B(0,\varepsilon_{1})\to T_x^*S $ defined
by $G: (x,v)\mapsto (x, \partial_{x}F (x,v))$.
  Notice that   for all  $v\in    B(0,\varepsilon_{1})$, $x\in \overline{S_1'}\mapsto F(x,v)=f(x)$ is already, by assumption,  a Morse function (with no critical point on $\pa S_1'$). 
   This implies that $G$ is  transverse to the zero section  $0_{T^{*}S}$ of $T^{*}S$ along $S_1'\times B(0,\varepsilon_{1})$.
  Thus,  to prove  Proposition~\ref{Lau1}, it remains to study the function   $x\in  S\setminus S'_1\mapsto F(x,v)$, for $v\in    B(0,\varepsilon_{1})$. 
For  $(x,v) \in \overline{S_1\setminus S'_1} \times B(0,\varepsilon_{1})$ and  for all $z\in T_xS$, it holds:
$$\partial_{x}F (x,v)z= d_xf(x)z+ O(\Vert v\Vert)\, z.$$ 
 Since by assumption $d_xf(x)\neq 0_{T^{*}_xS}$ for all $x$ belonging to the compact set $\overline{ S_1\setminus S_1'}$, one has, up to choosing $\ve_1>0$  smaller: for all $x\in \overline{S_1\setminus S_1'}$ and $v\in B(0,\varepsilon_{1})$, $\partial_{x}F (x,v)\neq 0_{T^{*}_xS}$.  
Finally, for  $(x,v) \in S\setminus S_1 \times B(0,\varepsilon_{1})$ and  for all $z\in T_xS$, it holds:
$$G(x,v)= (x,d_xf(x)z+  v_x^T\cdot z),$$
where $v_x^T$ is defined by~\eqref{vtt}.  
Thus, the function $G: S\setminus S_1\times B(0,\varepsilon_{1})\to T_x^*S $ is a submersion onto a small tube 
around the zero section $0_{T^{*}S}$ of $T^*S$. This implies that $G$ is  transverse to the zero section   of $T^{*}S$ along $S\setminus S_1\times B(0,\varepsilon_{1})$. 
 In conclusion, the function  $G: S \times B(0,\varepsilon_{1})\to T_x^*S $ is  transverse to the zero section of  $T^{*}S$.  The parametric transversality theorem  implies that for almost every   $v\in B(0,\varepsilon_{1})$,
$\partial_{x}(F|_{S\times \{v\}})=d (f|_{S_{v}})$ is transverse  to $0_{T^{*}S}$,  which is equivalent to~$f|_{S_{v}}$ is a Morse function. 
 This concludes the proof of     Proposition \ref{Lau1}. 
 \end{proof}

\subsection{Generalization of Theorems~\ref{thm.main} and~\ref{thm.2} }
\label{sec.gene_fin}
%\end{appendices}
 %%%%%%%%%%%%%%%%
 \noindent
 In view of the proof of Theorem~\ref{thm.2},
we have the following generalization of Theorems~\ref{thm.main} and~\ref{thm.2}.
%~\ref{thm.main} and~\ref{thm.2}. 
\begin{theorem}\label{thm.4}
Let us assume that  \eqref{H-M} holds. Let $\lambda\in \mathbb R$ and $\ft C_1,\ldots,\ft C_m\in \mathcal C$ be $m$ ($m\ge 1$) connected components of $\{f<\lambda\}$  such that:
\begin{equation}\label{eq.condition1}
\bigcup_{j=1}^m\overline{\ft C_j} \text{ is a connected component of } \{f\le \lambda\}
\end{equation}
%\comment{Pas de ssp autre que ceux dans les $\pa \ft C_i\cap \pa \ft C_j$ sinon on ne peut pas monter en energie ! (cf. $\ft V_\ft C$  et ensuite $D_\ft C$)}
and such that, up to reordering $\ft C_1,\ldots,\ft C_m$, 
$$\pa \ft C_1\cap \pa \Omega\neq \emptyset \  \text{ and } \ \forall j\in\{ 2,\ldots,m\}, \ \min_{ \ft C_1}f<\min_{   \ft C_j}f.$$
%Define  
%$$\{z_1,\ldots,z_{\ft n }\}:= \pa \Omega \cap \bigcup_{j=1}^m \pa \ft C_j    \subset \ft U_1^{\pa \Omega} \text{ with } \{z_1,\ldots,z_{\ft n_1 }\}:= \pa \Omega  \cap \pa \ft C_1 .$$
 Let  $F\in L^{\infty}(\partial \Omega,\mathbb R)$. For all $ j\in\{ 1,\ldots,m\}$ and $z\in \pa \ft C_j\cap \pa \Omega$, let  $\Sigma_{z}$ be an  open subset of~$\pa \Omega$ such that  $z\in \Sigma_z$ and such that the $\Sigma_{z}$'s are pairwise disjoint.  
    Let~$K$ be a compact subset of $\Omega$ such
  that $ K\subset \mathcal A(\ft C_{1})$ and  $x\in K$. 
    Then:
  \begin{enumerate} 
  \item 
There exists $c>0$ such that in the limit  $h\to 0$:
\begin{equation} \label{eq.t1'}
\mathbb E_x \left [ F\left (X_{\tau_{\Omega}} \right )\right]=\sum \limits_{z\in    \cup_{j=1}^m \pa \Omega \cap \pa \ft C_j  } 
\mathbb E_x \left [( \mathbf{1}_{\Sigma_{z}}F)\left (X_{\tau_{\Omega}} \right )\right]  +O\big (e^{-\frac ch}\big )
 \end{equation}
 and
 \begin{equation} \label{eq.t2'}
\sum \limits_{z\in    \cup_{j=2}^m \pa \Omega \cap \pa \ft C_j  } 
\mathbb E_x \left [ (\mathbf{1}_{\Sigma_{z}}F)\left (X_{\tau_{\Omega}} \right )\right] =O\big (h^{\frac14} \big ),
 \end{equation}
 uniformly in $x\in K$.% (with the convention $\sum_{\emptyset}=0$). 
 \item When for some $z\in \pa \ft C_1\cap \pa \Omega$ the function  $F$ is $C^{\infty}$ in a neighborhood  of $z$, one has when $h\to 0$:
\begin{equation} \label{eq.t3'}
 \mathbb E_x \left [ (\mathbf{1}_{\Sigma_{z}}F)\left (X_{\tau_{\Omega}} \right )\right]= \frac{ F(z)\,  \partial_nf(z)      }{  \sqrt{ {\rm det \ Hess } f|_{\partial \Omega}   (z) }  } \left (\sum \limits_{y\in\pa \ft C_1 \cap \pa \Omega  } \frac{  \partial_nf(y)      }{  \sqrt{ {\rm det \ Hess } f|_{\partial \Omega}   (y) }  }\right)^{-1}  \!\!\!\!\!\!+ O(h^{\frac14}),
\end{equation}
uniformly in~$x \in K$.
% where\label{page.ai}
% \begin{equation} \label{ai}
% a_i=\frac{  \partial_nf(z_i)      }{  \sqrt{ {\rm det \ Hess } f|_{\partial \Omega}   (z_i) }  } \left (\sum \limits_{j=1}^{\ft k_1^{\pa \ft C_{\ft{max}} }} \frac{  \partial_nf(z_j)      }{  \sqrt{ {\rm det \ Hess } f|_{\partial \Omega}   (z_j) }  }\right)^{-1}. \end{equation}
  \item 
When $m=1$, 
%the remainder term  $O( h^{\frac14})$ in \eqref{eq.t2'}
 % is   of the order  $O\big (e^{-\frac{c}{h}}\big )$ for some $c>0$  
 %and
 the remainder term  $O\big (h^{\frac14}\big )$  in \eqref{eq.t3'} is of   order $O(h)$  and admits a full asymptotic expansion in~$h$. 
\end{enumerate} 

\end{theorem}
\noindent
Theorem~\ref{thm.4} is a generalization of Theorem~\ref{thm.main} because we do not assume that $\pa \ft C_1\cap \pa \Omega\subset \argmin_{\pa \Omega}f$. 
Theorem~\ref{thm.4} is also a generalization of Theorem~\ref{thm.2}  since, in the framework of Theorem~\ref{thm.4}, when $m\ge 2$,  $\vert \nabla f \vert =0$ on $\pa \ft C_1\cap\pa   \ft C_j$, for $j=2,\ldots,m$, and thus Theorem~\ref{thm.2}   with $\ft C=\ft C_1$ does not apply. 
%~\ref{thm.main} and~\ref{thm.2}
%Theorem~\ref{thm.4} is a generalization of Theorem~\ref{thm.main} because it is not assumed   in Theorem~\ref{thm.4} that 
\begin{proof}
Let us denote by $\widetilde{\ft C}=\ft C_1\cup\ldots\cup \ft  C_m$. 
The proof of Theorem~\ref{thm.4}   consists in applying Theorem~\ref{thm.main} to a suitable subdomain $\Omega_{\widetilde {\ft C}}$ of $\Omega$ containing $\widetilde{\ft C}$.
This proof  is the same as the one made to prove Theorem~\ref{thm.2} except that the justification of the existence of the neighborhood  $\ft V_{\widetilde {\ft C}}$  of $\overline{\widetilde{\ft C}}$ in $\overline \Omega$ such that 
\begin{equation}\label{eq.vvc}
\vert\nabla f\vert \neq   0  \text{ on } \overline{\ft V_{\widetilde {\ft C}} }\setminus \overline{ \widetilde{\ft C} }, \ \  \ft V_{\widetilde {\ft C}} \cap \{f<\lambda \}= \widetilde{\ft C} \ \text{ and } \ \ft V_{\widetilde {\ft C}} \cap \{f\le\lambda \}=\overline{ \widetilde{\ft C} }
 \end{equation}
 is slightly different (see more precisely~\eqref{eq.vC-s} and~\eqref{eq.vC-s2} above).  
Actually, the only difference with the proof of  Theorem~\ref{thm.2} is that under the assumptions of  Theorem~\ref{thm.4},  it may  exist $z\in \pa \ft C_1\cup\ldots\cup \pa \ft  C_m$ such that   $\vert \nabla f\vert (z) =0$. Then, by assumption, $z\in \Omega$. Let $i\in \{1,\ldots,m\}$ be such that $z\in \pa \ft C_i$.    If $z$ is not a separating saddle point of $f$, there exists~$\ve_z>0$ such that $B(z,\ve_z)\cap  \{f<\lambda \}= B(z,\ve_z)\cap  \ft C_i $ and $B(z,\ve_z)\cap  \{f\le\lambda \}= B(z,\ve_z)\cap \overline{\ft C_i}$. Else, if $z$ is  a separating saddle point of $f$, 
 from~\eqref{eq.condition1},   there exists $j \neq i$ ($j\in \{1,\ldots,m\}$), $z\in \pa \ft C_i\cap \pa \ft C_j$. Thus, there exists~$\ve_z>0$ such that $B(z,\ve_z)\cap  \{f<\lambda \}= B(z,\ve_z)\cap (\ft C_i\cup \ft C_j)$. Moreover,   $B(z,\ve_z)\cap  \{f\le\lambda \}= B(z,\ve_z)\cap (\overline{\ft C_i}\cup \overline{\ft C_j})$. Together with the analysis of the local structure of $\{f<\lambda\}$ near   the points $z\in \pa \ft C_1\cup\ldots\cup \pa \ft  C_m$ such that $\vert \nabla f \vert (z)\neq 0$ (see items 1 and 2 just   before~\eqref{eq.GAMMAC01}), this justifies the existence of neighborhood  $\ft V_{\widetilde {\ft C}}$  of $\overline{\widetilde{\ft C}}$ in $\overline \Omega$ satisfying~\eqref{eq.vvc}.  The end of the proof then follows exactly the same lines as the proof of Theorem~\ref{thm.2}. 
\end{proof}
\noindent
As an illustration, here  are some simple examples of   results which can be obtained by Theorem~\ref{thm.4}   but not by Theorems~\ref{thm.main} and~\ref{thm.2}. In Figure~\ref{fig:exemple1}, Theorems~\ref{thm.main} and~\ref{thm.2} do  not apply. 
However, in the example depicted in Figure~\ref{fig:exemple1}, item 1 in Theorem~\ref{thm.4} implies that  in the limit $h\to 0$, for all $x\in \ft C_{\ft{max}}=(z_1,z)$, there exists $c>0$ such that $\mathbb P_x[X_{\tau_\Omega} =z_2]=O(e^{-\frac ch})$, and thus $\mathbb P_x[X_{\tau_\Omega} =z_1]= 1+O(e^{-\frac ch})$. 
 \begin{figure}[h!]
\begin{center}
\begin{tikzpicture}[scale=0.8]
\coordinate (b1) at (-0.66,4);
\coordinate (b2) at (1,-0.04);
\coordinate (b21) at (2.8,3.9);
\coordinate (b22) at (3.5,2.2);
\coordinate (b3) at (5,5.8);
\coordinate (b4) at (6.8,2);
 % \draw (b4) ..controls (7.9,-0.9)   and (7.8,4.1)   .. (9,4) ;
\draw [black!100, in=150, out=-10, tension=10.1]
  (b1)[out=-20]   to  (b2) to (b21) to (b22) to (b3) to (b4);
  \draw [dashed,-]   (-1,4) -- (7,4) ;
   \draw [dashed,-]   (-1,-0.1) -- (7,-0.1) ;
 
 %  \draw (-3.4,4) node[]{$ \{f=\min_{\pa \Omega} f\}$};
   \draw (-3.4,-0.1) node[]{$ \{f=\min_{\overline\Omega} f\}$};
     \draw [thick, densely dashed,<->]   (-0.4,4.2) -- (2.33,4.2) ;
   \draw [thick, densely dashed,<->]   (2.76,4.2) -- (4.15,4.2) ;
     \draw (0.7,4.6) node[]{$ \ft C_{\ft{max}}$};
          \draw (3.4,4.6) node[]{$ \ft C_{2}$};
    % \draw [thick, densely dashed,<->]   (6.1,4.2) -- (9,4.2) ;
     % \draw (7.5,4.6) node[]{$ \ft C_{2}$};
 
    \tikzstyle{vertex}=[draw,circle,fill=black,minimum size=4pt,inner sep=0pt]
 \draw (b1)  node[vertex,label=south: {$z_1$}](v){}; 
 \draw (2.6,4) node[vertex,label=north: {$z$}](v){};
 % \draw (9,4)   node[vertex,label=south: {$z_2$}](v){}; 
\draw (1.2,-0.09) node[vertex,label=south: {$x_1$}](v){}; 
\draw (b4) node[vertex,label=south: {$z_2$}](v){}; 
\draw (3.7,2.16) node[vertex,label=south: {$x_2$}](v){}; 
 %\draw  (4.7,5.9)  node[vertex,label=north: {$z$}](v){}; 
    \end{tikzpicture}
\caption{A one dimensional case when \eqref{H-M}, \eqref{eq.hip1}, and \eqref{eq.hip2}    are  satisfied with $\mathcal C=\{\ft C_{\ft{max}}, \ft C_2\}$. The assumption   \eqref{eq.hip3} is not satisfied because $\pa \ft C_{\ft{max}} \cap \pa \Omega=\{z_1\}$ and $f(z_1)>f(z_2)=\min_{\pa \Omega} f$. 
 } 
 \label{fig:exemple1}
 \end{center}
\end{figure}
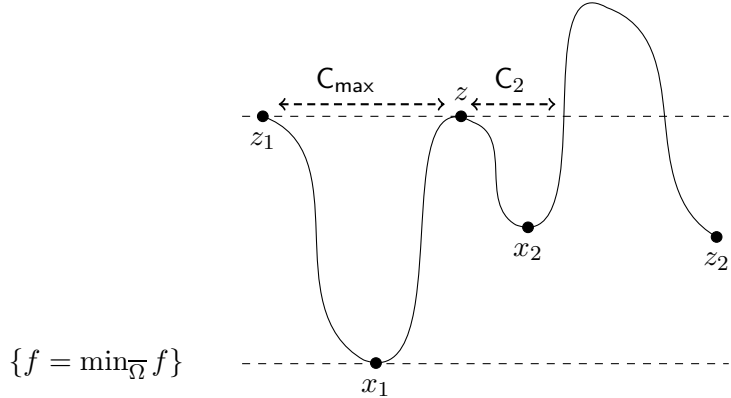
In Figure~\ref{fig:exemple2},  item 1 in Theorem~\ref{thm.main} only  implies that when $h\to 0$ and for all $x\in \ft C_{\ft{max}}=(z_1,z)$,  $\mathbb P_x[X_{\tau_\Omega} =z_2]=O(h^{1/4})$ whereas item 1 in Theorem~\ref{thm.4}  implies that  there exists $c>0$ such that $\mathbb P_x[X_{\tau_\Omega} =z_2]=O(e^{-\frac ch})$.
\begin{figure}[h!]
\begin{center}
\begin{tikzpicture}[scale=0.8]
\coordinate (b1) at (-0.66,4);
\coordinate (b2) at (1,-0.04);
\coordinate (b21) at (2.8,3.9);
\coordinate (b22) at (3.5,2.2);
\coordinate (b3) at (5,5.8);
\coordinate (b4) at (6.8,2);
 \draw (b4) ..controls (7.4,1.8)   .. (8,4) ;
\draw [black!100, in=150, out=-10, tension=10.1]
  (b1)[out=-20]   to  (b2) to (b21) to (b22) to (b3) to (b4);
  \draw [dashed,-]   (-1,4) -- (8.5,4) ;
   \draw [dashed,-]   (-1,-0.1) -- (8.5,-0.1) ;
 %  \draw (-3.4,4) node[]{$ \{f=\min_{\pa \Omega} f\}$};
   \draw (-3.4,-0.1) node[]{$ \{f=\min_{\overline\Omega} f\}$};
   \draw [thick, densely dashed,<->]   (-0.4,4.2) -- (2.33,4.2) ;
   \draw [thick, densely dashed,<->]   (2.76,4.2) -- (4.15,4.2) ;
     \draw (0.7,4.6) node[]{$ \ft C_{\ft{max}}$};
          \draw (3.4,4.6) node[]{$ \ft C_{2}$};
     \draw [thick, densely dashed,<->]   (6.1,4.2) -- (7.8,4.2) ;
      \draw (7,4.6) node[]{$ \ft C_{3}$};
 
    \tikzstyle{vertex}=[draw,circle,fill=black,minimum size=4pt,inner sep=0pt]
 \draw (b1)  node[vertex,label=south: {$z_1$}](v){}; 
 \draw (2.6,4) node[vertex,label=north: {$z$}](v){};
 % \draw (9,4)   node[vertex,label=south: {$z_2$}](v){}; 
\draw (1.2,-0.09) node[vertex,label=south: {$x_1$}](v){}; 
\draw (8,4)  node[vertex,label=north: {$z_2$}](v){}; 
\draw (7.2,1.9)  node[vertex,label=south: {$x_3$}](v){}; 
\draw (3.7,2.16) node[vertex,label=south: {$x_2$}](v){}; 
 %\draw  (4.7,5.9)  node[vertex,label=north: {$z$}](v){}; 
    \end{tikzpicture}
\caption{A one dimensional case when \eqref{H-M},~\eqref{eq.hip1},~\eqref{eq.hip2}, and~\eqref{eq.hip3} are satisfied with $\mathcal C=\{\ft C_{\ft{max}}, \ft C_2, \ft C_3\}$. The assumption~\eqref{eq.hip4} is not satisfied because $\pa \ft C_{\ft{max}}\cap \pa \ft C_2=\{z\}\neq \emptyset$.   
 } 
 \label{fig:exemple2}
 \end{center}
\end{figure}
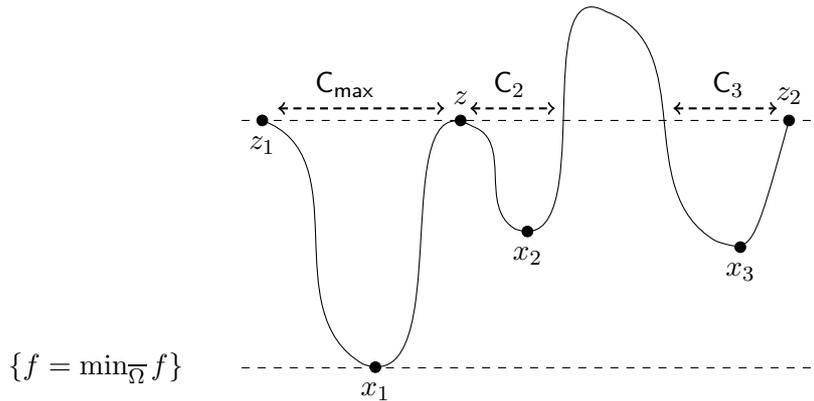

\section{Conclusion and perspectives}\label{sec:conc}

%In a future work, we intend to extend Theorem~\ref{thm.4} to the case when $f$ has critical points on $\pa \Omega$. 

In conclusion, the objective of this work was to identify the exit points from $\Omega$ of $(X_t)_{t \ge 0}$ solution to~\eqref{eq.langevin} in the regime $h \to 0$, when $X_0=x \in \Omega$. 

Under assumption~\eqref{H-M}, what is expected is the following: (i) let us consider the dynamics $\dot{y} = -\nabla f(y)$ with initial condition $y(0)=x$, and let us assume that $y(t)_{t \ge 0}$ remains in $\Omega$ and reaches a local minimum $x^*$ (recall that if $y(t)$ leaves $\Omega$ at a finite time, say at point $y^*$, then $y^*$ will be the exit point in the small temperature regime, see Remark~\ref{re.txfini}); (ii) let us then consider  $\lambda^*= \sup \{ 
\lambda> f(x^*)$ s.t. the connected component of $\{f < \lambda\}$ which contains  $x^*$ does not intersect $\partial \Omega\}$ and the associated connected component $\ft C^*$ defined as  the connected component of $\{f < \lambda^*\}$ which contains $x^*$  (notice that $\ft C^* = \ft C(x^*)$, as defined by~\eqref{eq.Cdef2}). Let us assume that $\partial \ft C^* \cap \partial \Omega \neq \emptyset$. Then, one expects that with probability one, $X_t$ leaves $\Omega$ through $\partial \ft C^* \cap \partial \Omega$ (in other words, the law of the first exit point concentrates on $\partial \ft C^* \cap \partial \Omega$, according to Definition~\ref{de.concentration}).

 What we have proven in this work is that this property can indeed be proven (with explicit relative likelihoods of the different exit points in $\partial \ft C^* \cap \partial \Omega$)  if  
 %$\overline{\ft C^*}$ is a connected component of $\{f \le \lambda^*\}$ ; (ii) 
 $\argmin\{f(x), x \in \ft C^*_{ext}\} =  \argmin\{ f(x), x \in \ft C^*\}$, and $f$ is constant (equal to $\lambda^*$) over the boundary of $\ft C^*_{ext}$ where $\ft C^*_{ext}$ is the connected component of $\{f \le \lambda^*\}$ which contains $x^*$. Indeed, 
 %(i) enters the setting of  Theorem~\ref{thm.2}, and (ii) is exactly what 
 this is exactly what is needed to apply Theorem~\ref{thm.4} with
 $\lambda=\lambda^*$, $\ft C^*=\ft C_1$, and $\ft C^*_{ext} =
 \cup_{j=1}^m \overline{\ft  C_j}$. Notice in particular that the result holds if  $\overline{\ft C^*}$ is a connected component of $\{f \le \lambda^*\}$. 
 %Notice that (ii) generalizes (i) since, in (ii), $\overline{\ft C^*}$ is not necessarily a connected component of $\{f \le \lambda^*\}$.

 In future works, we intend to study the same question while relaxing the assumption~\eqref{H-M}, by allowing $f$ to have critical points on the boundary, see~\cite{pcbord,pcbord2} for  preliminary works in that direction. 
This should also allow us to go beyond
some of the restrictions   above.

\section*{Acknowledgements}
This work is supported by the
European Research Council under the European Union's Seventh Framework
Programme (FP/2007-2013) / ERC Grant Agreement number 614492.
 The authors thank Fran\c cois Laudenbach for enlightening discussions
 on Sard's theorem. The authors also thank the anonymous referee for
 useful suggestions.

\bibliography{biblio_schuss} %You need to replace "rsc" on this
%line with the name of your .bib file
\bibliographystyle{plain}%the RSC's .bst file

\end{document}